\documentclass[11pt]{article}
\usepackage{amssymb,amsmath,amsthm,amsfonts,verbatim,tikz}
\usepackage[all,cmtip]{xy}
\usepackage{hyperref}
\usepackage[margin=1.25in]{geometry}
\usepackage{todonotes}
\usepackage{qtree}

\newtheorem{theorem}{Theorem}[section]
\newtheorem{proposition}[theorem]{Proposition}
\newtheorem{corollary}[theorem]{Corollary}
\newtheorem{lemma}[theorem]{Lemma}
\newtheorem{conjecture}[theorem]{Conjecture}

\newtheorem{convention}[theorem]{Convention}

\theoremstyle{definition}

\newtheorem{definition}[theorem]{Definition}
\newtheorem{example}[theorem]{Example}
\newtheorem{examples}[theorem]{Examples}

\newtheorem{question}[theorem]{Question}

\newtheorem{problem}[theorem]{Problem}
\newtheorem{remark}[theorem]{Remark}

\numberwithin{equation}{section}

\newcommand{\into}{\hookrightarrow}
\newcommand{\onto}{\twoheadrightarrow}

\newcommand{\Ac}{\mathcal{A}}
\newcommand{\tAc}{\widetilde{\Ac}}
\newcommand{\Cc}{\mathcal{C}}

\newcommand{\Mc}{\mathcal{M}}
\newcommand{\Uc}{\mathcal{U}}
\newcommand{\Oc}{\mathcal{O}}
\newcommand{\Pc}{\mathcal{P}}

\newcommand{\Hc}{\mathcal{H}}

\newcommand{\xdownarrow}[1]{%
  {\left\downarrow\vbox to #1{}\right.\kern-\nulldelimiterspace}
}

\newcommand{\Ab}{\mathbb{A}}

\newcommand{\Cb}{\mathbb{C}}
\newcommand{\C}{\mathbb{C}}

\newcommand{\F}{\mathbb{F}}

\newcommand{\Nb}{\mathbb{N}}

\newcommand{\Pb}{\mathbb{P}}

\newcommand{\Z}{\mathbb{Z}}

\newcommand{\et}{{et}}
\newcommand{\dr}{\dashrightarrow}

\newcommand{\Aut}{Aut}

\newcommand{\Spec}{Spec}

\newcommand{\Gr}{Gr}

\DeclareMathOperator{\PSL}{PSL}
\newcommand{\Sp}{Sp}

\newcommand{\tX}{\widetilde{X}}

\newcommand{\tPc}{\widetilde{\Pc}}
\newcommand{\St}{\mathcal{S}}

\DeclareMathOperator{\RD}{RD}
\DeclareMathOperator{\Bl}{Bl}
\DeclareMathOperator{\ed}{ed}

\DeclareMathOperator{\PGL}{PGL}

\DeclareMathOperator{\Gal}{Gal}
\DeclareMathOperator{\Alg}{Alg}

\DeclareMathOperator{\trdeg}{tr.deg}
\DeclareMathOperator{\Perm}{Perm}

\DeclareMathOperator{\Image}{Image}
\DeclareMathOperator{\Cay}{Cay}

\renewcommand{\P}{\operatorname{P}}

\newcommand{\para}[1]{\medskip\noindent\textbf{#1.}}

\title{Resolvent degree, Hilbert's 13th Problem and geometry}

\author{Benson Farb and Jesse Wolfson
\thanks{The first author was supported in part by National Science Foundation Grant Nos. DMS-1105643 and DMS-1406209, and the Jump Trading Mathlab Research Fund. The second author was supported in part by National Science Foundation Grant No. DMS-1400349.
}}

\begin{document}
\maketitle
\begin{abstract}
We develop the theory of resolvent degree, introduced by Brauer \cite{Br} in order to study the complexity of formulas for roots of polynomials and to give a precise formulation of Hilbert's 13th Problem.  We extend the context of this theory to enumerative problems in algebraic geometry, and consider it as an intrinsic invariant of a finite group.  As one application of this point of view, we prove that Hilbert's 13th Problem, and his Sextic and Octic Conjectures, are equivalent to various enumerative geometry problems, for example problems of finding lines on a smooth cubic surface or bitangents on a smooth planar quartic. 
\end{abstract}

\thispagestyle{empty}
\tableofcontents
\thispagestyle{empty}

\clearpage
\pagenumbering{arabic}

\section{Introduction}
In a never-cited 1975 paper \cite{Br}, Brauer introduced for a field extension $L/K$ an integer-valued invariant $\RD(L/K)$  that we call  {\em resolvent degree}.  Applying $\RD$ to function fields gives an invariant $\RD(Y\dashrightarrow X)$ of rational covers \footnote{See Definition~\ref{definition:ratcover} below.}  (e.g. finite branched covers) of complex algebraic varieties.  The resolvent degree $\RD(\widetilde{{\mathcal P}}_n\to {\mathcal P}_n)$ of the root cover of the universal family ${\mathcal P}_n$ of degree $n$ polynomials has the interpretation:
\[
\begin{array}{ll}
\RD(\widetilde{{\mathcal P}}_n\to {\mathcal P}_n)=&
\text{the least $d$ for which there exists a formula in}\\
&\text{algebraic functions of at most $d$ variables for the}\\
&\text{roots of a polynomial in terms of its coefficients.}
\end{array}
\]

While the formal definition seems to have waited until Brauer, the study of ``reduction of parameters'' for polynomials was initiated by Tschirnhaus \cite{Ts} in 1683.  It was developed and refined by Hamilton, Sylvester, Klein, Hilbert, Segre and others.   As we explain below, $\RD$ allows one to go beyond the solvable/unsolvable dichotomy provided by Galois theory; in particular, it was introduced by Brauer to give a precise formulation of Hilbert's 13th Problem (see below).

In this paper we pick up where Brauer left off.  We extend the scope of $\RD$ from polynomials to classical enumerative problems, placing Hilbert's 13th Problem in a broader context and restoring the geometric perspective pioneered by Klein in his study of quintic equations \cite{Kl2}.  One use of resolvent degree is that it gives a uniform framework for stating and relating disparate classical results.  As an example, we prove (Theorem~\ref{theorem:S6:varieties}) an equivalence of Hilbert's Sextic Conjecture to seven other problems, for example relating it to finding lines on cubic surfaces and finding fixed points for hyperelliptic involutions on genus $2$ curves. We prove similar theorems for Hilbert's 13th problem (Theorem~\ref{theorem:S7:varieties}), and Hilbert's Octic Conjecture (Theorem~\ref{theorem:S8:varieties}).

In \cite{W}, this viewpoint is used to extend a beautiful but little-known trick of Hilbert (who used the existence of lines on a smooth cubic surface to give an upper bound on $\RD(\widetilde{{\mathcal P}}_9\to {\mathcal P}_9)$) to improve the upper bounds on $\RD(\widetilde{{\mathcal P}}_n\to {\mathcal P}_n)$ given by Hamilton, Sylvester, B. Segre, Brauer and others.

\subsection{Resolvent degree}

We start with a problem central to classical (and modern) mathematics.

\begin{problem}
\label{problem:19th1}
Find and understand formulas for the roots of a polynomial
\begin{equation}
\label{eq:poly:general}
P(z)=z^n+a_1z^{n-1}+\cdots +a_n
\end{equation}
in terms of the coefficients $a_1,\ldots ,a_n$.
\end{problem}

It is well known that if $n\geq 5$ then no formula exists using only radicals and arithmetic operations in the coefficients $a_i$.\footnote{This was claimed by Ruffini in 1799; a complete proof was given by Abel in 1824.}  Less known is Bring's 1786 theorem \cite{Bri} that any quintic can be reduced via radicals to a quintic of the form $Q(z)=z^5+az+1$ (see \cite{CHM} for a contemporary translation).  In 1836, Hamilton \cite{Ham} extended Bring's results to higher degrees, showing, for example, that any sextic can be reduced via radicals to $Q(z)=z^6+az^2+bz+1$, making it a $2$-parameter ($a$ and $b$) problem.  He also proved that any degree $7$ polynomial can be reduced via radicals to one of the form
\begin{equation}
\label{eq:septic1}
Q(z)=z^7+az^3+bz^2+cz+1,
\end{equation}
and that any degree $8$ polynomial can be reduced via radicals to one of the form $Q(z)=z^8+az^4+bz^3+cz^2+dz+1$. Hilbert conjectured explicitly that one cannot do better: solving a sextic (resp.\ septic, resp.\ octic) is fundamentally a $2$-parameter (resp.\ $3$-parameter, resp.\ $4$-parameter) problem.    Of course we need to know the exact rules of the game here; that is, we need to give a precise definition of what it means to reduce a problem to $r$ parameters.  Surprisingly, a precise definition was only written down in 1975, by Brauer \cite{Br}, and a year later by Arnol'd-Shimura \cite{AS}, apparently unaware of Brauer's paper.  For motivation, let's look at an example.

Let ${\mathcal P}_n\cong\C^n$ be the space of monic, degree $n$ complex polynomials, and let
$\widetilde{{\mathcal P}}_n$ be the {\em root cover} of $\Pc_n$:
\[\widetilde{{\mathcal P}}_n:=\{(P,\lambda): P(\lambda)=0\}\subset {\mathcal P}_n\times\C.\]
The map $(P,\lambda)\mapsto P$ gives an $n$-sheeted branched cover $\widetilde{{\mathcal P}}_n\to{\mathcal P}_n$, with branch locus precisely the subset of ${\mathcal P}_n$ consisting of polynomials with a repeated root, given by the zero-set of the {\em discriminant} $\Delta_n(a_1,\ldots,a_n)$, a polynomial in the coefficients $a_i$.

Recall that a rational map $f\colon X\dashrightarrow Y$ between irreducible varieties is {\em dominant} if the image of $f$ is Zariski dense in $Y$; it is {\em generically finite} if the generic fiber is finite. For such a map there are Zariski opens $U\subseteq X, V\subseteq Y$ so that the restriction $f:U\to V$ is a finite cover.  

\begin{definition}[{\bf Rational cover}]
\label{definition:ratcover}
Let $X$ and $Y$ be irreducible varieties.\footnote{See Convention~\ref{convention:reducible} for the case of reducible varieties.}   A {\em rational cover} $f: X\dashrightarrow Y$ is a generically finite dominant rational map.
\end{definition}

With this definition in hand, ``solving an arbitrary degree $n$ polynomial by radicals'' means precisely that there is a sequence of rational covers \[X_r\dr\cdots\dr X_0={\mathcal P}_n\]
such that $X_r\dr\Pc_n$ factors through a rational cover 
$X_r\dr \widetilde{{\mathcal P}}_n$, and where each $X_{i+1}\dr X_i$ is birationally a pullback
\begin{equation*}
	\xymatrix{
		X_{i+1} \ar@{-->}[r] \ar@{-->}[d] & \Pb^1 \ar[d] & z \ar@{|->}[d] \\
		X_i \ar@{-->}[r]  & \Pb^1 & z^{d_i}
	}
\end{equation*}
The fact that each cover $X_{i+1}\dr X_i$ is a pullback from $\Pb^1$ reflects the fact that it is specified by $\dim_{\C}\Pb^1=1$ parameter, namely taking a $d_i$-th root, and so ``solving by radicals'' is a process involving only $1$ parameter at a time.  The final map $X_r\dr\widetilde{{\mathcal P}}_n$ is crucial. For example, for Cardano's solution in radicals of the cubic, this map has degree $2$, reflecting the fact that Cardano's formula actually produces $6$ solutions (with multiplicity), not just $3$.    While such towers of radicals exist only for $n\leq 4$, Bring's reduction of quintics mentioned above gives for $n=5$ a tower with each $X_{i+1}\dr X_i$ either a radical, or the pullback of the ``Bring curve'' $\Cc\to\Pb^1$ (see \cite{Green} for a beautiful treatment of this genus 4 curve); in particular we see that solving a general quintic is also a $1$-parameter problem.  More precisely, we have the following.

\begin{definition}[{\bf Resolvent degree}]
\label{definition:rd1}
Let $k$ be a field of characterisitc 0 and let $Y\dr X$ be a rational cover of $k$-varieties. The {\em essential dimension} $\ed_k(Y\dr X)$ is the minimal $d$ so that $Y\dr X$ is the ``rational pullback'' of a rational cover of $d$-dimensional varieties: there exists a rational cover $\widetilde{W}\dr W$ with $\dim(W)=d$, a Zariski open $U\subseteq X$, and a morphism $f:U\to W$ such that $f^*\widetilde{W}\cong Y|_U$.

The {\em resolvent degree} $\RD_k(Y\dr X)$ is the minimal $d$ for which there exists a tower of rational covers 
\begin{equation}
\label{rdtower}
    X_r\dr X_{r-1}\dr \cdots\dr X_1\dr X_0=X
\end{equation}
with $\ed_k(X_i\dr X_{i-1})\le d$ for all $i$ and with a dominant map of $X$-schemes $X_r\dr Y$.

\end{definition}
Definition~\ref{definition:rd1} is equivalent to Brauer's original, purely field-theoretic definition; see \S\ref{subsection:definition} below.  One can easily check \footnote{This is somewhat more clear via Brauer's definition.}  that $\RD(\widetilde{{\mathcal P}}_n\to{\mathcal P}_n)$ is the minimal number of parameters to which one can reduce a general degree $n$ polynomial in order to find a formula for the roots.  In this language, the results mentioned above on reduction of parameters can be restated succinctly as:
\[\RD(\widetilde{{\mathcal P}}_n\to{\mathcal P}_n)=1\ \ \forall n\leq 5, \ \ \ \ \text{and}\ \ \ \ \RD(\widetilde{{\mathcal P}}_n\to{\mathcal P}_n)\leq n-4\ \ \forall n> 5.\]

\begin{remark}
The theory of essential dimension has been developed by Buhler--Reichstein, Merkurjev and others into a 
beautiful and widely applicable theory; see Reichstein's 2010 ICM paper \cite{Re} for a survey.   This disallowing of so-called ``accessory irrationalities'' captures more of the arithmetic of the function field of the base, whereas $\RD$ captures more of the intrinsic complexity of the branched cover.  For the problems we are considering, forcing a solution in a single step does not give the correct measure.  For example, there are finite covers $\widetilde{X}\to X$ that are solvable (hence $\RD(\widetilde{X}\to X)=1$) but with $\ed(\widetilde{X}\to X)$ as large as one wants; and for example $\ed(\widetilde{{\mathcal P}}_4\to {\mathcal P}_4)=\ed(\widetilde{{\mathcal P}}_5\to {\mathcal P}_5)=2$, even though (as mentioned above) it was known by 1786 that these problems reduce to $1$ parameter.
\end{remark}

\subsection{Hilbert's problems}
As already noted by Brauer \cite{Br}, Hilbert's conjecture (explicitly asked by Hilbert in \cite[p.424]{Hi1} and \cite[p.247]{Hi2}) that Hamilton's reduction of parameters for the general polynomial of degree $6,7$, or $8$ is optimal,  can now be stated precisely, as can the problem for all degrees.  Both Klein and Hilbert worked on this general problem for decades (see \cite{Kl1,Hi1,Hi2}).

\begin{problem}[{\bf Klein, Hilbert, Brauer}]
\label{hilbert:problems}
Compute $\RD(\widetilde{{\mathcal P}}_n\to {\mathcal P}_n)$. In particular:

\medskip
\noindent
{\em Hilbert's Sextic Conjecture (\cite{Hi2}, p.247): } $\RD(\widetilde{{\mathcal P}}_6\to {\mathcal P}_6)=2$.

\medskip
\noindent
{\em Hilbert's 13th Problem (\cite{Hi1},p.424): } $\RD(\widetilde{{\mathcal P}}_7\to {\mathcal P}_7)=3$.

\medskip
\noindent
{\em Hilbert's Octic Conjecture (\cite{Hi2}, p.247): } $\RD(\tPc_8\to\Pc_8)=4$.
\end{problem}

Amazingly, no progress has been made on any of these three conjectures since Hilbert stated them. In 1957, Arnol'd and Kolmogorov proved (see \cite{Ar}) that there is no local topological obstruction to reducing the number of variables; however, as Arnol'd and many others have noted, the global problem remains open.  A lot of work has been done on finding upper bounds on $\RD(\widetilde{{\mathcal P}}_n\to {\mathcal P}_n)$.  This includes (in other language) theorems of Tschirnhaus (1683), Bring (1786), Hamilton (1836), Sylvester (1887), Klein (1888), Hilbert (1927), and Segre (1945).

The best general upper bound on $\RD(\widetilde{{\mathcal P}}_n\to {\mathcal P}_n)$, prior to the present, was given by Brauer \cite{Br}.  He proved for $n\geq 4$ that $\RD(\widetilde{{\mathcal P}}_n\to {\mathcal P}_n)\leq n-r$ once $n\geq (r-1)!+1$.  Brauer's method was to systematize the classical method of Tschirnhaus transformations.    In \cite{W}, the point of view developed here is used to give a significant improvement on Brauer's bound.  One of the key ideas is to expand the context of resolvent degree.

\subsection{Expanding the context}

Since Hilbert, resolvent degree has been considered primarily for root covers of polynomials.  However, as Klein first realized \cite{Kl1}, $\RD$ is much more widely applicable. After all, many algebraic problems can be reformulated in terms of a rational cover $(P,s)\mapsto P$ from the space $\widetilde{X}$ of pairs $(P,s)$ of input parameters $P$ and solutions $s$ to the space $X$ of parameters $P$, and
\[
\begin{array}{ll}
\RD(\widetilde{X}\dr X)=&\text{minimal number of parameters of any algebraic}\\
&\text{formula for $s$ in the coefficients of $P$.}
\end{array}\]

As Klein himself realized \cite{Kl71}, this general setup includes not only roots of polynomials  $\widetilde{{\mathcal P}}_n\to {\mathcal P}_n$ (see \S\ref{section:brauer}), but also a second fundamental source of examples, namely incidence varieties (see \S\ref{section:enumerative}).

\para{Incidence varieties}
Problems in enumerative geometry are typically set up with the following data:
\begin{enumerate}
\item a pair of moduli spaces $\Mc, {\mathcal C}$ of algebraic varieties;
\item a subvariety $\widetilde{\mathcal M}\subseteq {\mathcal M}\times{\mathcal C}$, called an {\em incidence variety}, consisting of pairs $(M,C)$ satisfying a given incidence relation; and
\item a rational cover $\pi: \widetilde{\mathcal M}\dr {\mathcal M}$ defined by
$\pi(M,C):=M$.
\end{enumerate}

We restrict to characteristic $0$ throughout this paper. By the definition of a rational cover, for each component ${\mathcal M}_0$ of
${\mathcal M}$ there exists $n\geq 1$ so that $\pi$ is an $n$-sheeted covering space over some Zariski open $U\subseteq{\mathcal M}_0$.  In particular for each $M\in U$ there is a set
$\pi^{-1}(M)=\{C_1,\ldots ,C_n\}$ of $n$ varieties in $\mathcal C$, satisfying the given incidence relation, varying in an algebraic way with
$M$.    Here are some examples.

\begin{examples} Let $\Hc_{d,n}$ denote the moduli space of smooth, degree $d$ hypersurfaces in $\Pb^n$.

\begin{enumerate}
\item 27 lines on a smooth cubic surface:
\[\Hc_{3,3}(1):=\{(S,L): \text{$S$ a smooth cubic surface, $L\subset S$ a line}\}\]
and $\pi:\Hc_{3,3}(1)\to \Hc_{3,3}$ is a $27$-sheeted cover.  See \S\ref{section:cubic:surfaces} for precise definitions.
\item 28 bitangents on a smooth planar quartic:
\[\Hc_{4,2}(1):=\{(C,L): \text{$C\subset\Pb^2$ a smooth quartic, $L\subset\Pb^2$ a line tangent to $C$ at $2$ points}\}\]
and $\pi:\Hc_{4,2}(1)\to \Hc_{4,2}$ is a $28$-sheeted cover.  See \S\ref{section:bitangents} for precise definitions.
\item 3264 conics tangent to $5$ given conics: Let $W$ be the linear system of conics in $\Pb^2$ and $W_0\subset W$ the Zariski open consisting of smooth conics. Then we can define
\[Y:=\{(C_1,\ldots ,C_5,C): \ \text{$C$ is tangent to each $C_i$}\}\in W^5\times W_0\]
and $\pi:Y\to W^5$ is a 3264-sheeted dominant map.
\end{enumerate}
\end{examples}

A first goal of enumerative problems is to find such $\widetilde\Mc\dr \Mc$
and then to compute the degree $n$.   One then wants to find points in $\pi^{-1}(M)$ in terms of the data needed to specify $M$.  ``Find'' can have several meanings.

\begin{example}[{\bf Finding a line on a cubic surface}] Cayley-Salmon proved in 1856 that a smooth cubic surface has $27$ lines.  How hard is it to find such a line? all $27$ lines given one of them?
Let $\Hc_{3,3}(r)$ (resp. \ $\Hc_{3,3}^{\rm skew}(r)$) denote the moduli space of $(r+1)$-tuples
$(S; L_1,\ldots ,L_r)$ where $S\in\Hc_{3,3}$ and $\{L_i\}$ are lines (resp.\ disjoint lines) in $S$; see \S\ref{section:cubic:surfaces} for precise definitions.   Harris \cite{Har} proved \footnote{The  first
statement was known to Camille Jordan.} :

\begin{itemize}
\item The monodromy group of the $27$-sheeted cover $\Hc_{3,3}(1)\to \Hc_{3,3}$ is the Weyl group $W(E_6)$; in particular it is not solvable.  Harris \cite[p. 718]{Har} deduces that ``there does not exist a formula
for the 27 lines of a general cubic surface.''

\item The monodromy group of $\Hc_{3,3}(27)\to \Hc_{3,3}^{\rm skew}(r)$ is solvable for $r=3$ but not for $r<3$. Thus there is a formula in radicals for the $27$ lines, given $3$ disjoint ones, but no fewer.
 \end{itemize}
\end{example}

The question remains: how hard is it to find a line on a smooth cubic surface? or $27$ lines given $1$?  We just saw examples where a formula in radicals does not exist, and
indeed this is typical for enumerative problems; this is the main theme of \cite{Har}.  But, in contrast to Harris's conclusion, algebraic formulas not-in-radicals do exist, and indeed have been an object of study since the 17th century.  Resolvent degree allows us to move beyond the solvable/unsolvable dichotomy to give a quantitative measure of the possible complexity of such formulas. In particular it allows us to ask:  what is $\RD(\Hc_{3,3}(r)\to \Hc_{3,3}(s))$?  Here is a simple but illustrative example.

\begin{example}
\label{example:lines1}
$\RD(\Hc_{3,3}(27)\to \Hc_{3,3}(1))\leq \RD(\widetilde{{\mathcal P}}_5\to {\mathcal P}_5)=1$.
\end{example}

Example \ref{example:lines1} follows from a beautiful classical trick: given a line $L$ on a smooth cubic surface $S$, each plane in the pencil containing $L$ intersects $S$ in $L$ union a conic, and this conic degenerates into a union of two lines at the roots of the discriminant $\Delta_L$ of this pencil of conics.  $\Delta_L$ is a one-variable polynomial of degree $5$, which by Bring \cite{Bri} has $\RD=1$.   One then gets $5$ pairs of distinct lines on $S$, and gets the other $16$ via radicals, by Harris's theorem.

\begin{conjecture}[{\bf The line-finding conjecture}]
\label{conjecture:lines2}
\
\[\RD(\Hc_{3,3}(27)\to \Hc_{3,3})=\RD(\Hc_{3,3}(1)\to \Hc_{3,3})=3.\]
\end{conjecture}
The upper bound of 3 comes from work of Klein and Burkhardt \cite{Kl1,Bur}. We give a concise proof in Theorem~\ref{thm:RD:lines} below.

In \S\ref{section:enumerative} we will see how theorems from classical geometry can be used to relate the resolvent degrees of different problems.  For example,  we use the result described in Figure \ref{figure:cubic:to:quartic} to prove the following.

\begin{theorem}
Any minimal algebraic formula for the $27$ lines on a smooth cubic surface (in terms of its coefficients) has the same number of parameters as any minimal algebraic formula for the $28$ bitangents on a plane quartic curve,  given one of them:
\[\RD(\Hc_{3,3}(27)\to\Hc_{3,3})= \RD(\Hc_{4,2}(28)\to \Hc_{4,2}(1)).\]
\end{theorem}

\begin{figure}
\includegraphics[scale=0.65]{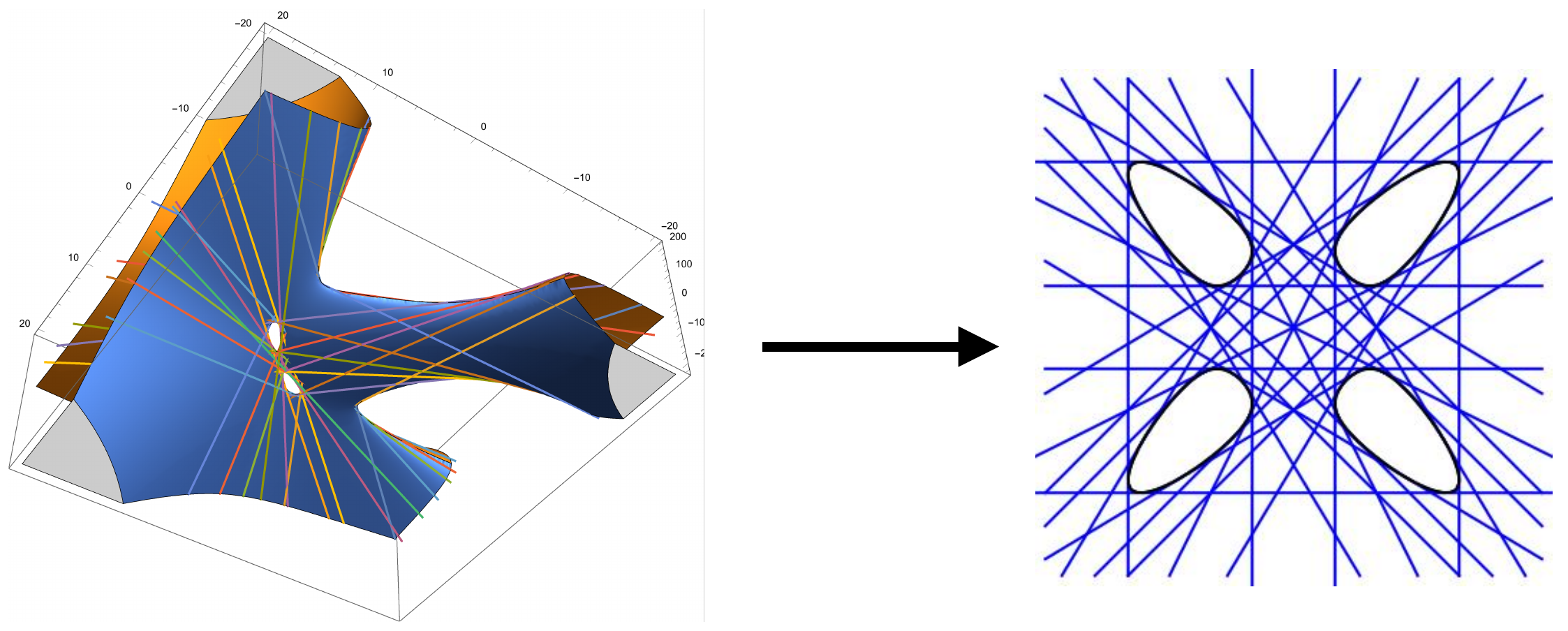}
\caption{\footnotesize The projection $\pi:\Bl_p(S)\rightarrow\Pb^2$
of the blowup at a point $p$ of a smooth cubic surface
$S$  is a $2$-sheeted branched cover, branched over a smooth
plane quartic $C$.  The branching locus in $S$ is the inner rim of each of the four holes in $S$, two of which go off to infinity in the left-hand picture.  The image $\pi(C)$ of each of the $27$ lines in $S$ is a bitangent of $C$.  Here we see (the real points of) a branched cover given by projection to the plane of the paper.  The left part of the figure is taken from \cite{SS}; the right from \cite{PSV}.}
\label{figure:cubic:to:quartic}
\end{figure}

We discuss in depth lines on smooth cubic surfaces and bitangents on smooth plane quartics in
\S\ref{section:cubic:surfaces}  and \S\ref{section:bitangents}, respectively.   We focus on these examples because of their richness and their close relationship to Hilbert's problems (see below).
In \S\ref{section:enumerative} we discuss $\RD$ of some other enumerative problems.  It is our hope that others will work out the resolvent degree story for these problems (and many more).

\begin{remark}[{\bf Explicit formulas}]
Part of the usefulness of the Galois criterion for solvability in radicals is that one can prove it without finding such a formula explicitly.  Similarly, one can give an upper bound for the resolvent degree of a problem without finding an explicit formula.  At the same time, the answers given by non-explicit methods can sometimes help indicate where to look for explicit formulas.
\end{remark}

\subsection{The scope of Hilbert's problems}

As with many of Hilbert's problems, the 13th Problem and the Sextic and Octic Conjectures are meant to
indicate a fundamental phenomenon whose understanding should have implications far beyond the original problem.  Hilbert was clearly interested in, and worked on (see, e.g.,  \cite{Hi1,Hi2}), the general problem of determining $\RD(\tPc_n\to\Pc_n)$, the cases $n=6$, $n=7$, and $n=8$ being the first open cases.    In \S\ref{section:equivalence} we prove the equivalence of the Sextic Conjecture with seven  other statements, the equivalence of Hilbert's 13th Problem with four other statements, and the equivalence of the Octic Conjecture with six other statements.

The point is both to exhibit how rich these problems are, and also to recast them in ways that may be more amenable to solution.   As a sample, here is an abridged version of Theorem~\ref{theorem:S6:varieties} below; for definitions see \S\ref{section:equivalence}.

\begin{theorem}[{\bf The geometry in Hilbert's Sextic Conjecture}]\label{thm:H6intro}
The following statements are equivalent:

\begin{enumerate}
\item Hilbert's Sextic Conjecture is true: $\RD(\tPc_6\to\Pc_6)=2$.
\item $\RD=2$ for the problem of finding the $27$ lines on a cubic, given a ``double six'' set of lines (unordered) (see \S\ref{subsection:h33covers} and Figure~\ref{figure:double-six}):
\[\RD(\Hc_{3,3}(27)\to \Hc_{3,3}(6,6))=2.\]
\end{enumerate}

In fact, the resolvent degrees of the above problems coincide.
\end{theorem}

For further equivalences, as well as for problems about $G$-varieties with $G=W(E_6),S_7, S_8$ or  $W(E_7)$, see
\S\ref{section:equivalence}.

\medskip
Our approach to proving Theorem~\ref{thm:H6intro} (and the versions for other $G$) is to define $\RD$ as an intrinsic invariant of a finite group, in this case $S_6$ and $S_2\times S_6$ respectively.  We do this in \S\ref{section:RD:group}.  We then show that each of the specific covers in the theorem realizes the resolvent degree of their Galois group. Finally, we show that if a group contains as subgroups all the simple factors in its Jordan-holder decomposition, then its resolvent degree is the maximum of these simple factors (Theorem \ref{thm:JH}).  From a classical perspective, a $G$-variety $X$ gives an algebraic function expressing $X$ in terms of coordinates on $X/G$.   The proof of Theorem~\ref{thm:H6intro} proceeds by showing that $\RD(G)=\RD(X\to X/G)$ when $X$ is a ``{\em versal}'' $G$-variety, for an appropriate notion of ``versal'', and then to prove the versality of the varieties listed above.  What ``versality'' means, in this context, is that, up to accessory irrationalities, all $G$-varieties are birationally pullbacks of any versal one.  See \S\ref{subsection:versal} for details.  We give a similar treatment for Hilbert's 13th Problem and $S_7$, Hilbert's Octic Conjecture and $S_8$, as well as for various $W(E_6)$ and $W(E_7)$-varieties. For a more detailed treatment of versality in connection with modular functions, see \cite{FKW}. 

\begin{figure}
\begin{center}
\includegraphics[scale=0.4]{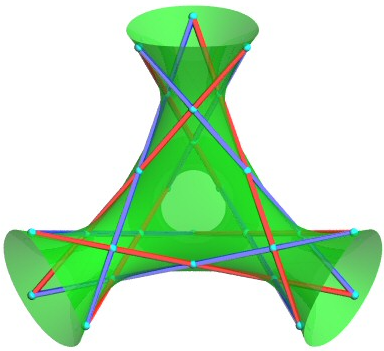}
\end{center}
\caption{\footnotesize A double-six of lines on (the real points of) a smooth cubic surface.  The intersection pattern is given in \eqref{equation:double-six}, with the $a_i$ colored blue and $b_j$ colored red.  One can ask for a formula for the other $15$ lines on a smooth cubic given a double-six.  The resolvent degree of this problem is $2$ if and only if Hilbert's Sextic Conjecture is true. Figure taken
from https://www.mathcurve.com/surfaces.gb/clebsch/doublesix.shtml.}
\label{figure:double-six}
\end{figure}

%
%

\subsection{Lower bounds}
\label{section:lower}
Theorems on resolvent degree to date have exclusively concerned providing upper bounds. As Dixmier concludes in his 1993 paper \cite{Di} (using `$s(n)$' for $\RD(\tPc_n\to\Pc_n)$):
\begin{quote}
{\it Terminons sur une note dramatique, qui prouve notre incroyable ignorance. Bien que cela paraisse improbable, il n'est pas exclu que $s(n) =1$ pour tout $n$!  $\ldots$  Toute minoration de $s(n)$ serait un progr\`{e}s s\'{e}rieux. En particulier, il serait temps de savoir si $s(6)=1$ ou $s(6) = 2$.''}\  \footnote{In English: ``Let's end on a dramatic note, which proves our incredible ignorance. Although this seems unlikely, it is not excluded that $s(n)=1$ for all $n$!  $\ldots$ Any lower bound for $s(n)$ would be serious progress.
In particular, it's time that we know whether $s(6)=1$ or $s(6)=2$.''}
\end{quote}

In fact, we still cannot solve the following problem, implicit in Klein, Hilbert and Brauer, and stated more explicitly by Arnol'd-Shimura \cite{AS}.

\begin{problem}[{\bf Arnol'd-Shimura}]
\label{problem:main5}
Prove that there exists $\widetilde{X}\dr X$ with \[\RD(\widetilde{X}\dr X)>1.\]
\end{problem}

In fact, we believe that the following stronger statement should hold.

\begin{conjecture}
\label{conjecture:main2}
$\RD(\widetilde{{\mathcal P}}_n\to {\mathcal P}_n)\to\infty$ as $n\to \infty$.
\end{conjecture}

Along with Hilbert's Sextic and Octic Conjectures and Hilbert's 13th Problem, these are clearly among the most important conjectures about resolvent degree. While we make no definite progress in this paper toward solving these problems, we hope that with renewed attention to them, and to the broader framework of resolvent degree, future progress may be more forthcoming.

\subsection{Historical Remarks}

The concept of resolvent degree originates with the classical problem of solving polynomials.  It emerged in the 17th century with the work \cite{Ts} of Tschirnhaus.\footnote{See \cite{KK} for a discussion of Tschirnhaus' work and the relevant correspondence with Leibniz.} In 1786, Bring \cite{Bri} proved $\RD=1$ for the problem of solving the quintic, and in 1836 Hamilton \cite{Ham} gave a general sequence of upper bounds on $\RD(\tPc_n\to\Pc_n)$ for increasing $n$.  Hamilton's work was picked up by Sylvester and his student Hammond \cite{Sy,SH1,SH2}, by Klein \cite{Kl1,Kl2}, and by Hilbert \cite{Hi1,Hi2}.  Sixty-four years after Hamilton's work, Hilbert brought to the fore the fundamental issue: no {\em lower} bounds for $\RD(\tPc_n\to\Pc_n)$ had ever been shown.  Hilbert's Sextic Conjecture, Hilbert's 13th Problem \footnote{We will state what is sometimes called the ``algebraic version'' of this problem.  Hilbert's original phrasing of the problem leaves room for various  interpretations.}, and Hilbert's Octic Conjecture pose the challenge of proving that $\RD(\tPc_6\to\Pc_6)=2$, $\RD(\tPc_7\to\Pc_7)=3$, and $\RD(\tPc_8\to\Pc_8)=4$ respectively.

Resolvent degree was first defined explicitly in 1975 by Brauer \cite{Br} in order to make precise Hilbert's 13th Problem.  Brauer also gave new upper bounds on $\RD(\tPc_n\to\Pc_n)$ for all $n$; see
\S\ref{section:brauer} below.  A year later Arnol'd-Shimura \cite{AS}, apparently unaware of
Brauer's paper, gave an equivalent definition of $\RD$, also in order to make precise Hilbert's 13th.
The definition of $\RD$ seems to have lain dormant until the paper \cite{Di} of Dixmier, who helped publicize the concept of resolvent degree.  This concept was also discussed in passing by Buhler-Reichstein \cite{BuRe2} and Chernousov--Gille--Reichstein \cite{CGR}.   The present paper is the first to cite \cite{Br}.  The problem of finding any extension $L/K$ with $\RD(L/K)>1$ remains open.

\subsection{Acknowledgements}
This paper owes a huge intellectual debt to many people. First, the ideas of Klein, Hilbert and their contemporaries are seminal to most work in this area.  The idea of resolvent degree was implicit in their work, but it was first defined formally by Brauer, to whom we owe the modern start of the theory.
We have benefited greatly from the work of Igor Dolgachev, which has (among other things) helped to explain and place in modern terms the work of the classical algebraic geometers.  We mention in particular his book \cite{Do} and his book \cite{DoOr} with Ortland.   We have also benefitted greatly from the point of view of Bruce Hunt's book \cite{Hu}, which emphasizes the connections with locally symmetric varieties.  The first author would like to thank Eduard Looijenga, from whom he learned many of the classical constructions used in this paper, and the modern take on these.  Finally, this paper builds on the theory of essential dimension, introduced by Buhler-Reichstein \cite{BuRe1} and developed by  Merkurjev, Reichstein and many others.

It is a pleasure to thank the following people for useful conversations, correspondence and comments: Maxime Bergeron, Izzet Coskun, Igor Dolgachev, Jordan Ellenberg, Etienne Ghys, Sam Grushevsky, Nate Harman, Eriko Hironaka, Heisuke Hironaka, Sean Howe, Mark Kisin, Eduard Looijenga, Curt McMullen, Madhav Nori, Daniil Rudenko, Aaron Silberstein and Amie Wilkinson.  We especially thank Nat Mayer and Zinovy Reichstein for careful readings and numerous detailed comments on an earlier draft of this paper.  We would also like to thank the algebraic geometry/topology/representation theory working group at the University of Chicago for many conversations. It a pleasure to thank Sebastian Hensel for translating Hilbert's paper \cite{Hi2} for us.  Finally, we thank the anonymous referee, whose comments helped to greatly improve this paper.

\section{The resolvent degree of a rational cover} 

In this section we study the basics of resolvent degree $\RD$.  After giving the definition of $\RD$ of a rational cover, we establish some basic properties of $\RD$, we prove that our definition is equivalent to Brauer's original definition in \cite{Br} of the resolvent degree of a finite field extension, and we prove a number of technical foundational results that are useful for computations.  More specifically, we relate $\RD$ of an extension to that of its Galois closure, and we prove a crucial result on ``accessory irrationalities'', a classical concept studied by Kronecker, Klein and others, that is a key feature of $\RD$.

\subsection{Definitions of resolvent degree}
\label{subsection:definition}

For expositional reasons, we state the results in this paper in the language of $k$-varieties. For the reader who prefers to work with schemes, we will signal when a result or proof does not trivially extend to this case.

\begin{convention}\mbox{}
\begin{enumerate}
\item Unless otherwise specified, throughout this paper we take the base field $k$ to be an arbitrary field of characteristic 0.  
\item By a {\em $k$-variety} we mean a reduced, possibly reducible $k$-scheme of finite type.  
\item When the ground field $k$ is clear we will generally omit the subscript $k$ and simply write $\RD(-)$.
\item A solid arrow $X\to Y$ denotes a regular map of varieties; a dashed arrow $X\dashrightarrow Y$ denotes a rational map of varieties.
\item Given a rational cover $\tX\dr X$, we will refer to a tower \eqref{rdtower} as in Definition~\ref{definition:rd1} 
as a ``tower solving $\tX\dr X$ in $d$ variables'', or as a ``tower solving $\tX$'' for short.
\item We say that $f\colon X\dr Y$ is a ``rational pullback'' of $g\colon W\dr Z$ if there exist dense opens $U'\subset X$, $U\subset Y$, $V'\subset W$, $V\subset Z$ and a pullback square of regular maps
	\begin{equation*}
		\xymatrix{
				U' \ar[r] \ar[d]_{f|_U} & V' \ar[d]^{g|_V} \\
				U \ar[r] & V
		}
	\end{equation*}
\item The ``domain'' of a rational map $f\colon X\dr Y$ is the largest $U\subset X$ for which $f|_U$ is a regular. map. The ``image'' of $f$ is defined to be $f(U)$.
\end{enumerate}
\end{convention}

\begin{convention}[{\bf Rational cover of a reducible variety}]
\label{convention:reducible}
Let $\tX$ and $X$ be a (possibly reducible) varieties.  By a {\em rational cover} $\tX\dashrightarrow X$ we mean 
a rational map $\tX\dashrightarrow X$ with which restricts on each irreducible component $\tX_i\subset\tX$ to a dominant rational map $\tX_i\dr X_j$ for some irreducible component $X_j\subset X$; some Zariski open of each $X_j$ lies in the image of some $\tX_i$; and for each $j$ the generic fiber of $\tX$ over $X_j$ 
is finite. In particular, we want to avoid pathologies such as $X\coprod \{x\}\to X$ (where $\dim(X)>0$ and 
$x\in X(k)$).
\end{convention}

Recall that we defined in Definition~\ref{definition:rd1} the resolvent degree of a rational cover.  We can also define it in terms of field extensions.

\begin{definition}[{\bf Resolvent Degree of a field extension}]
\label{definition:RD2}
    Let $K\into L$ be a finite extension of fields over $k$. The {\em resolvent degree} $\RD_k(L/K)$ is the minimal $d$ for which there exists a finite sequence of finite extensions
    \begin{equation*}
        K=L_0\into L_1\into\cdots \into L_r
    \end{equation*}
    with $L\into L_r$ (as extensions of $K$) and for all $i=1,\ldots,r$,
    \begin{equation*}
        L_i=L_{i-1}\otimes_{F_i}\tilde{F}_i
    \end{equation*}
    where $F_i\into L_{i-1}$ is a subfield with $\trdeg_k(F_i)\le d$ and where $F_i\into \tilde{F}_i$ is a finite extension.  Here $\trdeg_k(F_i)$ denotes the transcendence degree of $F_i$ over $k$.
\end{definition}

The definition of resolvent degree in terms of rational covers and in terms of field extensions are equivalent. 
\begin{proposition}[{\bf Equivalence of definitions}]\label{prop:equivofdefs}
    If $\tX\dr X$ is a rational cover of irreducible $k$-varieties then
    \[\RD(\tX\dr X)=\RD(k(\tX)/k(X)).\]
\end{proposition}
We defer the proof until we have assembled basic properties of $\RD$ (as defined in Definition~\ref{definition:rd1}) in the next section. 

\para{Comparison with essential dimension}  Essential dimension has its origins in work of Hermite \cite{He}, Kronecker \cite{Kr}, Joubert \cite{Jou} and Klein \cite{Kl2}.  The theory was revived, and definitions made explicit, around twenty years ago by Buhler-Reichstein in \cite{BuRe1}.  It has been studied intensively ever since.    See \cite{Re} and \cite{Me} for recent surveys.

A central feature of the theory of essential dimension are the invariants $\ed_k(-;p)$.  These measure the {\em prime-to-$p$ essential dimension}; that is, any auxiliary tower of covers of degree prime to $p$ is allowed before one finds a dominant map to a variety of minimal dimension.  One could define the analogous invariant $\RD_k(-;p)$ by saying that in the tower giving a solution, one allows arbitrary prime-to-$p$ covers, but for covers whose degree is divisible by $p$, only those of $\ed_k(-)\le d$. Field theoretically, this amounts to working over the prime-to-$p$ closure of the function field of the base; that is,  base-changing to $\Spec$ of the fixed field of a $p$-Sylow of the absolute Galois group of $k(X)$. Since $p$-groups and pro-$p$ groups are solvable, we immediately see that $\RD_k(-;p)\equiv 1$ for all $k$ and all $p$. This is in strong contrast to the case of essential dimension, and shows that the study of resolvent degree is a strictly ``Type 2'' problem in the dichotomy of \cite[\S 5]{Re}.

\subsection{Basic properties}

In this subsection we establish some of the basic properties of $\RD$.

\begin{lemma}[{\bf Easy upper bounds}]
\label{lemma:first}
    Let $\tX\dr X$ be a rational cover of $k$-varieties.
    \begin{enumerate}
        \item $\RD(\tX\dr X)\le \ed(\tX\dr X)\le \dim(X)$.
        \item Let $k\into k'$ be any field extension. Then
            \begin{equation*}
                \RD_{k'}(\tX\times_k k'\dr X\times_k k')\le \RD_k(\tX\dr X).
            \end{equation*}
        \item Let $Y\dr X$ be any dominant rational map of $k$-varieties. Then
            \begin{equation*}
                \RD(\tX\times_X Y\dr Y)\le \RD(\tX\dr X).
            \end{equation*}
        \item If the rational map $\tX\dr X$ is birational over $k$ to $\widetilde{Y}\dr Y$; that is, if
        \begin{equation*}
        \xymatrix{
            \widetilde{X}\ar@{-->}[r]^\sim \ar@{-->}[d] & \widetilde{Y} \ar@{-->}[d] \\
            X \ar@{-->}[r]^\sim & Y
        }
    \end{equation*}
        for some birational horizontal maps, then
            \begin{equation*}
                \RD(\tX\dr X)=\RD(\widetilde{Y}\dr Y).
            \end{equation*}
    \end{enumerate}
\end{lemma}
\begin{proof}
    The first statement is immediate from the definitions. The second, third and fourth statements follow from base change: e.g. given a tower solving $\tX\dr X$ over $k$, by base change we obtain an analogous tower over $k'$ solving $\tX\times_k k'\dr X\times_k k'$.  This shows that any upper bound for towers over $k$ immediately gives one over $k'$ as well. The argument for the third and fourth is analogous.
\end{proof}

Many natural branched covers are reducible; indeed such covers arise in Cardano's solution to the cubic; these components are responsible for so-called ``parasitic roots'' in the solution.  The following lemma allows us to reduce the study of $\RD$ to irreducible components.

\begin{lemma}[{\bf Irreducible components}]
\label{lemma:conn}
    Let $\tX\dr X$ be a rational cover.  Let $\{X_i\subset X\}$ be the set of irreducible components of $X$, and let $\{\tX_{i,j}\subset \tX|_{X_i}\}$ be the set of irreducible components of $\tX|_{X_i}\dr X_i$. Then
    \begin{equation*}
        \RD(\tX\dr X)=\max_{i,j}\{\RD(\tX_{i,j}\dr X_i)\}.
    \end{equation*}
\end{lemma}
\begin{proof}
    From the definition of resolvent degree, if $X=\coprod_i X_i$, then
    \begin{equation*}
        \RD(\tX\dr X)=\max_i \{\RD(\tX|_{X_i}\dr X_i)\}.
    \end{equation*}
    Let $X=\bigcup X_i$, and let $X^\sigma=\bigcup_{i\neq j} X_i\cap X_j$ be the set of points contained in more than one irreducible component. Then $X-X^\sigma$ is a disjoint union of irreducible components, and $X-X^\sigma$ is birationally equivalent to $X$. Because resolvent degree is a birational invariant (Lemma \ref{lemma:first}), it suffices to assume that $X$ is irreducible, and that $\tX=\coprod_i \tX_i$.

    The inequality
    \begin{equation*}
        \RD(\tX\dr X)\le \max_i\{\RD(\tX_i\dr X)\}
    \end{equation*}
    is clear.  Indeed, given a tower solving $\tX_i\dr X$ for each $i$, we construct a tower solving $\tX\dr X$ as follows, first if $r$ is the length of the longest tower solving one of the $\tX_i\dr X$, we extend all the other towers (for $j\neq i$) to towers of length $r$ by adding identity maps after the final stage.  Next, we form a tower over $X$ whose $\ell^{th}$ stage is the disjoint union of the $\ell^{th}$ stages of the towers for the $\tX_i$s.  By construction, each stage of this tower is pulled back from something of dimension at most $\max_i\{\RD(\tX_i\dr X)\}$.  It remains to show that
    \begin{equation*}
        \RD(\tX\dr X)\ge \RD(\tX_i\dr X)
    \end{equation*}
    for any $i$. This follows from a standard argument in covering space theory (equivalently the \'etale fundamental group). Without loss of generality, take $i=1$. A simple induction reduces us to the case where $\tX$ is the disjoint union of two irreducible components.  Write $\tX=\tX_1\coprod \tX_2$. Shrinking $X$ as necessary, we can further assume that $X$ and $\tX$ are regular (since, here and throughout this paper, we work in characteristic 0). Suppose now that we have a tower of rational covers 
    \begin{equation*}
        Y_r\dr \cdots \dr Y_0= X
    \end{equation*}
    solving $\tX\dr X$ in functions of at most $d$ variables. Let $U_i\subset Y_i$ be smooth dense opens such that we have a tower of regular \'etale maps
    \begin{equation*}
    	U_r\to \cdots\to U_0\subset X,
    \end{equation*}
    a dominant regular map $p\colon U_r\to \tX$, and for each $i$, a pullback diagram
    \begin{equation*}
    	\xymatrix{
    		U_i \ar[r] \ar[d] & \tilde{Z}_i \ar[d] \\
    		U_{i-1} \ar[r] & Z_i 
		}
    \end{equation*}
    where $\dim Z_i\le d$. Let $U_{r,i}$ be the union of irreducible components mapping dominantly onto $\tX_i$. Let $s$ be the greatest integer for which $U_s$ is irreducible (note that by assumption, $U_0\subset X$ is irreducible). We induct on $r-s$. For the base, $r-s=1$, we have a pullback diagram
    \begin{equation*}
        \xymatrix{
            U_r \ar[r] \ar[d] & \tilde{Z}_r \ar[d]\\
            U_{r-1} \ar[r] & Z_r
        }
    \end{equation*}
    where $U_{r-1}$ is irreducible, and without loss of generality $Z_r$ is too.  If the branched cover $\tilde{Z}_r$ can be partitioned as $\tilde{Z}_{r,i}$ with $U_{r,i}\cong U_{r-1}\times_{Z_r}\tilde{Z}_{r,i}$, then, by replacing $U_r$ with $U_{r,1}$, we obtain a tower solving $\tX_1$ in the same number of variables as the tower solving $\tX$. Suppose therefore that $\tilde{Z}_r$ is connected. Therefore, the connected generically \'etale map $\tilde{Z}_r\to Z_r$ splits when pulled back along $U_{r-1}\to Z_r$. Equivalently, fixing a geometric point $\Omega\to U_{r-1}\to Z_r$, the image
    \begin{equation*}
        \pi_1^{\et}(U_{r-1},\Omega)\to \pi_1^{\et}(Z_r,\Omega)\to \Perm(\tilde{Z}_r|_\Omega)
    \end{equation*}
    lies in a subgroup of the form $\Perm(A_1)\times\Perm(A_2)\subset\Perm(\tilde{Z}_r|_\Omega)$. Let $H\subset\pi_1^{\et}(Z_r,\Omega)$ be the pre-image of $\Perm(A_1)\times\Perm(A_2)$, and let
    \begin{equation*}
        \tilde{Z}_H\to Z_r
    \end{equation*}
    denote the corresponding \'etale map. Because $\pi_1^{\et}(U_{r-1},\Omega)$ factors through the inclusion $H\subset\pi_1^{\et}(Z_r,\Omega)$, the map $U_{r-1}\to Z_r$ factors through $\tilde{Z}_H$.   By construction, the pullback $\tilde{Z}_r\times_{Z_r} \tilde{Z}_H$ splits as
    \begin{equation*}
        (\tilde{Z}_r\times_{Z_r} \tilde{Z}_H)_1\coprod (\tilde{Z}_r\times_{Z_r} \tilde{Z}_H)_2
    \end{equation*}
    with $(\tilde{Z}_r\times_{Z_r} \tilde{Z}_H)_i\times_{Z_H} U_{r-1}\cong U_{r,i}$. Because $\dim(Z_H)=\dim(Z_r)$, we have reduced to the case where the cover $\tilde{Z}_r\to Z_r$ is disconnected, and thus have exhibited a tower solving $\tX_1\to X$ with the same bounds as the tower solving $\tX\to X$. This completes the base of the induction. The inductive step follows from the same construction.  If $r-s>1$, then applying the above construction in sequence, we obtain a tower
    \begin{equation*}
        U_r'\to \cdots U_{s+1}'\to U_s'=U_s\to\cdots \to U_0\subset X
    \end{equation*}
    solving $\tX_1\to X$, which agrees with the tower solving $\tX\to X$ for $i\le s$, and in which $U_i'\to U_{i-1}'$ for $i>s$ is pulled back from a variety of the same dimension which $U_i\to U_{i-1}$ is. We conclude that $\RD(\tX\dr X)\ge \RD(\tX_1\dr X)$.
\end{proof}

\begin{proof}[Proof of Proposition~\ref{prop:equivofdefs}]
	The inequality $\RD(\tX\dr X)\ge \RD(k(\tX)/k(X))$ follows from pulling back any tower solving $\tX\dr X$ along the map $\Spec(k(X))\to X$, and then applying Lemma~\ref{lemma:conn}.
	
	For the reverse inequality, let
	\begin{equation*}
	k(X)=L_0\into L_1\into\cdots \into L_r
	\end{equation*}
	be any tower solving $k(\tX)/k(X)$. For each $i$, pick varieties $Y_i$, $Z_i$ and $\tilde{Z}_i$ such that  $k(Y_i)=L_i$, $k(Z_i)=F_i$ and $k(\tilde{Z}_i)=\tilde{F}_i$ respectively.  Then we obtain a tower of 
	rational covers 
	\begin{equation*}
	Y_r\dr\cdots\dr Y_1\dr Y_0=X
	\end{equation*}
	such that $Y_r\dr X$ factors through a rational cover $Y_r\dr \tX\dr X$, and such that each $Y_i$ sits in a birational pullback diagram
	\begin{equation*}
	\xymatrix{
		Y_i \ar@{-->}[r] \ar@{-->}[d] & \tilde{Z_i} \ar@{-->}[d] \\
		Y_{i-1} \ar@{-->}[r] & Z_i
	}
	\end{equation*}
	Because $\dim(Z_i)=\trdeg(F_i)$, the upper bound on $\RD(k(\tX)/k(X))$ provided by the tower over $k(X)$ carries over to give an identical upper bound on $\RD(\tX\dr X)$. Taking the minimum over all such towers gives
	\[\RD(\tX\dr X)\le \RD(k(\tX)/k(X))\]
	as desired.
\end{proof}

\begin{lemma}[{\bf \boldmath$\RD$ of a composition}]
\label{lemma:max}
    Let $Z\dr Y\dr X$ be a pair of rational covers of $k$-varieties. Then
    \begin{equation*}
        \RD(Z\dr X)=\max\{\RD(Z\dr Y),\RD(Y\dr X)\}.
    \end{equation*}
\end{lemma}
\begin{proof}
    The definition immediately implies that $\RD(Z\dr X)\le\max\{\RD(Z\dr Y),\RD(Y\dr X)\}$ and $\RD(Z\dr X)\ge\RD(Y\dr X)$. To see that $\RD(Z\dr X)\ge \RD(Z\dr Y)$, note that
    \begin{align*}
        \RD(Z\dr X)&\ge \RD(Z\times_X Y\dr Y)\intertext{and, because $Z\dr Y$ embeds as a collection of components of $Z\times_X Y\dr Y$, Lemma \ref{lemma:conn} implies}
        \RD(Z\times_X Y\dr Y)&\ge \RD(Z\dr Y).
    \end{align*}
\end{proof}

\begin{definition}
        A rational cover $\tX\dr X$ is {\em generically n-to-1} if $n=[k(X_i):\Oc(\tX|_{\Spec(k(X_i))})]$ for each irreducible component $X_i\subset X$.
\end{definition}

While the resolvent degree $\RD(\tPc_n\to \Pc_n)$ of the root cover of the space of degree $n$ polynomials is a specific example, it is universal in the following sense.

\smallskip
\begin{lemma}[{\bf Universality of \boldmath$\tPc_n\to \Pc_n$}]
\label{lemma:universal}
    Let $\tX\dr X$ be a generically n-to-1 rational cover. Then
    \begin{equation*}
        \RD(\tX\dr X)\le \RD(\tPc_n\to\Pc_n).
    \end{equation*}
\end{lemma}
\begin{proof}
    By the Theorem of the Primitive Element (using that we are in characteristic 0), there exists $\alpha\in k(\tX)$ such that
    \begin{equation*}
        k(\tX)\cong k(X)(\alpha)\cong k(X)[z]/p_\alpha(z)
    \end{equation*}
    where
    \begin{equation*}
        p_\alpha(z)=z^n+a_1z^{n-1}+\cdots+a_n
    \end{equation*}
    is a minimal polynomial for $\alpha$. Let $U\subset X$ denote the largest Zariski open for which all the coefficients $a_i\in k(X)$ are regular functions. The polynomial $p_\alpha$ determines a map
    \begin{align*}
        U&\to^{p_\alpha} \Pc_n\\
        u&\mapsto (a_1(u),\ldots,a_n(u))
    \end{align*}
    and this map determines a pullback square
    \begin{equation*}
        \xymatrix{
            \tX|_U \ar[r] \ar[d] & \tPc_n \ar[d] \\
            U \ar[r]^{p_\alpha} & \Pc_n
        }
    \end{equation*}
    Therefore, by Lemma \ref{lemma:first},
    \begin{align*}
        \RD(\tX\dr X)=\RD(\tX|_U\to U)\le \RD(\tPc_n\to\Pc_n).
    \end{align*}
\end{proof}

This universal property will show up in many of the examples and computations below.

\subsection{Galois closures and resolvent degree}

In this subsection we will relate the resolvent degree of $L/K$ with the resolvent degree of various related extensions, for example the Galois closure of $L$ over $K$.  This often will allow us in practice to reduce to the case of Galois covers.

\begin{definition}[{\bf Galois theory terminology for rational covers}]
    Let $\tX\dr X$ be a rational cover of $k$-varieties.
    \begin{enumerate}
        \item If $\tX$ is irreducible, then the map $\tX\dr X$ is {\em Galois} if the associated extension of function fields $k(X)\into k(\tX)$ is Galois.  We write $\Gal(\tX\dr X)$ for the Galois group of the associated extension of function fields.
        \item If $\tX$ is irreducible, we say that a map $\tX'\dr X$ is a {\em Galois closure} of $\tX\dr X$ if it factors as $\tX'\dr \tX\dr X$ and if $k(X)\into k(\tX')$ is a Galois closure of $k(X)\into k(\tX)$.
        \item Given $Z\dr Y\dr X$ irreducible, with $Z\dr X$ Galois, the {\em Galois closure} of $Y\dr X$ in $Z\dr X$ is any integral model of the Galois closure of $k(X)\into k(Y)$ in $k(Z)$.
        \item If $\tX$ is reducible, we say $\tX\dr X$ is {\em Galois} if the restriction of the map to each irreducible component of $\tX$ is Galois. Similarly, we say $\tX'\dr X$ is a {\em Galois closure} of $\tX\dr X$ if there is a bijection between the set of irreducible components of $\tX'$ and of $\tX$ such that the restiction of the map $\tX'\dr X$ realizes each component of $\tX'$ as a Galois closure of the corresponding component of $\tX$. Given $Z\dr Y\dr X$ with $Z$ Galois, a {\em Galois closure} of $Y$ in $Z\dr X$ is union of Galois closures of the components of $Y$.
    \end{enumerate}
\end{definition}

The following lemma will allow us to pass to Galois closures when computing $\RD$.  The analogous lemma for $\ed$ is Lemma 2.3 of \cite{BuRe1}.

\begin{lemma}[{\bf \boldmath$\RD$ is preserved under Galois closure}]
\label{lemma:galois}
    Let $\tX\dr X$ be a rational cover of $k$-varieties. Let $\tX'\dr X$ be a Galois closure of $\tX$. Then
    \begin{equation*}
        \RD(\tX\dr X)=\RD(\tX'\dr X).
    \end{equation*}
\end{lemma}
\begin{proof}
    By Lemma \ref{lemma:conn}, it suffices to prove this in the case where $\tX$ is irreducible. For this, we induct on the degree of the map $\tX\dr X$. For the base case, $n=2$, every quadratic extension (in characteristic 0) is already Galois, so the lemma holds trivially.

    For the induction step, assume the lemma holds for all rational covers of $k$-varieties of degree less than $n$.

    Let $\tX\dr X$ be a rational cover of degree $n$.  Consider the composition
    \begin{equation*}
        \tX\times_X\tX \dr \tX\dr X
    \end{equation*}
    The fiber product $\tX\times_X\tX$ splits as $\tX\coprod \tX_1$ (at the level of function fields, this follows from the Primitive Element Theorem), where $\tX\to\tX$ is the identity, and $\tX_1\dr \tX$ is a rational cover of degree $n-1$. By the inductive hypothesis,
    \begin{equation*}
        \RD(\tX_1'\dr\tX)=\RD(\tX_1\dr\tX)
    \end{equation*}
    for any Galois closure $\tX_1'\dr\tX$ of $\tX_1\dr\tX$. By Lemma \ref{lemma:conn},
    \begin{equation*}
        \RD(\tX_1\dr\tX)\le \RD(\tX\dr X).
    \end{equation*}
    Therefore, by Lemma \ref{lemma:max},
    \begin{equation*}
        \RD(\tX_1'\dr X)=\max\{\RD(\tX_1\dr\tX),\RD(\tX\dr X)\}=\RD(\tX\dr X).
    \end{equation*}
    But, by construction, we see that $\tX_1'\dr X$ is a Galois closure of $\tX\dr X$, and this completes the induction step.

\end{proof}

\subsection{Accessory irrationalities}

We now give two results about resolvent degree of field extensions; we defer stating the corresponding results for rational covers of $k$-varieties to below.  We adopt this presentation to make use of constructions such as compositum and intersection of subfields which are easier to state in the setting of field extensions than for covering spaces, where they correspond to greatest lower bounds and least upper bounds in a lattice of covering spaces.

The following allows one to pass to towers of Galois covers when analyzing $\RD$.

\begin{lemma}[{\bf Improving towers}]
\label{lemma:galoisstep}
  Let $K\into L$ be a finite extension of $k$-fields. Then without loss of generality, in any tower realizing $\RD(L/K)$, we can assume that the extension at each stage is Galois.  More precisely, for any $d>0$ (e.g. $d=\RD(L/K)$), let
  \begin{equation*}
        K=K_0\into K_1\into\cdots\into K_r
  \end{equation*}
  be any sequence of extensions with $L\into K_r$ (as fields over $K$) and such that $\ed(K_i/K_{i-1})\le d$ for all $i$. Then there exists a diagram of sequences of extensions
  \begin{equation}\label{galois1}
    \xymatrix{
        K \ar@{=}[d] \ar@{^{(}->}[r] & K_1 \ar@{^{(}->}[r] \ar@{^{(}->}[d] & K_2 \ar@{^{(}->}[r] \ar@{^{(}->}[d] & \cdots \ar@{^{(}->}[r] & K_r \ar@{^{(}->}[d]\\
        K \ar@{=}[d] \ar@{^{(}->}[r] & K_1' \ar@{^{(}->}[r] \ar@{=}[d] & K_2' \ar@{^{(}->}[r] \ar@{^{(}->}[d] & \cdots \ar@{^{(}->}[r] & K_r' \ar@{^{(}->}[d]\\
        K \ar@{^{(}->}[r] & \tilde{K}_1 \ar@{^{(}->}[r] & \tilde{K}_2 \ar@{^{(}->}[r] & \cdots \ar@{^{(}->}[r] & \tilde{K}_r
    }
  \end{equation}
  such that for all $i$,
  \begin{enumerate}
    \item $K_i'$ is Galois over $K_{i-1}'$,
    \item $\tilde{K}_i$ is a Galois closure of $K_i$ over $K$,
    \item $\ed(K_i'/K_{i-1}')\le d$ for all $i$, and
    \item $\RD(\tilde{K}_i/K)=\RD(K_i/K)\le d$ for all $i$.
  \end{enumerate}
\end{lemma}
\begin{proof}
	Because we work in characterstic 0, all extensions are separable.  Therefore, for the bottom row of \eqref{galois1}, define $\tilde{K}_r$ to be a Galois closure of $K_r$ over $K$, and for $i<r$, let $\tilde{K}_i$ denote the Galois closure over $K$ of $K_i$ in $\tilde{K}_r$. Lemma \ref{lemma:galois} implies that
    \[\RD(\tilde{K}_i/K)\le\RD(K_i/K)\le d.\]

    To construct the middle row, we prove by induction that for any $1\le j\le r$ there exists a diagram of sequences of extensions of the form \eqref{galois1} in which $\ed(K_i'/K_{i-1}')\le \ed(K_i/K_{i-1})$ for all $i$, and in which $K_i'$ is Galois over $K_{i-1}'$ for $i\le j$. For the base case $j=1$, let $K_1'=\tilde{K}_1$. This is Galois over $K_0$.  For the induction step, suppose that we have defined $K_j'$ for $j\le i$.  Define $K_{i+1}'$ to be the Galois closure (in $\tilde{K}_r$) of the compositum (in $\tilde{K}_r$) of $K_{i+1}$ with $K_i'$ over $K_i$. Then the definition of essential dimension and \cite[Lemma 2.3]{BuRe1} imply that
    \begin{align*}
        \ed(K_{i+1}'/K_i')\le\ed(K_{i+1}/K_i)
    \end{align*}
    as required to complete the induction step.
\end{proof}

The following proposition is quite useful when analyzing the resolvent degree of $G$-covers (and their subcovers) for $G$ simple. In particular, it shows that a general solution can always be put into a reduced form where the monodromy of the original rational cover occurs precisely at the last stage.

\begin{proposition}[{\bf Accessory irrationalities}]\label{prop:accirr}
    Let $G$ be a finite simple group.  Let $K\into L$ be a Galois extension of $k$-fields with $\Gal(L/K)=G$. Fix $d\ge 0$. Let
    \begin{equation}\label{accstart}
        K=K_0\into K_1\into\cdots\into K_r
    \end{equation}
    be a sequence of extensions such that
    \begin{enumerate}
        \item $\ed(K_i/K_{i-1})\le d$ for all $i$, and
        \item $L\into K_r$ as fields over $K$.
    \end{enumerate}
    Then, there exists $s<r$ and a modified tower
    \begin{equation*}
         K=K_0\into K_1\into\cdots\into K_s\into K_s'
    \end{equation*}
    such that
    \begin{enumerate}
        \item $K_s'$ is a subfield of the Galois closure of $K_{s+1}$ over $K_s$,
        \item $\ed(K_s'/K_s)\le \ed(K_{s+1}/K_s)\le d$,
        \item $L\into K_s'$ as $K$-fields, and under this embedding, $K_s\otimes_K L\to^{\cong} K_s'$.
    \end{enumerate}
\end{proposition}
\begin{proof}
    Define $s$ to be the maximum $i$ such that the absolute Galois group of $K_i$ surjects onto $G$, i.e.
    \begin{equation*}
        s:=\max\{i~|~\Gal(\overline{K}/K_i)\onto G\}.
    \end{equation*}
    Let $\tilde{K}_{s+1}$ denote the Galois closure of $K_{s+1}$ over $K_s$. Then
    \begin{align*}
        \Gal(\overline{K}/\tilde{K}_{s+1})&\unlhd \Gal(\overline{K}/K_s)\intertext{and, by Lemma \ref{lemma:galois}}
        \ed(\tilde{K}_{s+1}/K_s)&=\ed(K_{s+1}/K_s).
    \end{align*}
    Because $\Gal(\overline{K}/K_s)\onto G$ is a surjection, it must take $\Gal(\overline{K}/\tilde{K}_{s+1})$ to a normal subgroup of $G$.  By the definition of $s$, $\Gal(\overline{K}/\tilde{K}_{s+1})\subset \Gal(\overline{K}/K_{s+1})$ does not surject onto $G$.  Therefore, because $G$ is simple, $\Gal(\overline{K}/\tilde{K}_{s+1})$ must be in the kernel of the map to $G$.  This implies that $L$ is contained in $\tilde{K}_{s+1}$, because
    \begin{equation*}
        L=\overline{K}^{\Gal(\overline{K}/L)}=\overline{K}^{\ker(\Gal(\overline{K}/K)\to G)}\subset\overline{K}^{\Gal(\overline{K}/\tilde{K}_{s+1})}=\tilde{K}_{s+1}.
    \end{equation*}
    Therefore, we have $L\into\tilde{K}_{s+1}$ but $L$ is not contained in $K_s$. Define
    \begin{equation*}
        N:=\ker(\Gal(\tilde{K}_{s+1}/K_s)\onto G).
    \end{equation*}
    Define
    \begin{equation*}
        K_{s'}:=(\tilde{K}_{s+1})^N.
    \end{equation*}
    Observe that $\ed(K_{s'}/K_s)\le \ed(K_{s+1}/K_s)=\ed(\tilde{K}_{s+1}/K_s)$, because if $\tilde{K}_{s+1}=K_s\otimes_F \tilde{F}$, then $K_{s'}:=K_s\otimes_F \tilde{F}^N$. Finally, because $\Gal(\overline{K}/K_s)$ surjects onto $G=\Gal(L/K)$, we conclude that
    \begin{equation*}
        K_{s'}=K_s\otimes_K L.
    \end{equation*}
\end{proof}

\begin{corollary}
    Let $G$ be a finite simple group.  Let $L/K$ be any finite extension of $k$-fields for which the Galois closure has Galois group $G$. Then $\RD(L/K)$ equals the minimal $d$ for which there exists a tower
    \begin{equation*}
        K=K_0\into K_1\into\cdots \into K_{r-1} \into K_r
    \end{equation*}
    for which
    \begin{enumerate}
        \item $\ed(K_i/K_{i-1})\le d$, and
        \item $K_r\cong K_{r-1}\otimes_K L$.
    \end{enumerate}
\end{corollary}
\begin{proof}
    For any tower solving the Galois closure $\tilde{L}$ of $L$ over $K$, we can apply Proposition~\ref{prop:accirr}.  Let $H\subset G$ be the subgroup such that $L=\tilde{L}^H$. Applying Proposition \ref{prop:accirr} and Lemma \ref{lemma:galois}, $\RD(L/K)$ is the minimal $d$ for which there exists a tower
    \begin{equation*}
        K=K_0\into K_1\into\cdots \into K_{r-1} \into K_r
    \end{equation*}
    for which
    \begin{enumerate}
        \item $\ed(K_i/K_{i-1})\le d$, and
        \item $K_r\cong K_{r-1}\otimes_K \tilde{L}$.
    \end{enumerate}
    Replacing $K_r$ by $K_r^H\cong K_{r-1}\otimes_K L$, we obtain a tower of the desired form.
\end{proof}

\begin{remark}
An {\em accessory irrationality} to a rational cover $\tilde{X}\dr X$ is any rational cover 
$E\dr X$ which does not factor through $\tilde{X}$.  If $\RD(L/K)\neq\ed(L/K)$, then accessory irrationalities are intrinsic features of any solution of $L/K$ in $d<\ed(L/K)$ variables.  The notion of accessory irrationality first appeared in work of Kronecker and received intensive study in Klein's lectures on the icosahedron \cite{Kl2} (see also the appendix to \cite{DoMc}). In particular, Klein proved that
    \[\ed(\tPc_5\to\Pc_5)=2\neq \RD(\tPc_5\to\Pc_5)=1\] and thus that accessory irrationalities are an inescapable feature of solutions of the quintic in one variable.
\end{remark}

\begin{question}
    Let $K\into L$ be a finite extension of $k$-fields.  Among towers solving $L/K$ in the minimal number of variables, can we always find one in which the stages of the tower have monotone increasing essential dimension?
\end{question}

The geometric statement of Lemma \ref{lemma:galoisstep} is the following.
\begin{corollary}[{\bf Improving towers, geometric version}]
\label{cor:galoisstep}
  Let $\tX\dr X$ be a rational cover. Then without loss of generality, in any tower solving $\tX\dr X$ in $d$ variables, we can assume that the map at each stage is Galois.  More precisely, for any $d>0$ (e.g. $d=\RD(\tX\dr X)$), let
  \begin{equation*}
        Y_r\dr \cdots \dr Y_1\dr Y_0=X
  \end{equation*}
  be a tower of rational covers with $Y_r\dr X$ factoring through $\tX$ and such that for all $i$, $Y_i\dr Y_{i-1}$ is pulled back from a rational cover of varieties of dimension at most $d$. Then there exists a diagram of sequences of rational covers
  \begin{equation}\label{galois}
    \xymatrix{
        \tilde{Y}_r \ar@{-->}[r] \ar@{-->}[d] & \cdots \ar@{-->}[r] & \tilde{Y}_2 \ar@{-->}[r] \ar[d] & \tilde{Y_1} \ar@{-->}[r] \ar@{=}[d] & X \ar@{=}[d] \\
        Y_r' \ar@{-->}[r] \ar@{-->}[d] & \cdots \ar@{-->}[r] & Y_2' \ar@{-->}[d]\ar@{-->}[r] & Y_1'\ar@{-->}[r] \ar@{-->}[d] & X \ar@{=}[d]\\
        Y_r \ar@{-->}[r] & \cdots \ar@{-->}[r] & Y_2 \ar@{-->}[r] & Y_1\ar@{-->}[r] & X
    }
  \end{equation}
  such that for all $i$,
  \begin{enumerate}
    \item $Y_i'\dr Y_{i-1}'$ is Galois,
    \item $\tilde{Y}_i\dr X$ is a Galois closure of $Y_i\dr X$,
    \item $\ed(Y_i'\dr Y_{i-1}')\le d$, and
    \item $\RD(\tilde{Y}_i\dr X)=\RD(Y_i\dr X)\le d$.
  \end{enumerate}
\end{corollary}

The geometric statement of Proposition \ref{prop:accirr} is the following.
\begin{corollary}[{\bf Geometric accessory irrationalities}]
\label{cor:accirr}
    Let $G$ be a finite simple group.  Let $\tX\dr X$ be a 
    rational cover for which the Galois closure has Galois group $G$. Then $\RD(\tX\dr X)$ equals the minimal $d$ for which there exists a tower
    \begin{equation*}
        Y_r\dr\cdots\dr Y_1\dr Y_0=X
    \end{equation*}
    for which
    \begin{enumerate}
        \item $Y_r\cong Y_{r-1}\times_X \tX$, and
        \item for each $i$, $Y_{i+1}\dr Y_i$ is pulled back from a map of varieties of dimension at most $d$, i.e. there is a rational pullback square with $\dim_k(Z_i)\le d$
            \begin{equation*}
                \xymatrix{
                    Y_{i+1}\ar@{-->}[r] \ar@{-->}[d] & \widetilde{Z}_i \ar@{-->}[d]\\
                    Y_i \ar@{-->}[r] & Z_i
                }
            \end{equation*}
    \end{enumerate}
\end{corollary}

\section{The resolvent degree of a finite group}
\label{section:RD:group}

In this section we define the resolvent degree $\RD(G)$ of a finite group $G$.  This intrinsic invariant of $G$ gives a uniform upper bound on the complexity of all $G$-covers of all varieties.  Just as with the theory of essential dimension from which it was inspired, $\RD(G)$ will be quite useful.

\subsection{Definition and basic properties}
Throughout this section we fix a ground field $k$ of characteristic 0.  We will consider finite groups $G$ with $G$-actions by automorphisms on varieties $X$, so that $X/G$ is a variety.  We say that a $G$-variety $X$ is {\em primitive} if $G$ acts transitively on the set of irreducible components of $X$.  We say that $X$ is {\em faithful} if the representation $G\to\Aut(X)$ is injective.

\begin{definition}[{\bf Resolvent degree of a finite group}]
\label{definition:RD:group}
    Let $G$ be a finite group. The {\em resolvent degree} $\RD(G)$ of $G$ is defined to be
    \[\RD(G):=\sup\{\RD(X\to X/G): \text{$X$ is a primitive, faithful $G$-variety over $k$}\}.  \]
\end{definition}

While $\RD(G)$ gives a universal upper bound on any $\RD(X\to X/G)$, it does not in general provide any lower bound on any particular $G$-cover; see below. On the other hand we will prove that
$\RD(G)=\RD(V\to V/G)$ for any faithful linear $G$-variety $V$, and more generally for any ``versal'' $G$-variety.  Replacing $\RD$ by $\ed$ in Definition~\ref{definition:RD:group} gives the definition
of Buhler-Reichstein \cite{BuRe1} for the essential dimension of a finite group.  Indeed, the two invariants of $G$-varieties compare as follows.

\begin{lemma}
    Let $G$ be any finite group. Then
    \[\RD(G)\le \ed(G)<\infty.\]
\end{lemma}

\begin{proof}
For any rational cover $X\dr Y$ we have by definition $\RD(X\dr Y)\le \ed(X\dr Y)$. In particular, if $X$ is any faithful $G$-variety then
    \begin{align*}
        \RD(X\to X/G)&\le \ed(X\to X/G)\\
        &\le \ed(\Ab^G\to\Ab^G/G)\tag{by Theorem 3.1 of \cite{BuRe1}}\\
        &=\ed(G)<\infty
    \end{align*}
    where $\Ab^G$ denotes the regular representation of $G$ viewed as a faithful linear $G$-variety.
\end{proof}

\begin{theorem}\label{thm:JH}
   Let $G$ be a finite group, and let $\{G_i\}_{i=1}^n$ denote the set of simple factors in its Jordan--H\"older decomposition. Then
   \begin{equation*}
        \RD(G)\le\max_{1\le i\le n}\{\RD(G_i)\}.
   \end{equation*}
   Moreover, if $G_i\hookrightarrow G$ for all $i$, then
   \begin{equation*}
        \RD(G)=\max_{1\le i\le n}\{\RD(G_i)\}.
   \end{equation*}
\end{theorem}

The analogue of Theorem~\ref{thm:JH} for essential dimension is false, even in simple examples: take $G_1=G_2=\Z/2\Z$ and $G=G_1\times G_2$.  Note too that $\ed(G/H)$ can be much larger than $\ed(G)$ for normal subgroups $H\lhd G$; see Theorem 1.5 of \cite{MR}.  We do not know if the hypothesis in Theorem~\ref{thm:JH} that $G_i\subseteq G$ for all $i$ is necessary.

\begin{proof}
    If $G_i\hookrightarrow G$ for all $i$, then by Lemma \ref{lemma:subgroup} below, $\RD(G)\ge\max_i\{\RD(G_i)\}$. To show the opposite inequality in general, we induct on the number of simple factors (with multiplicity). For the base of the induction $n=1$, there is nothing to show.  Assume therefore that we have shown it for $n-1$. Let
    \begin{equation*}
        0\unlhd H_1\unlhd \cdots\unlhd H_n=G
    \end{equation*}
    be a composition series for $G$ with $H_i/H_{i-1}=G_i$. Let $X$ be a primitive faithful $G$-variety.

    The map $X\to X/G$ factors as
    \begin{equation*}
        X\to X/H_{n-1}\to X/G.
    \end{equation*}
    If $X$ is not primitive as an $H_{n-1}$-variety, then the set of $H_{n-1}$-orbits on the set of irreducible components of $X$ partitions $X$ into a union of primitive $H_{n-1}$-varieties. Moreover, because the $G$-action is primitive and $H_{n-1}\unlhd G$, the union of the $H_{n-1}$-quotients is a primitive $G_n=G/H_{n-1}$-variety. Lemma \ref{lemma:conn} implies that
    \begin{equation*}
       \RD(X\to X/H_{n-1})=\max_j\{\RD(X_j\to X_j/H_{n-1})\}
    \end{equation*}
    where the maximum is taken over the set of primitive $H_{n-1}$-varieties in the above partition of $X$. In particular,
    \begin{equation*}
        \RD(H_{n-1})\ge \RD(X\to X/H_{n-1}).
    \end{equation*}
    Therefore
    \begin{align*}
        \max\{\RD(G_n),\RD(H_{n-1})\}&\ge \max\{\RD(X\to X/H_{n-1}),\RD(X/H_{n-1}\to X/G)\}\\
        &=\RD(X\to X/G)\tag{by Lemma \ref{lemma:max}.}
    \end{align*}
    Passing to the supremum and invoking the induction hypothesis, we obtain the desired inequality
    \begin{equation*}
        \max_{1\le i\le n}\{\RD(G_i)\}\ge \RD(G).
    \end{equation*}
\end{proof}

As a simple application of Theorem~\ref{thm:JH}, we have the following.

\begin{corollary}\label{cor:asolv}
    Let $G$ be an ``almost solvable'' group, i.e. a group whose simple factors are cyclic or $A_5$. Then $\RD(G)=1$.
\end{corollary}
\begin{proof}
    By Theorem \ref{thm:JH},
    \begin{equation*}
        \RD(G)\le\max\{\{\RD(\Z/n\Z)\}_{n\in\Nb},\RD(A_5)\}.
    \end{equation*}
    Because $G$ is nontrivial, there exists a faithful, geometrically connected $G$-variety $X$ of dimension $\ge 1$. Because $X$ is geometrically connected, there is no faithful $G$-equivariant rational map $X\dr Z$ for $Z$ any faithful 0-dimensional $G$-variety. We conclude that $\RD(G)\ge 1$.

    By Bring's bound and item \ref{corollary:Hilbert:for:faithful} of Corollary \ref{cor:Hfaithandn1} \ref{corollary:Hilbert:for:faithful} below,
    \begin{equation*}
        1=\RD(\tPc_5\to\Pc_5)=\RD(S_5)=\RD(A_5)
    \end{equation*}
    where the last equality follows from Theorem \ref{thm:JH}. The result now follows from the equality
    \begin{equation*}
        \RD(\Z/n\Z)=1 \ \ \text{for all $n\geq 2$}
    \end{equation*}
    which follows from the classical fact that any characteristic 0 field extension with solvable Galois group is solvable in radicals.
\end{proof}

Corollary \ref{cor:asolv} follows from the primary cases of simple groups where $\RD$ is currently known
exactly (i.e. cyclic groups and $A_5$).\footnote{Klein also proved that $\RD(\PSL_2(\F_7)=1)$. See \cite[Proposition 4.2.4]{FKW}  for a contemporary treatment.} In general, we have at best upper bounds, e.g. $\RD(A_6)\leq 2$ and $\RD(A_7)\leq 3$.  Theorem \ref{thm:JH} indicates the importance of computing the resolvent degree of finite groups.

\begin{problem}[{\bf \boldmath$\RD(G)$ for \boldmath$G$ finite simple}]
    Compute the resolvent degree of all finite simple groups $G$.
\end{problem}

\subsection{Versal $G$-varieties}
\label{subsection:versal}

It is useful to have a model (not always unique) $G$-variety to which all other $G$-varieties can be compared.  Such varieties, called ``versal $G$-varieties'', play a crucial role in the theory of essential dimension.  After recalling the definition (cf. \cite{DuRe1}) and some variations that arise naturally when studying resolvent degree, we give some examples.

\begin{definition}[{\bf Versal \boldmath$G$-variety}]
    A faithful $G$-variety $X$ is {\em versal} if for every $G$-invariant Zariski open $U\subset X$ and every faithful $G$-variety $Y$, there exists a $G$-equivariant rational map $Y\dr U$.
\end{definition}

Our interest in versality comes from the following.
\begin{proposition}\label{prop:whyversal}
    Let $X$ be a versal $G$-variety. Then
    \begin{enumerate}
        \item $\ed(X\to X/G)=\ed(G)$.
        \item $\RD(X\to X/G)=\RD(G)$.
    \end{enumerate}
\end{proposition}

\begin{proof}
    The proof for essential dimension is standard; we recall it here as we will use it.  Let $X$ be a versal $G$-variety. Recall that $\ed(G)=\sup\{\ed(Y\to Y/G)\}$ where the supremum is over all faithful $G$-varieties $Y$. Let $U\subset X$ be a dense $G$-invariant Zariski open which admits a $G$-equivariant dominant map $U\to Z$ to a faithful $G$-variety $Z$ with $\dim(Z)=\ed(X\to X/G)$. By the definition of versality, there exists a $G$-equivariant rational map $Y\dr U$. Composing with $U\to Z$, we obtain a $G$-equivariant rational dominant map $Y\dr Z$, which implies
    \[\ed(Y\to Y/G)\le \dim(Z)=\ed(X\to X/G).\] Therefore $\ed(X\to X/G)=\ed(G)$.

    We now prove the statement for resolvent degree. By definition, $\RD(X\to X/G)\le \RD(G)$. It remains to prove that $\RD(X\to X/G)\ge \RD(Y\to Y/G)$ for any faithful $G$-variety $Y$. Let
    \begin{equation*}
        \xymatrix{
            &&& X \ar[d] \\
            X_r \ar@{-->}[urrr] \ar@{-->}[r] & \cdots \ar@{-->}[r] & X_1 \ar@{-->}[r] & X/G
        }
    \end{equation*}
    be a solution of $X\to X/G$.  Let $\bar{U}\subset \Image(X_r\dr X/G)$ be a Zariski open, and let $U\subset X$ be its preimage under the map $X\to X/G$. By the definition of versality, there exists a $G$-equivariant map
    \begin{equation*}
        V\to U
    \end{equation*}
    for some dense Zariski open $V\subset Y$. Since both $G$-varieties are faithful, this determines a pullback diagram
    \begin{equation*}
        \xymatrix{
            V \ar[r] \ar[d] & U \ar[d] \\
            V/G \ar[r] & U/G
        }
    \end{equation*}
    and we can pull back the above solution of $X\to X/G$ to $V\to V/G$.  Since every solution in $d$-variables of $X\to X/G$ gives rise to a solution in $d$-variables of $V\to V/G$, and since $V\to V/G$ is birational to $Y\to Y/G$, we conclude, from the definition, that $\RD(X\to X/G)\ge\RD(Y\to Y/G))$.
\end{proof}

The notion of versal is stronger than we strictly need for resolvent degree.
\begin{definition}[{\bf Solvably-versal, RD-versal}]
    Let $G$ be a finite group. A faithful $G$-variety $X$ is:
    \begin{enumerate}
        \item {\em solvably-versal} if, for every $G$-invariant Zariski open $U\subset X$ and any faithful $G$-variety $Y$, there exists a rational cover
                \begin{equation*}
                    \widetilde{Y}\dr Y/G
                \end{equation*}
                with $k(Y/G)\into k(\widetilde{Y})$ a solvable extension, and a $G$-equivariant rational map
                \begin{equation*}
                    \widetilde{Y}\times_{Y/G} Y\dr U;
                \end{equation*}
        \item {\em $\RD$-versal} if, for every $G$-invariant Zariski open $U\subset X$ and any faithful $G$-variety $Y$, there exists a rational cover
                \begin{equation*}
                    \widetilde{Y}\dr Y/G
                \end{equation*}
                with $\RD(\widetilde{Y}\dr Y/G)\le \RD(X\to X/G)$ and a $G$-equivariant rational map \begin{equation*}
                    \widetilde{Y}\times_{Y/G} Y\dr U.
                \end{equation*}
    \end{enumerate}
\end{definition}

Note that solvably-versal implies $\RD$-versal; we do not know if the converse is true or not.

\begin{example}[{\bf Klein}]
    Klein \cite{Kl2} proved a ``Normalformsatz'' for the group $A_5$, showing that perhaps after passing to an intermediate degree 2 cover, every $A_5$-cover is pulled back from the canonical $A_5$-cover of $\Pb^1\to\Pb^1/A_5\cong\Pb^1$. In our language, this shows that $\Pb^1$ with its standard $A_5$ action is solvably versal.
\end{example}

$\RD$-versal $G$-varieties realize the resolvent degree of $G$.

\begin{proposition}\label{prop:whyRDversal}
    Let $G$ be a finite group, and let $X$ be an $\RD$-versal $G$-variety. Then
    \begin{equation*}
        \RD(X\to X/G)=\RD(G).
    \end{equation*}
\end{proposition}
\begin{proof}
    The proof is similar to that of Proposition \ref{prop:whyversal}. It suffices to show that $\RD(X\to X/G)\ge \RD(Y\to Y/G)$ for any faithful $G$-variety $Y$. Let
    \begin{equation*}
        \xymatrix{
            &&& X \ar[d] \\
            X_r \ar@{-->}[urrr] \ar@{-->}[r] & \cdots \ar@{-->}[r] & X_1 \ar@{-->}[r] & X/G
        }
    \end{equation*}
    be a solution of $X\to X/G$.  Let $\bar{U}\subset \Image(X_r\dr X/G)$ be a Zariski open, and let $U\subset X$ be its preimage under the map $X\to X/G$. By the definition of $\RD$-versality, there exists a 
    rational cover
    \begin{equation*}
        \widetilde{Y}\dr Y/G
    \end{equation*}
    with $\RD(\widetilde{Y}\dr Y/G)\le \RD(X\to X/G)$ and with a $G$-equivariant map
    \begin{equation*}
        V\to U
    \end{equation*}
    for some dense Zariski open $V\subset \widetilde{Y}\times_{Y/G}Y$. Since both $G$-varieties are faithful, this determines a pullback diagram
    \begin{equation*}
        \xymatrix{
            V \ar[r] \ar[d] & U \ar[d] \\
            V/G \ar[r] & U/G
        }
    \end{equation*}
    and we can pullback the above solution of $X\to X/G$ to $V\to V/G$.  Since every solution in $d$-variables of $X\to X/G$ gives rise to a solution in $d$-variables of $V\to V/G$, and since $V\to V/G$ is birational to $\widetilde{Y}\times_{Y/G}Y\dr \widetilde{Y}$, we conclude, from the definition, that $\RD(X\to X/G)\ge\RD(\widetilde{Y}\times_{Y/G}Y\dr \widetilde{Y})$. By Lemma \ref{lemma:max},
    \begin{align*}
        \RD(\widetilde{Y}\times_{Y/G}Y\dr Y/G)&=\max\{\RD(\widetilde{Y}\times_{Y/G}Y\dr \widetilde{Y}),\RD(\widetilde{Y}\dr Y/G)\}\\
        &\le \RD(X\to X/G).
    \end{align*}
\end{proof}

\subsection{Criteria for versality}
\label{subsection:versality:criteria}

In this subsection we give some basic properties of versality, as well as criteria for detecting it.
To start, a {\em $G$-compression}  (i.e. $G$-equivariant dominant rational map) of a versal $G$-variety is versal.

\begin{lemma}[{\bf Compressions of versal are versal}]
\label{lemma:dom}
    Let $X$ be a faithful $G$-variety, and let $Y$ be a versal $G$-variety. If there exists a $G$-equivariant dominant rational map $f\colon Y\dr X$, then $X$ is versal.
\end{lemma}
\begin{proof}
    Let $U\subset X$ be a $G$-invariant Zariski open, and let $Z$ be any faithful $G$-variety. Then $f^{-1}(U)\subset Y$ is a $G$-invariant Zariski open, and by the definition of versality, there exists a $G$-equivariant rational map $Z\dr f^{-1}(U)$. Composing with $f$, we obtain a $G$-equivariant rational map $Z\dr U$ as desired.
\end{proof}

Versal $G$-varieties are also versal for subgroups.

\begin{lemma}[{\bf Versality descends}]
\label{lemma:subversal}
    Let $G$ be a finite group. If $X$ is a versal $G$-variety, then $X$ is also a versal $H$-variety for any subgroup $H\subseteq G$.
\end{lemma}
\begin{proof}
    By the definition of versal, we must show that for every $H$-invariant Zariski open $U\subset X$ and every faithful $H$-variety $Y$, there exists an $H$-equivariant rational map $Y\dr U$. Given $U$, let $U'\subseteq U$ be the maximal $G$-invariant Zariski open contained in $U$ (i.e. $U'=\bigcap_{g\in G} g\cdot U$).  Consider the $G$-variety
    \begin{equation*}
        G\times_H Y:=G\times Y/\sim
    \end{equation*}
    where $\sim$ is the equivalence relation given by $(g,hy)\sim (gh,y)$, and the $G$-action given by
    \begin{equation*}
        g'\cdot[(g,y)]:=[(g'g,y)].
    \end{equation*}
    It is straightforward to check that because $Y$ is a faithful $H$-variety, the variety $G\times_H Y$ is a faithful $G$-variety. Because $X$ is versal, there exists a $G$-equivariant rational map
    \begin{equation}\label{subversal}
        G\times_H Y\dr U'
    \end{equation}
    One can check explicitly that the map
    \begin{align*}
        Y&\to G\times_H Y\\
        y&\mapsto [(e,y)]
    \end{align*}
    is $H$-equivariant. Composing this with \eqref{subversal}, we obtain an $H$-equivariant rational map
    \begin{equation*}
        Y\dr U'\subset U
    \end{equation*}
    as required.
\end{proof}

Lemma~\ref{lemma:subversal} has the following consequence.

\begin{lemma}\label{lemma:subgroup}
    Let $H\subset G$ be a subgroup. Then $\RD(H)\le\RD(G)$.
    \end{lemma}

\begin{proof}
    Let $X$ be a versal $G$-variety. Then $X$ is a versal $H$-variety by Lemma \ref{lemma:subversal}. By Proposition \ref{prop:whyversal} and Lemma \ref{lemma:max},
    \begin{align*}
        \RD(G)&=\RD(X\to X/G)\\
        &=\max\{\RD(X\to X/H),\RD(X/H\to X/G)\}\\
        &=\max\{\RD(H),\RD(X/H\to X/G)\}\\
        &\ge \RD(H).
    \end{align*}
\end{proof}

There exist criteria to check whether a given $G$-variety is versal.

\begin{lemma}[{\bf Versality criterion}]
\label{lemma:showversal}
    Let $X$ be a faithful $G$-variety. Suppose both of the following statements hold.
    \begin{enumerate}
        \item \label{Gmoving} For every faithful, closed $G$-invariant subvariety $Z_1\subset X$, and any closed (not necessarily faithful) $G$-invariant subvariety $Z_2\subsetneq X$, there exists a $G$-equivariant rational map $\alpha\colon X\dr X$ such that $Z_1$ is not contained in the indeterminacy locus of $\alpha$ and such that $\alpha(Z_1)\nsubseteq Z_2$.
        \item \label{weakversal} For any faithful $G$-variety $Y$, there exists a $G$-equivariant rational map $Y\dr X$.
    \end{enumerate}
    Then $X$ is versal.
\end{lemma}
\begin{proof}
    Let $U\subset X$ be a $G$-invariant Zariski open. Denote by $Z_2:=X-U$. Let $Y$ be a faithful $G$-variety. By Assumption \ref{weakversal}, there exists a $G$-equivariant rational map $f\colon Y\dr X$. Let $Z_1:=\overline{f(Y)}$. By Assumption \ref{Gmoving}, there exists a $G$-equivariant rational map $\alpha\colon X\dr X$ such that the restriction of $\alpha$ to $Z_1$ is defined, and such that $\alpha(Z_1)\nsubseteq Z_2$. Then $\alpha\circ f$ restricts to a $G$-equivariant rational map $Y\dr U$ as desired.
\end{proof}

\begin{example}
    Let $\Ab^G$ denote the regular representation of $G$. Then $\Ab^G$ is a versal $G$-variety.  Indeed, Lemma 3.1(b) of \cite{BuRe1} shows that $\Ab^G$ satisfies Assumption \ref{Gmoving} of Lemma \ref{lemma:showversal}, while Lemma 3.4 of \cite{BuRe1} shows that $\Ab^G$ satisfies Assumption \ref{weakversal}.
\end{example}

\subsection{Examples of versal \boldmath$G$-varieties}

In this subsection we use the tools from \S\ref{subsection:versality:criteria} to give examples of versal $G$-varieties.  We begin with a result essentially proven by Buhler-Reichstein in \cite{BuRe1}; we include a proof for completeness.

\begin{proposition}[{\bf Linear varieties are versal}]
\label{proposition:RD:linear}
    Let $G$ be a finite group. Let $V$ be any faithful linear $G$-variety. Then $V$ is versal.
\end{proposition}

\begin{proof}
    Because $\Ab^G$ is versal, it suffices to prove that for any proper $G$-invariant closed subvariety $Z\subset V$, there exists a $G$-equivariant map $f\colon \Ab^G\to V$ such that $f(\Ab^G)\nsubseteq Z$. Let $v\in V-Z$ be any point such that $|G\cdot v|=|G|$. Define
    \begin{align*}
        f_v\colon\Ab^G&\to V\\
        \sum_{g\in G} c_g g&\mapsto \sum_{g\in G} c_g(g\cdot v).
    \end{align*}
    Then $f_v$ is a $G$-equivariant linear embedding, and $f(\Ab^G)\nsubseteq Z$ as claimed.
\end{proof}

We highlight a specific instance of the above: while Hilbert asked about the resolvent degree of the permutation representation $\Cb^7$ of $S_7$, Proposition~\ref{proposition:RD:linear} implies
that that one can equivalently consider any faithful representation of $S_7$.  This gives an equivalent rephrasing of Hilbert's 13th problem, one for each faithful $S_n$-representation.
\begin{corollary}\label{cor:Hfaithandn1}
The following statements are true.
\begin{enumerate}
\item  \label{corollary:Hilbert:for:faithful}

    Let $n\geq 1$. Let $V$ be any faithful representation of $S_n, n\geq 2$.   Then
    \[\RD(S_n)=\RD(V\to V/S_n)=\RD(\tPc_n\to\Pc_n).\]
    In particular, $\RD(\tPc_n\to\Pc_n)\le\RD(\tPc_{n+1}\to\Pc_{n+1})$.
\item ({\bf Universality of $\RD(S_n)$})\label{corollary:nto1}
    Let $\tX\dr X$ be a generically $n$-to-$1$ rational cover. Then
    \begin{equation*}
        \RD(\tX\dr X)\le\RD(S_n).
    \end{equation*}
\end{enumerate}
\end{corollary}
\begin{proof}
    Proposition \ref{proposition:RD:linear}  gives the first equality of item \ref{corollary:Hilbert:for:faithful}, and shows that $\RD(V\to V/S_n)=\RD(W\to W/S_n)$ for any two faithful representations $V$ and $W$.  In particular, we can take $W=\Ab^n$ to be the standard permutation representation. Since $\Ab^n\to\Ab^n/S_n$ is the normalization of the branched cover $\tPc_n\to\Pc_n$, the second equality of Item \ref{corollary:Hilbert:for:faithful} follows from Lemma \ref{lemma:galois}.  Item \ref{corollary:nto1} now follows from Lemma \ref{lemma:universal}.
\end{proof}

Another equivalent restatement of the problem of computing $\RD(\tPc_n\to\Pc_n)$ comes from the following. Denote by $\Mc_{0,n}$ the moduli of $n$ distinct ordered points in $\Pb^1$.  More generally, let $\Cc_n(\Pb^m):=((\Pb^m)^{\times n}-\Delta)/\PGL_{m+1}$, where $\Delta\subset(\Pb^m)^{\times n}$ denotes the ``fat diagonal'', i.e. the locus of $n$-tuples in which at least two points coincide.

\begin{corollary}\label{cor:Cn}
    For $n\ge 5$, the moduli of marked, genus 0 curves $\Mc_{0,n}$ is a versal $S_n$-variety. In particular,
    \begin{equation*}
        \RD(S_n)=\RD(\Mc_{0,n}\to\Mc_{0,n}/S_n).
    \end{equation*}
    More generally, $\Cc_n(\Pb^m)$ is a versal $S_n$-variety for all $n\ge \max\{5,m+3\}$.
\end{corollary}
\begin{proof}
    There exists a dominant $S_n$-equivariant rational map $\Ab^n\dr \Mc_{0,n}$. More generally, consider the $m$-fold direct sum $(\Ab^n)^m$ of the permutation representation of $S_n$. This admits a dominant $S_n$-equivariant rational map $(\Ab^n)^m\dr ((\Pb^m)^{\times n}-\Delta)/\PGL_{m+1}=:\Cc_n(\Pb^m)$. The corollary now follows from Lemma \ref{lemma:dom} and Proposition \ref{proposition:RD:linear} once we verify that the $S_n$-action on $\Cc_n(\P^m)$ is faithful, but this follows from the assumptions that $n\ge \max\{5,m+3\}$.
\end{proof}

\section{Lines on smooth cubic surfaces}
\label{section:cubic:surfaces}

Since the problem of finding lines on smooth cubic surfaces connects with so many other problems, we devote an entire section to it.  We also look at this one example in depth because it demonstrates how resolvent degree can be an organizing principle that gives a single framework for many classical results.

\subsection{The moduli space of smooth cubic surfaces, and its covers}\label{subsection:h33covers}

Let $\Hc_{3,3}$ denote the moduli space of smooth cubic surfaces.  This is a $4$-dimensional quasi-projective variety, the quotient of a hypersurface complement $(\Pb^{19}-\Sigma)$ by the action of $\PGL_4$ induced from its action on $\Pb^3$.
Let $\Gr(2,4)$ denote the Grassmannian of projective lines in $\Pb^3$.  Let
\[\Hc_{3,3}(1):=\{(S,L)\in (\Pb^{19}-\Sigma)\times \Gr(2,4) : L\subset S\}/\PGL_4\]
be the moduli space of smooth cubic surfaces $S$ equipped with a line; here $\PGL_4$ acts diagonally.  Cayley and Salmon proved that the projection $\pi:\Hc_{3,3}(1)\to \Hc_{3,3}$ given by $\pi(S,L):=S$ is a $27$-sheeted covering, and so its monodromy is a subgroup of $S_{27}$.  However, the mondromy must preserve the intersection pattern of the $27$ lines.  Camille Jordan proved (see, e.g., \cite{Do} or \cite{Har} for a modern treatment) that the monodromy group of $\pi:\Hc_{3,3}(1)\to \Hc_{3,3}$ is isomorphic to the Weyl group $W(E_6)$.    Recall that this is the reflection group given by the Dynkin diagram :

\begin{center}
  \begin{tikzpicture}[scale=.4]
    \draw (-1,1) node[anchor=east]  {$E_6$};
    \foreach \x in {0,...,4}
    \draw[thick,xshift=\x cm] (\x cm,0) circle (3 mm);
       \foreach \y in {0,...,3}
    \draw[thick,xshift=\y cm] (\y cm,0) ++(.3 cm, 0) -- +(14 mm,0);
    \draw[thick] (4 cm,2 cm) circle (3 mm);
    \draw[thick] (4 cm, 3mm) -- +(0, 1.4 cm);
  \end{tikzpicture}
\end{center}
Here each vertex represents (reflection in the hyperplane perpendicular to) a root, and $W(E_6)$ has presentation with a generator $s_\alpha$ for each vertex of the diagram, with relations given by :
\begin{itemize}
\item $s_\alpha^2=1$ for all $\alpha$.
\item $(s_\alpha s_\beta)^2=1$ if $\alpha$ and $\beta$ are not connected by an edge.
\item $(s_\alpha s_\beta)^3=1$ if $\alpha$ and $\beta$ are connected by an edge.
\end{itemize}

$W(E_6)$ is a group of order $51840$; it contains the unique finite simple group of order 25920 as an index $2$ subgroup; we denote this group by $W(E_6)^+$.    Let $\Hc_{3,3}(27)$ denote the Galois closure of $\pi: \Hc_{3,3}(1)\to\Hc_{3,3}$; this is the (connected) Galois cover of $\Hc_{3,3}$ with deck group $W(E_6)$, corresponding to the kernel of the monodromy representation $\pi_1(\Hc_{3,3})\onto W(E_6)$.  We use the notation $\Hc_{3,3}(27)$ since this cover corresponds to the moduli space of $28$-tuples $(S;L_1,\ldots,L_{27})$ of smooth cubic surfaces equipped with $27$ lines with a choice of labelling of the intersection graph of the set of $27$ lines.

Let
\begin{equation}
\label{eq:lines:skew}
\Hc_{3,3}^{\rm skew}(r):=\{(S;L_1,\ldots,L_r)\in(\Pb^{19}-\Sigma)\times \Gr(2,4)^r : L_i\subset S,~L_i\cap L_j=\emptyset \ \forall i\neq j\}/\PGL_4.
\end{equation}
denote the moduli space of smooth cubic surfaces $S$ with a choice of $r\leq 6$ skew (i.e. disjoint) lines on $S$.  We remark that $\Hc_{3,3}^{\rm skew}(6)$ is connected; this follows for example from the fact that it is isomorphic to the moduli of $6$ generic points in $\Pb^2$ (cf. Section~\ref{subsection:moduli:p2} below). There is a cover $\Hc_{3,3}^{\rm skew}(r)\to\Hc_{3,3}$ given by $(S;L_1,\ldots,L_r)\mapsto S$.    This projection gives a (typically non-Galois) finite covering map $\Hc_{3,3}^{\rm skew}(r)\to\Hc_{3,3}$.

The action of $W(E_6)$ on $\Hc_{3,3}(27)$ is free on a Zariski open.  $W(E_6)\cong\Aut({\rm Pic}(S))$, and for any class $[L_0]$ of a line we have:
\[{\rm Stab}([L_0])\cong W(D_5)\cong (\Z/2\Z)^4\rtimes S_5\]
where the $S_5$ action on $(\Z/2\Z)^4$ is given by the standard $4$-dimensional irreducible permutation representation of $S_5$.  The action of $S_5$ on a marking is given by permuting the divisor classes of the $5$ lines $L_1,\ldots,L_5$ disjoint from $L_0$.  Further, $W(D_5)$ is generated by this $S_5$ together with a Cremona transformation.  Since the monodromy $W(E_6)$ acts transitively on the set of lines of any basepoint cubic, this implies that 

\begin{equation}
\label{eq:H33:W(D_5)}
\Hc_{3,3}(1)= \Hc_{3,3}(27)/W(D_5).
\end{equation}

We will see throughout this paper how many classical problems about smooth cubic surfaces can be rephrased as understanding various (branched) covers of $\Hc_{3,3}$; for problems about lines the covers are intermediate between $\Hc_{3,3}(27)\to\Hc_{3,3}$.    For now we give one example.

\para{Sch\"{a}fli's double sixes} One of the more well-studied types of configurations of lines on a smooth cubic surface $S$ is the so-called {\em (Schl\"{a}fli) double six}: it consists of two pairs $\{a_i\}$ and $\{b_j\}$ of $6$ disjoint lines on $S$ with intersection pattern given (in Schl\"{a}fli's original notation):
\begin{equation}
\label{equation:double-six}
\large\{
\begin{array}{cccccc}
a_1&a_2&a_3&a_4&a_5&a_6\\
b_1&b_2&b_3&b_4&b_5&b_6
\end{array}
\large\}
\end{equation}
where any line does not meet any of the lines in the same row or column, but does meet the other $5$ lines. See Figure~\ref{figure:double-six}.

The group $W(E_6)$ acts transitively on the set of $6$-tuples of disjoint lines on $S$, with stabilizer the symmetric group $S_6$.  There are thus $[W(E_6):S_6]=51840/720=72$ choices of such $6$-tuples.  Each such $6$-tuple determines a unique double-six, and since any double-six contains
$2$ such $6$-tuples, there are $72/2=36$ double-sixes.  Denote the moduli of smooth cubic surfaces equipped with a double-six by
\begin{equation*}
    \Hc_{3,3}(6,6):=\{(S,D): S\in\Hc_{3,3} \ \text{and $D$ is a double-six in $S$.}\}.
\end{equation*}
The stabilizer of a double-six is the maximal subgroup $S_6\times\Z/2\Z\subset W(E_6)$ (cf. \cite[Proposition 9.4, Theorem 9.5.2]{Do}).  We can thus make the identification
\begin{equation}
\label{eq:doublesix}
\Hc_{3,3}(27)/(S_6\times \Z/2\Z)= \Hc_{3,3}^{\rm skew}(6)/(S_6\times\Z/2\Z)=\Hc_{3,3}(6,6)
\end{equation}
where the first equality comes from \eqref{eq:six:determine} below.

\subsection{Finding $27$ lines from a given line}
\label{subsection:cubic:surfaces}

In this subsection we consider the following problem: given a single line on a smooth cubic surface, how hard is it to find more lines?  We will prove that given one line, the problem of finding the other $27$ lines has resolvent degree $1$, by which we mean $\RD(\Hc_{3,3}(27)\to\Hc_{3,3}(1))=1$.   This result is essentially 100 years old.
For a nice modern reference, see Dolgachev's book \cite{Do}, page 480.

\begin{proposition}[{\bf Finding lines on a cubic surface, given a line}]\label{prop:27from1}
With notation as above :
\[\RD(\Hc_{3,3}(27)\to\Hc_{3,3}(1))=1.\]
\end{proposition}

This is in contrast to Harris's Theorem \cite{Har} that $\Hc_{3,3}(27)\to\Hc_{3,3}(1)$ is not solvable by radicals.

\begin{proof}   We take the argument from the classic \cite{Hil}, page 349.
Suppose that we are given a smooth cubic surface $S=V(f)$ and a line $\ell_0$ on $S$.  The line $\ell_0$ is given as a zero set of two linear forms : $\ell_0=V(A_1,A_2)$.  Since $\ell_0\subset S$ this gives
\[f=A_1Q_1+A_2Q_2\]
for quadratic forms $Q_1,Q_2$.  Consider the pencil of planes
\[\Pi(\lambda_1,\lambda_2)=V(\lambda_1A_1-\lambda_2A_2)\]
through the line $\ell_0$.  Each plane in this pencil intersects $S$ in the union of $\ell_0$ and a conic $C(\lambda_1,\lambda_2)$ on $S$.    One can check that the discriminant of each $C(\lambda_1,\lambda_2)$ is a homogeneous polynomial $P(\lambda_1,\lambda_2)$ of degree $5$, and that the general $P(\lambda_1,\lambda_2)$ has $5$ distinct roots.   Each of these solutions gives a reducible conic on $S$.  Since $S$ is smooth none of these is a double line.

We thus have found five distinct pairs of distinct lines $\ell_i,\ell'_i, 1\leq i\leq 5$, and in fact all $10$ of these lines are distinct from each other and from $\ell_0$, giving $11$ lines on $S$.  The important thing for us is to observe that the $\ell_i$ are pairwise disjoint for $0\leq i\leq 5$.   Since we obtained these with a degree $5$ polynomial it follows that
\[ \RD(\Hc_{3,3}^{\rm skew}(5)\to\Hc_{3,3}(1))\leq \RD(\widetilde{\Pc_5}\to\Pc_5)=1.\]

We can repeat the above procedure with $\ell_0$ replaced by any $\ell_i$ or $\ell'_i$ to find the remaining $27$ lines; that is, to prove
\begin{equation}
\label{eq:restof:lines}
\RD(\Hc_{3,3}(27)\to\Hc_{3,3}^{\rm skew}(5))\leq 1
\end{equation}
Alternately, Harris proves in \cite{Har} that the monodromy of the cover $\Hc_{3,3}(27)\to\Hc_{3,3}^{\rm skew}(3)$ is in fact solvable, hence so is the monodromy of $\Hc_{3,3}(27)\to\Hc^{\rm skew}_{3,3}(5)$, giving \eqref{eq:restof:lines}.  Lemma \ref{lemma:max} (on $\RD$ of a tower) then implies

\[\begin{array}{ll}
\RD(\Hc_{3,3}(27)\to\Hc_{3,3}(1))&\leq \max\{\RD(\Hc_{3,3}(27)\to\Hc^{\rm skew}_{3,3}(5)), \RD(\Hc^{\rm skew}_{3,3}(5)\to\Hc_{3,3}(1))\\
&=\max\{1,1\}=1
\end{array}\]
giving the proposition.
\end{proof}

\subsection{Finding a single line}
The following fundamental problem still remains.  As we will see throughout this paper, it relates to many other problems about resolvent degree.

\begin{problem}
\label{problem:line1}
Determine $\RD(\Hc_{3,3}(1)\to\Hc_{3,3})$.
\end{problem}

While there is a vast literature on lines on smooth cubic surfaces, and while much of it concerns relationships between various intermediate covers of $\Hc_{3,3}(27)\to\Hc_{3,3}$, there are far fewer results on Problem \ref{problem:line1}.  The best results of which we are aware are due to Burkhardt \cite{Bur}, following a suggestion of Klein (see \cite[Ch. 4.3.2]{Hu} for a modern treatment).

\begin{theorem}[{\bf Burkhardt, Klein}]\label{thm:RD:lines}
    Let $k$ be any field of characteristic $\neq 2,3$. Then \[\RD_k(\Hc_{3,3}(1)\to\Hc_{3,3})\le 3.\]
\end{theorem}

The proof of Theorem~\ref{thm:RD:lines} will use the following proposition, the first part of which we learned from \cite[Lemma 6.1]{DuRe1}.

\begin{proposition}[{\bf Finding the \boldmath$27$ lines is versal}]
\label{prop:linesversal}
    For any $G\subset W(E_6)$, the $k$-variety $\Hc_{3,3}(27)$ is a versal $G$-variety. In particular
    \[\RD_k(W(E_6))=\RD_k(\Hc_{3,3}(27)\to \Hc_{3,3}).\]
\end{proposition}
\begin{proof}
    Let $\mathfrak{h}$ denote a Cartan subalgebra of any simple Lie $k$-algebra of type $E_6$.  Let $W(E_6)$ act on $\mathfrak{h}$ via the defining representation, and let $\Ab(\mathfrak{h})$ denote the corresponding faithful linear $W(E_6)$-variety.  Then by \cite[Lemma 6.1]{DuRe2}, there exists a $W(E_6)$-equivariant dominant rational map
    \begin{equation*}
        \Ab(\mathfrak{h})\dr \Cc_6'(\Pb^2)\dr \Hc_{3,3}(27).
    \end{equation*}
   Applying Proposition \ref{proposition:RD:linear} and Lemma \ref{lemma:dom}, the proposition follows.
\end{proof}

\begin{proof}[Proof of Theorem \ref{thm:RD:lines}]
    Recall that $W(E_6)\cong W(E_6)^+\rtimes \Z/2\Z$. By Theorem \ref{thm:JH},
    \begin{equation*}
        \RD(W(E_6))=\max\{\RD(W(E_6)^+),\RD(\Z/2\Z)\}=\RD(W(E_6)^+).
    \end{equation*}
    The group $W(E_6)^+$ has an action on $\Pb^3$ defined over $\Z[\sqrt{-3}]$ (see, e.g., \cite{Atl}); therefore after adjoining $\sqrt{-3}$ to $k$ ($\RD=1$), this action is defined over $k$. By Proposition \ref{prop:whyRDversal}, it suffices to prove that $\Pb^3$ is solvably-versal for $W(E_6)^+$. Note that there is an isomorphism $\Sp_4(\F_3)/\F_3^\times\cong W(E_6)^+$ and the $W(E_6)^+$-action on $\Pb^3$ lifts to a faithful linear action of $\Sp_4(\F_3)$ on $\Ab^4$ defined over $\Z[\sqrt{-3}]$.

    Given any $W(E_6)^+$-variety $X$, the obstruction to realizing it as a quotient of a faithful $\Sp_4(\F_3)$-variety is the associated Brauer class in $H^2_{\et}(k(X/W(E_6)^+);\mu_2)$. However, by Merkurjev's Theorem \cite{Me1}, any class in $H^2_{\et}(k(X)^{W(E_6)^+};\mu_2)$ trivializes over some multi-quadratic extension of $k(X/W(E_6)^+)$. We conclude that there exists a faithful $\Sp_4(\F_3)$-variety $\tX$ such that $\tX/\Sp_4(\F_3)\dr X/W(E_6)^+$ is a generically 2-to-1 rational cover. By Proposition \ref{proposition:RD:linear}, $\Ab^4$ is a versal $\Sp_4(\F_3)$ variety, and by the definition of versality, there exists an $\Sp_4(\F_3)$-equivariant rational map $\tX\dr \Ab^4$. Composing with the projection $\Ab^4\dr \Pb^3$, we obtain a $W(E_6)^+$-equivariant rational map $\tX/\Z/2\Z\dr \Pb^3$. But this shows that $\Pb^3$ is $W(E_6)^+$-solvably versal as claimed. We conclude
    \begin{align*}
        \RD(\Hc_{3,3}(1)\to\Hc_{3,3})&=\RD(\Hc_{3,3}(27)\to\Hc_{3,3})\tag{by Lemma \ref{lemma:galois}}\\
        &=\RD(W(E_6))\tag{by Proposition \ref{prop:linesversal}}\\
        &=\RD(W(E_6)^+)\\
        &=\RD(\Pb^3\to\Pb^3/W(E_6)^+)\le \dim(\Pb^3)=3.
    \end{align*}
\end{proof}

\subsection{Moduli of \boldmath$6$ points in \boldmath$\Pb^2$}
\label{subsection:moduli:p2}

Let $\Sigma\subset(\Pb^2)^6$ denote the subvariety of $6$-tuples of distinct points in $\Pb^2$ that are {\em non-generic}; that is, with either $3$ colinear or with all $6$ points lying on a conic.    Let
\[\Cc_6'(\Pb^2):=((\Pb^2)^6\setminus \Sigma)/\PGL_3\]
be the moduli space of generic $6$-tuples in $\Pb^2$.  For any (orbit representative of) $(z_1,\ldots,z_6)\in\Cc_6'(\Pb^2)$, blowing up $\Pb^2$ at each $z_i$ gives a smooth cubic surface $S_{(z_1,\ldots,z_6)}$ equipped with a $6$-tuple $(L_1,\ldots ,L_6)$ of $6$ skew lines corresponding to the exceptional divisors.  Every smooth cubic surface arises in this way, and indeed it is classical that the map
\[\psi:\Cc_6'(\Pb^2)\to \Hc_{3,3}^{\rm skew}(6)\]
defined by $\psi(z_1,\ldots,z_6):=(S_{(z_1,\ldots,z_6)};L_1,\ldots,L_6)$ is birational, where $\Hc_{3,3}^{\rm skew}(6)$ is defined in \ref{eq:lines:skew}.    It is classical that $6$ skew lines $L_1,\ldots,L_6$ on a smooth cubic surface $S$ determine via explicit formulas the other $21$ lines on $S$; see, e.g.,  \S 4 of \cite{Hu}.  The ordering on the $L_i$ determines an ordering on the set of all $27$ lines, from which we deduce that there is an isomorphism
\begin{equation}
\label{eq:six:determine}
\tau: \Hc_{3,3}^{\rm skew}(6)\stackrel{\cong}{\to} \Hc_{3,3}(27).
\end{equation}
Composition thus gives an isomorphism
\[\tau\circ\psi: \Cc_6'(\Pb^2)\stackrel{\cong}{\to} \Hc_{3,3}(27).\]

The permutation action of $S_6$ on $(\Pb^2)^6$ leaves invariant $\Sigma$ and induces a well-defined action of $S_6$ on $\Cc_6'(\Pb^2)$.  As explained in e.g.\ \cite[\S 3]{Se}, this action extends (via adding a birational
automorphism induced by an explicit Cremona transformation) to an action by birational automorphisms of $W(E_6)$ on $ \Cc_6'(\Pb^2)$ for which the isomorphism $\tau\circ \psi$ is $W(E_6)$-equivariant.   We remark that the $W(E_6)$ action on $\Cc_6'(\Pb^2)$ is not regular.

As a corollary to Proposition \ref{prop:linesversal} and Theorem \ref{thm:RD:lines}, we have the following.
\begin{corollary}\label{cor:RD:6pts}
    Let $k$ be a field of characteristic $\neq 2,3$. For any $G\subset W(E_6)$, the $k$-variety $\Cc_6'(\Pb^2)$ is a versal $G$-variety. In particular,
    \begin{equation*}
        \RD(\Cc_6(\Pb^2)\dr \Cc_6'(\Pb^2)/W(E_6))=\RD(W(E_6))\le 3.
    \end{equation*}
\end{corollary}

\subsection{Pentahedral form}

Pentahedral form is a classical normal form for smooth cubic surfaces. We now consider this form from the point of view of resolvent degree.

For any fixed $[a_0:\cdots :a_4]\in\Pb^4$ the equations

\begin{equation}
\label{eq:pentahedral1}
\begin{array}{c}
a_0X_0^3+a_1X_1^3+a_2X_2^3+a_3X_3^3+a_4X_4^3=0\\
\\
X_0+X_1+X_2+X_3+X_4=0
\end{array}
\end{equation}
determine a cubic surface in $\Pb^3$.  Any permutation of the $a_i$ gives an isomorphic cubic surface.
We thus have a family $\Pb^4/S_5$ of cubic surfaces.  The elementary symmetric functions $\sigma_1,\ldots,\sigma_5$ in the $a_i$ give coordinates on $\Pb^4/S_5$. The open subset
\[\Pc:=\{[\sigma_1:\cdots :\sigma_5]: \sigma_5\neq 0\}\subset\Pb^4/S_5\]
is the family of smooth cubic surfaces admitting a {\em (proper) pentahedral form}, and the classifying map $\tau: \Pc\to \Hc_{3,3}$ is an open embedding (see, e.g.\cite[Lemma 3.5]{EJ}).  The hyperplane complement
\begin{equation*}
    \tPc:=\Pb^4-\bigcup_{i=0}^4 \{a_i=0\}=\Pb^4\times_{\Pb^4/S_5} \Pc
\end{equation*}
is the space of smooth cubic surfaces in {\em proper pentahedral form}. We can pull back the cover $\Hc_{3,3}(27)\to\Hc_{3,3}$ along the map
\begin{equation*}
    \tPc\to\Pc\to^\tau\Hc_{3,3}
\end{equation*}
to obtain a cover $\tPc(27)\to\tPc$.

\begin{proposition}
Pentahedral form is an accessory irrationality: the cover $\tPc(27)\to\tPc$ has Galois group $W(E_6)$. Further, the total space $\tPc(27)$ has two connected components, each component is preserved by the index two subgroup $W(E_6)^+\subset W(E_6)$, and the components are permuted under the action of the full group $W(E_6)$.
\end{proposition}
\begin{proof}
The cover
\begin{equation*}
    \Hc_{3,3}(27)/W(E_6)^+\to\Hc_{3,3}
\end{equation*}
corresponds to adjoining a square-root of the discriminant of the cubic. Note that the discriminant of the cubic equals the discriminant of each of its pentahedral forms (cf. \cite[\S 9.4.5]{Do}). As a consequence, the map $\tPc\to \Hc_{3,3}$ factors through the cover
\begin{equation*}
    \tPc\to \Hc_{3,3}(27)/W(E_6)^+.
\end{equation*}
The map $\tPc\to \Hc_{3,3}(27)/W(E_6)^+$ is a Galois $A_5$-cover of its image. On the other hand, because $W(E_6)$ only has proper, nontrivial quotients of order $2$; in particular $A_5$ is not such a quotient.  We conclude that $\tPc\to\Hc_{3,3}(27)/W(E_6)^+$ and $\Hc_{3,3}(27)\to\Hc_{3,3}(27)/W(E_6)^+$ share no intermediate covers, and thus
\begin{equation*}
    \Hc_{3,3}(27)\times_{\Hc_{3,3}(27)/W(E_6)^+}\tPc\to\tPc
\end{equation*}
is a connected Galois $W(E_6)^+$ cover. From the above, each of the two components of $\tPc(27)$ is isomorphic to this connected $W(E_6)^+$ cover, with the full group $W(E_6)$ interchanging the two components.
\end{proof}

\subsection{Hexahedral form}
\label{subsection:hexahedral}
The following is taken from Example 3.7 of \cite{EJ}.    Let $H\cong\Pb^4$ be the hyperplane in $\Pb^5$ given by $a_0+\cdots+a_5=0$.  The group $S_6$ acts on $H$ with quotient isomorphic to the weighted projective space $\Pb(2,3,4,5,6)$.  The key thing is a sequence of maps (using our notation as above):

\begin{equation}
\label{eq:hexahedral3}
\Hc_{3,3}(27)\stackrel{t_1}{\to} H\stackrel{t_2}{\to}H/S_6\stackrel{t_3}{\to}\Hc_{3,3}
\end{equation}
where $t_1$ is an unramified $2$-sheeted cover, $t_3$ is an unramified $36$-sheeted cover, and
$t_2$ is a generically $720$-to-$1$ branched cover.   Note that the fact that $t_1$ is $2$-sheeted, so that $\RD(\Hc_{3,3}(27)\to H)=1$, corresponds to the classical fact that, given a smooth cubic surface $S$ in hexahedral form, one can write down explicitly (as a linear function in the coefficients of $S$) a formula for 15 of the lines on $S$ (see, e.g. \cite{Do}, section 9.4). One can obtain the remaining 12 lines by adjoining a square root. By the classification of maximal subgroups in $W(E_6)$ (see \cite[Theorem 9.5.2]{Do}), the stabilizer of an unordered hexahedral form is isomorphic to $S_6\times \Z/2\Z$. As a consequence, the moduli of unordered hexabedral forms $H/S_6$ is isomorphic over $\Hc_{3,3}$ to the moduli of cubics equipped with a double-six:
\begin{equation*}
    \xymatrix{
        H/S_6 \ar[rr]^\cong \ar[dr] && \Hc_{3,3}(6,6)\ar[dl]\\
        & \Hc_{3,3}
    }.
\end{equation*}
Moreover,
\begin{align*}
    \RD(\Hc_{3,3}(27)\to \Hc_{3,3})&=\max\{\RD(H\to H/S_6),\RD(H/S_6\to\Hc_{3,3})\}\\
    &\le\max\{2,\RD(H/S_6\to\Hc_{3,3})\}
\end{align*}
where the last inequality follows from
\begin{align*}
    \RD(H\to H/S_6)&=\RD(S_6)\tag{by Proposition \ref{proposition:RD:linear} and Lemma \ref{lemma:dom}}\\
    &\le 2 \tag{by Hamilton's bound}.
\end{align*}

\section{Bitangents to plane quartics}
\label{section:bitangents}

The story of $28$ bitangents on a smooth plane quartic is analogous to that for the $27$ lines on a smooth cubic surface, and indeed the two are directly related, as we will see in \S\ref{subsection:relating-lines-and-bitangents} below.

\subsection{The moduli space of smooth planar quartics, and its covers}\label{subsection:h42covers}

Let $\Hc_{4,2}$ denote the moduli space of smooth quartic curves in $\Pb^2$.   This is a $6$-dimensional quasi-projective variety, the quotient of a hypersurface complement $(\Pb^{14}-\Sigma)$ by the action of $\PGL_3$ induced from its action on $\Pb^2$.  Let $\Gr(2,3)$ denote the Grassmannian of projective lines in $\Pb^2$.  Jacobi proved in 1850 that any smooth plane quartic $C$ has precisely $28$ {\em bitangents}; that is, lines $T\subset\Pb^2$ that are tangent to $C$ at two points (counted with multiplicity).   Let
\[\Hc_{4,2}(1):=\{(C,L)\in (\Pb^{14}-\Sigma)\times \Gr(2,3) : \text{$L$ bitangent to $C$}\}/\PGL_3\]

be the moduli space of smooth plane quartics equipped with a bitangent; here $\PGL_3$ acts diagonally.  The map $(C,L)\mapsto C$ is a $28$-sheeted covering space.    Let $\Hc_{4,2}(28)$ denote the Galois closure of $\pi: \Hc_{4,2}(1)\to\Hc_{4,2}$; this is a (connected) Galois cover of $\Hc_{4,2}$.
We use the notation $\Hc_{4,2}(28)$ since this cover corresponds to the moduli space of $29$-tuples $(C;L_1,\ldots,L_{28})$ of smooth plane quartics equipped with $28$ lines with a choice of labelling of the intersection graph of the set of $28$ lines.

The deck group of the Galois cover $\Hc_{4,2}(28)\to\Hc_{4,2}$ is the same as the monodromy group of the cover $\Hc_{4,2}(1)\to \Hc_{4,2}$.  This group is isomorphic to the unique simple group of order $1,451,520$, which we denote $W(E_7)^+$. There exists a split injection $W(E_7)^+\into W(E_7)$, the Weyl group of type $E_7$.  Recall that $W(E_7)$ is the reflection group with Dynkin diagram :

\begin{center}
  \begin{tikzpicture}[scale=.4]
    \draw (-1,1) node[anchor=east]  {$E_7$};
    \foreach \x in {0,...,5}
    \draw[thick,xshift=\x cm] (\x cm,0) circle (3 mm);
    \foreach \y in {0,...,4}
    \draw[thick,xshift=\y cm] (\y cm,0) ++(.3 cm, 0) -- +(14 mm,0);
    \draw[thick] (4 cm,2 cm) circle (3 mm);
    \draw[thick] (4 cm, 3mm) -- +(0, 1.4 cm);
  \end{tikzpicture}
\end{center}

It is given by order $2$ generators $s_\alpha$, one for each vertex, satisfying the same relations as $W(E_6)$ given above.  $W(E_7)$ has order $2,903,040$, and is a direct product of $\Z/2\Z$ with $W(E_7)^+$.  The action of $W(E_7)^+$ on $\Hc_{4,2}(28)$ is free on a Zariski open.  $W(E_7)^+\cong\Aut({\rm Pic}(C)[2])$, and for any class $[L_0]$ of a line we have:
\[{\rm Stab}([L_0])\cong W(E_6)\]
This action is most easily seen as follows (cf. \cite[Chapter IX.2]{DoOr}).  The moduli $\Hc_{4,2}(28)$ is the target of a generically 2-to-1 dominant rational map
\begin{equation*}
    \Cc_7(\Pb^2)\dr \Hc_{4,2}(28).
\end{equation*}
Concretely, given 7 points $\{x_1,\ldots,x_7\}\subset\Pb^2$ in general position, form the degree 2 Del Pezzo surface $V(x_1,\ldots,x_7)$ by blowing up $\Pb^2$ at these points.  The anti-canonical map
\begin{equation}\label{2DP}
    V\to\Pb^2.
\end{equation}
realizes $V$ as a 2-fold branched cover, branched over a quartic curve $C$, and takes every exceptional curve on $V$ to a bitangent of $C$. By Proposition 1 of \cite[Chapter IX.2]{DoOr}, this gives a 2-fold covering map
\begin{equation}\label{2DPM}
    U\to \Hc_{4,2}(28)
\end{equation}
where $U\subset\Cc_7(\Pb^2)$ is the locus of points in general position, and the map sends $V$ with its exceptional curves to $C$ with its 28 bitangents.  The Weyl group $W(E_7)$ acts on $\Cc_7(\Pb^2)$ (via the Coble representation) and this action factors through the projection
\begin{equation*}
    W(E_7)\cong \Z/2\Z\times W(E_7)^+\onto W(E_7)^+.
\end{equation*}
The map \eqref{2DPM} is equivariant for this action (see \cite[Chapter IX]{DoOr}, esp.\ p.\ 194, for a verification of this  equivariance). Under the map \eqref{2DPM}, the stabilizer of a bitangent lifts to the stabilizer of a marked point on $\Cc_7(\Pb^2)$, i.e. to $W(E_6)\subset W(E_7)^+\subset W(E_7)$.

Just as for lines on cubics, we will see throughout this paper how many classical problems about smooth quartic curves can be rephrased as understanding various (branched) covers of $\Hc_{4,2}$; for problems about bitangents the covers are intermediate between $\Hc_{4,2}(28)\to\Hc_{4,2}$. We now give several examples.

\para{Aronhold sets}
One of the more well-studied types of configurations of bitangents on a smooth plane quartic curve $C$ is the so-called {\em Aronhold set}. Recall that a collection of $n\ge 3$ bitangents on a smooth plane quartic is {\em asyzygetic} (resp. {\em syzygetic}) if the collection of $2n$ points of contact of the bitangents with the quartic are not (resp. are) contained in a conic.

\begin{definition}[{\bf Aronhold set of bitangents}]
\label{definition:Aronhold}
An {\em Aronhold set} $\Ac$ on a smooth plane quartic $C$ is an asyzygetic, unordered set of seven bitangents $\{T_1,\ldots ,T_7\}$ on $C$.  An {\em Aronhold basis} is an Aronhold set with an ordering of its elements.
\end{definition}
Let $\Hc_{4,2}(\tAc)$ denote the moduli of smooth plane quartics equipped with an Aronhold basis, and let $\Hc_{4,2}(\Ac)$ denote the moduli of smooth plane quartics equipped with an Aronhold set. Note that the forget-the-ordering map is a Galois $S_7$-cover
\begin{equation*}
    \Hc_{4,2}(\tAc)\to\Hc_{4,2}(\Ac).
\end{equation*}
Aronhold sets have been studied for over a century (for recent treatments, see e.g. \cite{DoOr} or \cite[Chapter 6.1.2]{Do}). One of the reasons is that an Aronhold basis on $C$ determines the other $21$ bitangents to $C$, i.e. we have an $W(E_7)^+$-equivariant isomorphism
\begin{equation*}
    \Hc_{4,2}(\tAc)\to^{\cong} \Hc_{4,2}(28).
\end{equation*}
Perhaps even more surprising, an Aronhold basis in fact determines the equation for $C$ itself \cite{Le}. The group $W(E_7)^+$ acts simply transitively on the set of Aronhold bases, and thus acts transitively on the set of Aronhold sets, with stabilizer the symmetric group $S_7$.  There are thus $[W(E_7)^+:S_7]=288$ choices of Aronhold sets. The complexity of finding an Aronhold basis, given an Aronhold set, as measured by resolvent degree, is equivalent to Hilbert's 13th problem, as we show in Theorem \ref{theorem:S7:varieties}.

\para{Steiner Complexes}
A second well-studied type of configuration of bitangents on a smooth quartic curve is the {\em Steiner complex} (cf. \cite[Chapter XIX.3]{Hil} and \cite[Chapter 6.1.2]{Do}).
\begin{definition}[{\bf Steiner complex of bitangents}]
    A {\em Steiner complex} of bitangents on a smooth plane quartic $C$ is an unordered collection of six unordered pairs of bitangents $\{(\alpha_1,\beta_1),\ldots,(\alpha_6,\beta_6)\}$ such that any two pairs give a syzygetic collection of bitangents.
\end{definition}
Any two bitangents determine a Steiner complex, and any one of the six pairs of a Steiner complex determine the same complex, so there are $\binom{28}{2}/6=378/6=63$ Steiner complexes. Denote the moduli of smooth plane quartics equipped with a Steiner complex by
\begin{equation*}
    \Hc_{4,2}(\St):=\{(C,S): C\in\Hc_{4,2} \ \text{and $S$ is a Steiner complex for $C$.}\}.
\end{equation*}
The group $W(E_7)^+$ acts transitively on the set of Steiner complexes, and the stabilizer of a Steiner complex is isomorphic to $W(D_6)\cong (\Z/2\Z)^{\times 5}\rtimes S_6$, where the action of $S_6$ is via its standard 5-dimensional permutation representation. We can thus make the identification
\begin{equation}
\Hc_{4,2}(\St)=\Hc_{4,2}(28)/W(D_6)= \Hc_{4,2}(\tAc)/W(D_6).
\end{equation}
where the second equality comes from the fact that an Aronhold basis determines the remaining 21 lines.

\para{Cayley Octads}
A third configuration of classical interest is the {\em Cayley octad} (cf. \cite[Chapter 6.3.2]{Do}).
\begin{definition}
    A {\em Cayley octad} is a collection of 8 distinct unordered points in $\Pb^3$ that arises as a complete intersection of 3 quadrics. Denote the moduli space of Cayley octads by $\Cay$.
\end{definition}
There is a close relationship between Cayley octads and smooth plane quartics, which is summed up in the \cite[Chapter 6.3]{Do} (especially Corollary 6.3.12). In particular, the moduli of plane quartics equipped with an Aronhold set $\Hc_{4,2}(\Ac)$ admits an 8-to-1 covering map to the moduli space of Cayley octads, which is in turn birational to the moduli space of smooth plane quartics equipped with an even $\theta$-characteristic:
\begin{equation*}
        \Hc_{4,2}(\Ac)\to^{8:1}\Cay\simeq\Hc_{4,2}(\theta^{\rm ev}).
\end{equation*}
Moreover, the group $W(E_7)^+$ acts transitively on the set of Cayley octads, respectively even $\theta$-characteristics, and the stabilizer of an octad, respectively even $\theta$-characteristic, is $S_8$.

\subsection{The resolvent degree of finding bitangents to plane quartics}
In this subsection we consider the resolvent degree of the problem of finding bitangents on smooth plane quartics.

\begin{proposition}[{\bf Finding \boldmath$28$ bitangents, given $2$}]
\label{prop:28from2}
With the notation as above:
\[\RD(\Hc_{4,2}(28)\to\Hc_{4,2}(2))=1.\]
\end{proposition}

The proof of Proposition~\ref{prop:28from2} that we now give should feel similar to the proof of Proposition \ref{prop:27from1}, and indeed we formalize this similarity as a precise statement in
\S\ref{subsection:relating-lines-and-bitangents}.  We include the proof here for its beauty and historical interest.

\begin{proof}
Now, since the $T_i$ are distinct, any two intersect in a single point.  Let

\[\Hc_{4,2}'(2):=\{(C;T_1,T_2)\in \Hc_{4,2}(2): T_1\cap T_2 \not \in C\}.\]

This is a Zariski open subset of $\Hc_{4,2}(2)$.  It is enough to prove the theorem for the pullback
cover $\Hc'_{4,2}(28)\to\Hc'_{4,2}(2)$.  The advantage of $\Hc'_{4,2}(2)$ is that it gives us $4$ points of contact, $2$ each from $T_1\cap C$ and $T_2\cap C$.  We can then perform a classical construction, which we take from the 1920 book \cite{Hil}, which posits (see p.334 of \cite{Hil}):

\medskip
{\it ``Through the four points of contact of two bitangents of a non-singular quartic pass five conics each of which passes through the points of contact of two more bitangents.''}

\medskip
More precisely, let $(C;T_1,T_2)\in\Hc'_{4,2}(2)$ be given. We consider $\Pb^2$ with coordinates $[x:y:z]$.  By picking representatives in the $\PGL_3$ orbit of $(C;T_1,T_2)$, we can assume that $T_1$ and $T_2$ are given by the equations $x=0$ and $y=0$, respectively.  The assumption that $C$ has a bitangent given by $x=0$ and a bitangent given by $y=0$ puts the equation of $C$ in a very special form, namely:

\begin{equation}
\label{eq:quartic1}
C:=\{[x:y:z]\in\Pb^2: xy(U+2kV+t62xy)-(V+txy)^2=0\}
\end{equation}
for some $t$, where $U=0$ and $V=0$ are conics.  Consider the condition that $U+2kV+t^2xy$ factors as a product of linear forms $p(x,y,z)$ and $q(x,y,z)$.  One can check that this condition is a degree $5$ polynomial in $t$.  For such $t$ the equation  \eqref{eq:quartic1} for the quartic $C$ then becomes
\[xyp(x,y,z)q(x,y,z)-W^2=0\]
where $W:=V+txy$.  It is then clear that the lines given by $p=0$ and $q=0$ are both bitangent to $C$.  Further, the conic $W=0$ passes through the eight points of contact of the four bitangents $x=0, y=0, p=0, q=0$.  We have thus proven that
\begin{equation}
\label{eq:quartic2}
\RD(\Hc'_{4,2}(4)\to\Hc'_{4,2}(2))\leq\RD(\widetilde{{\mathcal P}}_5\to{\mathcal P}_5)=1
\end{equation}
where $\Hc'_{4,2}(4)$ is the pullback to $\Hc'_{4,2}(2)$ of the cover $\Hc_{4,2}(4)\to\Hc_{4,2}(2)$.  Although we will not need it, we remark that there are $5$ distinct roots of the degree $5$ polynomial determining such $t$, and so this gives us $5$ additional pairs of bitangents to $C$, for a total of $2+5\cdot 2=12$ bitangents.

Harris \cite{Har} proves the following: given any three bitangents whose points of contact lie on a conic, or any four whose points of contact do not, we can solve for the remaining ones in radicals; further, no smaller sets suffice.  This in particular gives that the cover $\Hc'_{4,2}(28)\to\Hc'_{4,2}(4)$ is solvable by radicals, and so has resolvent degree equal to $1$.  Combining this with \eqref{eq:quartic2} thus gives
\[\RD(\Hc_{4,2}(28)\to\Hc_{4,2}(2))=1\]
as desired.
\end{proof}

Proposition \ref{prop:28from2} naturally suggests the following fundamental problem.
\begin{problem}[{\bf Finding bitangents on smooth quartics}]
\label{problem:bitangent3}
Compute the following:
\begin{enumerate}
\item $\RD(\Hc_{4,2}(28)\to\Hc_{4,2}(1))$.
\item $\RD(\Hc_{4,2}(1)\to\Hc_{4,2})$.
\end{enumerate}
\end{problem}
In the next section, we relate this to the problem of finding lines on cubic surfaces, and in Section \ref{section:equivalence}, we put this problem in the context of Hilbert's 13th problem and Hilbert's Octic Conjecture.
%
%
%

\subsection{Relating lines on cubic surfaces to bitangents on plane quartics}
\label{subsection:relating-lines-and-bitangents}

In this subsection we relate the resolvent degrees of two classical problems: finding a line on a smooth cubic surface and finding a bitangent on a smooth quartic curve in $\Pb^2$.  We then relate these to the resolvent degrees of other problems.

\begin{theorem}
\label{theorem:quartics-to-cubics}
For any subgroup $G\subset W(E_6)\subset W(E_7)^+$,
\[\RD(G)=\RD(\Hc_{3,3}(27)\to\Hc_{3,3}/G)= \RD(\Hc_{4,2}(28)\to \Hc_{4,2}/G).\]
In particular:
\begin{enumerate}
    \item $\RD(W(D_5))=\RD(\Hc_{3,3}(27)\to\Hc_{3,3}(1))=\RD(\Hc_{4,2}(28)\to\Hc_{4,2}(2))=1$.
    \item $\RD(W(E_6))=\RD(\Hc_{3,3}(27)\to\Hc_{3,3})=\RD(\Hc_{4,2}(28)\to\Hc_{4,2}(1))\le 3$.
\end{enumerate}
Similarly, for any other subgroup $G\subset W(E_7)^+$,
\[\RD(G)=\RD(\Hc_{4,2}(28)\to\Hc_{4,2}/G)\]
In particular:
\begin{enumerate}
    \item $\RD(S_7)=\RD(\Hc_{4,2}(28)\to\Hc_{4,2}(\Ac))\le 3$.
    \item $\RD(S_8)=\RD(\Hc_{4,2}(28)\to\Hc_{4,2}(\theta^{ev}))\le 4$.
    \item $\RD(W(E_7)^+)=\RD(\Hc_{4,2}(28)\to\Hc_{4,2})$.
\end{enumerate}
\end{theorem}

We will deduce Theorem~\ref{theorem:quartics-to-cubics}  from the following, which should be compared with Proposition \ref{prop:linesversal} above.

\begin{proposition}[{\bf Versality of the bitangents problem}]
\label{prop:bitangentsversal}
   For any $G\subset W(E_7)^+$, the $k$-variety $\Hc_{4,2}(28)$ is a versal $G$-variety.
\end{proposition}
\begin{proof}
    We recall a construction due to Dolgachev--Ortland \cite[Chapter IX]{DoOr}, which in its essentials dates to Coble. We claim there exists a sequence of $W(E_7)$-equivariant  dominant rational maps
    \begin{equation}\label{bitan}
        \Ab(\mathfrak{h})\dr\Pb(\mathfrak{h})\dr \Cc_7(\Pb^2)\dr\Hc_{4,2}(28)
    \end{equation}
    where $\Ab(\mathfrak{h})$ denotes the variety given by a Cartan subalgebra of a simple Lie group of type $E_7$, with its canonical $W(E_7)$-action. By Proposition \ref{proposition:RD:linear}, $\Ab(\mathfrak{h})$ is a versal $W(E_7)$ variety, and in fact a versal $G$-variety for all $G\subset W(E_7)$. By Lemma \ref{lemma:dom}, all the varieties in \eqref{bitan} dominated by $\Ab(\mathfrak{h})$ are also versal $G$-varieties for all $G\subset W(E_7)$ which act faithfully on them. Since the action of $W(E_7)$ on all but $\Ab(\mathfrak{h})$ factors through the projection $W(E_7)\cong \Z/2\Z\times W(E_7)^+\onto W(E_7)^+$ (cf. \cite[Remark 7.2]{DuRe2}), we conclude the result.

    It remains to construct the diagram \eqref{bitan}. The rational map
    \begin{equation*}
        \Cc_7(\Pb^2)\dr\Hc_{4,2}(28)
    \end{equation*}
    was constructed above as \eqref{2DPM}. The map
    \begin{equation*}
        \Ab(\mathfrak{h})\dr\Pb(\mathfrak{h})
    \end{equation*}
    is just the projectivization, and is thus manifestly $W(E_7)$-equivariant. It remains to construct the map
    \begin{equation*}
        \Pb(\mathfrak{h})\dr \Cc_7(\Pb^2)
    \end{equation*}
    We again follow \cite[Chapter IX]{DoOr}. We begin by identifying $\Pb(\mathfrak{h})$ with the set of ordered points $\{x_1,\ldots,x_7\}$ in the non-singular locus of a fixed cuspidal cubic, up to projective equivalence (cf. Pinkham \cite{Pi}). Since there are 21 cuspidal cubics through a general collection of 7 points in $\Pb^2$, forgetting the cubic gives the above 21-sheeted map. This concludes the construction of \eqref{bitan} and the proof.
\end{proof}

\begin{proof}[Proof of Theorem \ref{theorem:quartics-to-cubics}]
By Proposition \ref{prop:bitangentsversal}, the variety $\Hc_{4,2}(28)$ is versal for any $G\subset W(E_7)^+$. By Proposition \ref{prop:linesversal}, the variety $\Hc_{3,3}(27)$ is versal for any $G\subset W(E_6)\subset W(E_7)^+$. Proposition \ref{prop:whyversal} therefore implies that for any $G\subset W(E_6)$
\begin{equation*}
    \RD(G)=\RD(\Hc_{3,3}(27)\to\Hc_{3,3}(27)/G)=\RD(\Hc_{4,2}(28)\to\Hc_{4,2}(28)/G)
\end{equation*}
and that for any subgroup $G\subset W(E_7)^+$ not contained in $W(E_6)$,
\begin{equation*}
    \RD(G)=\RD(\Hc_{4,2}(28)\to\Hc_{4,2}(28)/G).
\end{equation*}
The special cases above now follow from the discussions of the quotients of $\Hc_{3,3}(27)$ and $\Hc_{4,2}(28)$ of classical interest in Sections \ref{subsection:h33covers} and \ref{subsection:h42covers}.

The bound
\[\RD(W(D_5))=\RD(\Hc_{3,3}(27)\to\Hc_{3,3}(1))=\RD(\Hc_{4,2}(28)\to\Hc_{4,2}(2))=1\]
now follows alternately from Theorem \ref{thm:JH}, Proposition \ref{prop:27from1}, or Proposition \ref{prop:28from2}. The bound
\[\RD(W(E_6))=\RD(\Hc_{3,3}(27)\to\Hc_{3,3})=\RD(\Hc_{4,2}(28)\to\Hc_{4,2}(1))\le 3\]
follows from Theorem \ref{thm:RD:lines}. The bounds
\begin{align*}
    \RD(S_7)&=\RD(\Hc_{4,2}(28)\to\Hc_{4,2}(\Ac))\le 3,\\
    \RD(S_8)&=\RD(\Hc_{4,2}(28)\to\Hc_{4,2}(\theta^{\rm ev}))\le 4
\end{align*}
follow from Corollary \ref{cor:Hfaithandn1} \ref{corollary:Hilbert:for:faithful}, and the Bring-Hamilton bounds $\RD(\tPc_7\to\Pc_7)\le 3$ and $\RD(\tPc_8\to\Pc_8)\le 4$.
\end{proof}


We now use a classical construction to give a more explicit proof of the first equality of 
Theorem~\ref{theorem:quartics-to-cubics}.

\medskip
\para{The classical construction} Let $S$ be a smooth cubic surface containing lines $L_1,\ldots ,L_{27}$.
A choice of a point $p\in S-\cup_{i=1}^{27}L_i$ determines via projection a morphism
\[\pi_p:\Bl_p(S)\to\Pb^2\]
from the blowup $\Bl_p(S)$ to the plane $\Pb^2$.  This setup has the following remarkable properties:
\begin{enumerate}
\item $\pi_p$ is a $2$-sheeted branched cover, branched over a smooth quartic curve $C_p\subset\Pb^2$.
\item The $27$ images $\pi_p(L_i), 1\leq i\leq 27$ are $27$ of the $28$ bitangents of $C_p$, with
the 28th bitangent to $C_p$ being the image under $\pi_p$ of the exceptional divisor in $\Bl_p(S)$.
\item For every smooth quartic curve $C$ in $\Pb^2$ there exists $S$ and $p\in S$ as above so that
$C$ is the branch locus of $\pi_p$, as above.
\end{enumerate}

See Figure \ref{figure:cubic:to:quartic}.

\para{Modular interpretation} We can interpret this classical construction in terms of Del Pezzo surfaces of degree 2 and 3, and thus of maps of moduli spaces and their covers.

Consider the universal family
\[
\xymatrix{
    S\ar[r] & \Uc_{3,3}\ar[d]^\pi \\
    &\Hc_{3,3}
}\]
of smooth cubic surfaces.  Note that $\Uc_{3,3}$ can also be thought of as the moduli space of pairs
$\{(S,p):S\in \Hc_{3,3}, p\in S\}$ and the projection $\pi(S,p):=S$.

We now give a second presentation of $\Uc_{3,3}$. Recall that
\[\Hc_{3,3}(27)\cong\Hc_{3,3}^{\rm skew}(6)\cong\Cc_6'(\Pb^2)\]
Adding the data of a point on a cubic, we get birational maps
\begin{equation*}
    \Cc_7(\Pb^2)\dr^\simeq\Uc_{3,3}^{\rm skew}(6)\cong\Uc_{3,3}(27)
\end{equation*}
where $\Uc_{3,3}^{\rm skew}(6)$ (resp. $\Uc_{3,3}(27)$) denotes the space of cubic surfaces equipped with an ordered set of 6 skew lines (resp. an ordered set of 27 lines) and a point on the surface. These isomorphisms are equivariant with respect to the $W(E_6)\subset W(E_7)$ action on $\Cc_7(\Pb^2)$ and the $W(E_6)$ actions on $\Uc_{3,3}^{\rm skew}(6)$ (resp. $\Uc_{3,3}(27)$). In particular there is an open embedding
\begin{equation*}
    \Cc_7(\Pb^2)/W(E_6)\subseteq \Uc_{3,3}
\end{equation*}
onto the cubics equipped with a point not lying on any of the 27 lines.

On the other hand, as discussed above, we have a generically 2-to-1 $W(E_7)$-equivariant dominant map
\begin{equation*}
    \Cc_7(\Pb^2)\dr\Hc_{4,2}(28).
\end{equation*}
Therefore, for any $G\subset W(E_6)$, we obtain a pullback diagram in which the horizontal maps are generically 2-to-1 rational covers
\begin{equation*}
    \xymatrix{
        \Uc_{3,3}(27)\ar@{-->}[r] \ar[d] & \Hc_{4,2}(28)\ar[d]\\
        \Uc_{3,3}(27)/G \ar@{-->}[r] & \Hc_{4,2}(28)/G
    }
\end{equation*}
This diagram shows that, at the cost of adjoining a square root, any explicit solution for $\Hc_{4,2}(28)\to\Hc_{4,2}(28)/G$ determines one for $\Uc_{3,3}(27)\to\Uc_{3,3}(27)/G$, and vice versa.

It remains to relate this to solutions of
\begin{equation*}
    \Hc_{3,3}(27)\to\Hc_{3,3}(27)/G.
\end{equation*}
One direction is trivial: because we have a pullback diagram in which all maps are dominant
\begin{equation*}
    \xymatrix{
        \Uc_{3,3}(27)\ar@{-->}[r] \ar[d] & \Hc_{3,3}(27)\ar[d]\\
        \Uc_{3,3}(27)/G\ar@{-->}[r] & \Hc_{3,3}(27)/G
    }
\end{equation*}
any solution to $\Hc_{3,3}(27)\to\Hc_{3,3}(27)/G$ immediately pulls back to give one for $\Uc_{3,3}(27)\to\Uc_{3,3}(27)/G$. For the other direction, given an explicit tower solving $\Uc_{3,3}(27)\to\Uc_{3,3}(27)/G$
\begin{equation}\label{U33solve}
    \xymatrix{
    & & & \Uc_{3,3}(27)\ar[d] \\
    X_r \ar@{-->}[r] \ar@{-->}[urrr] & \cdots \ar@{-->}[r] & X_1 \ar@{-->}[r] &\Uc_{3,3}(27)/G
    }
\end{equation}
Let $Z\subset\Uc_{3,3}(27)/G$ be the closure of the complement of the image of $X_r$ in $\Uc_{3,3}(27)/G$. Because $X_r\dr \Uc_{3,3}(27)/G$ is dominant, $Z$ is a proper subvariety.

Fix a line $L\subset\Pb^3$ and let $U\subset \Hc_{3,3}(27)/G$ be the Zariski open consisting of cubic surfaces which intersect $L$ transversely. Define
\begin{equation*}
    \widetilde{U}_L:=\{(\tilde{S},p):\tilde{S}\in U\subset \Hc_{3,3}(27)/G, p\in S\cap L\}
\end{equation*}
By Bezout's Theorem, the projection
\begin{equation*}
    \widetilde{U}_L\to U
\end{equation*}
is a 3-to-1 dominant map. Because $Z\subset\Uc_{3,3}(27)/G$ is a proper closed subvariety, for a generic choice of $L\subset\Pb^3$, the embedding
\begin{equation*}
    \widetilde{U}_L\subset \Uc_{3,3}(27)/G
\end{equation*}
is not contained in $Z$. We can therefore pull back the solution \eqref{U33solve} along this embedding to get a tower solving
\begin{equation*}
    \widetilde{U}_L\times_{\Hc_{3,3}(27)/G}\Hc_{3,3}(27)\to\widetilde{U}_L
\end{equation*}
We conclude from Lemma \ref{lemma:max} and Corollary \ref{cor:Hfaithandn1} \ref{corollary:nto1} that
\begin{align*}
    \RD(\Hc_{3,3}(27)\to\Hc_{3,3}(27)/G)&\le\max\{\RD( \widetilde{U}_L\times_{\Hc_{3,3}(27)/G}\Hc_{3,3}(27)\to\widetilde{U}_L),\RD(\widetilde{U}_L\to U)\}\\
    &\le\max\{\RD(\Uc_{3,3}(27)\to\Uc_{3,3}(27)/G),1\}\\
    &=\RD(\Uc_{3,3}(27)\to\Uc_{3,3}(27)/G).
\end{align*}
\qed

\begin{remark}
    The construction above using Bezout's theorem suggests a general method. We develop this further in Section \ref{subsection:findpoint} below.
\end{remark}

The proof of Theorem \ref{theorem:quartics-to-cubics} also implies the following.
\begin{corollary}[{\bf RD for Double-Sixes and Steiner Complexes}]
    The resolvent degree of finding an ordered sixer given a double-six equals the resolvent degree of finding an Aronhold basis given a Steiner complex equals the resolvent degree of $S_6$, i.e.
    \[\RD(S_6)=\RD(\Hc_{3,3}^{\rm skew}(6)\to\Hc_{3,3}(6,6))=\RD(\Hc_{4,2}(\tAc)\to\Hc_{4,2}(\St)).\]
\end{corollary}
\begin{proof}
    By Theorem \ref{theorem:quartics-to-cubics},
    \begin{align*}
        \RD(S_2\times S_6)&=\RD(\Hc_{3,3}^{\rm skew}(6)\to\Hc_{3,3}(6,6))\intertext{and, because $\Hc_{4,2}(\tAc)\cong\Hc_{4,2}(28)$ as $W(E_7)^+$-varieties, Theorem \ref{theorem:quartics-to-cubics} also gives}
        \RD(W(D_6))&=\RD(\Hc_{4,2}(\tAc)\to\Hc_{4,2}(\St)).
    \end{align*}
    Because $W(D_6)=(\Z/2\Z)^{\times 5}\rtimes S_6$, Theorem \ref{thm:JH} gives
    \begin{equation*}
        \RD(W(D_6))=\max\{1,\RD(S_6)\}=\RD(S_2\times S_6)=\RD(S_6).
    \end{equation*}
\end{proof}

\section{The resolvent degree of some enumerative problems}
\label{section:enumerative}

Consider an enumerative problem $\widetilde{\Mc}\dr\Mc$ as in the introduction.  As mentioned
there, a typical first goal is to prove that this is a branched cover.  One then tries to find its degree.
The third step is to compute the Galois group of (the normal closure of) the covering.  Computing 
$\RD(\widetilde{\Mc}\dr\Mc)$ can be interpreted as computing the number of parameters needed to specify a point in $\widetilde{\Mc}$ given a point of $\Mc$.  This seems to us like a fundamental problem.  We worked through the explicit examples of lines on a smooth cubic surface and bitangents on a smooth quartic in Sections~\ref{section:cubic:surfaces} and \ref{section:bitangents}. In this section we present a few more examples.

\subsection{Tangency problems for plane curves}

\para{Steiner's 5 conics problem}
A classical problem of Steiner asks how many conics in $\Pb^2$ are tangent to $5$ given conics.    After many incorrect answers and a long, rich history, the problem was answered around 40 years ago; see, e.g.\ \cite{EH} and the references contained therein. The answer is $3264$.  But how to find these conics given the original $5$, given by the coefficients of their defining equations?

Harris proves in \cite[IV]{Har} that this problem is not solvable by radicals, as follows.  Let $W\cong\Pb^5$ denote the linear system of conics in $\Pb^2$, and let $W_0$ denote the Zariski open subset of smooth (i.e. non-degenerate) conics.    Let
\[Y:=\{(C_1,\ldots ,C_5,C)\in W^5\times W_0: \ \text{$C$ is tangent to each $C_i$}\}.\]
Consider the map $\pi:Y\to W^5$ be $\pi(C_1,\ldots ,C_5,C):=(C_1,\ldots ,C_5)$.  Then $\pi$ is a $3264$-sheeted branched cover.  Harris (see \S IV of \cite{Har}) computes the monodromy group of this cover to be the full symmetric group $S_{3264}$.  As this group is not solvable, Harris deduces that there is no formula in radicals for the coefficients of $C$ in terms of the coefficients of the $C_i$.

\begin{problem}[{\bf Refinements of Steiner's problem}]
Determine the monodromy of the natural branched covers of $W^5$ lying between $Y$ and $W^5$. Determine which if any are solvable by radicals. For these, determine explicit formulas.
\end{problem}

\begin{problem}[{\bf Resolvent degree of the $5$ conics problem}]
Compute $\RD(Y\to W^5)$.
\end{problem}

There are many generalizations of Steiner's Problem, for many of which the associated monodromy group has been computed; see, e.g.\ \cite{EH,HS}.  It would be interesting to work out bounds on the resolvent degree for these problems.

\para{Curves through specified points}
There are many more such enumerative problems.   For example, we have the following.
Let ${\cal P}_{d}\subset (\Pb^2)^{3d-1}/S_{3d-1}$ be the parameter space of $(3d-1)$-tuples of distinct points in $\Pb^2$ in general position.    A dimension count gives that the number $n_d$ of degree $d$ rational curves that pass through $3d-1$ such points in $\Pb^2$ is finite.   It was known classically that $n_2=1, n_3=12$ and $n_4=620$.    In the early 1990's the following recursive formula
for $n_d$ was given by Kontsevich-Manin and Ruan-Tian (see, e.g.\ \cite{EH} and the references contained therein):

\[n_d=\sum_{d_1+d_2=d, d_1,d_2>0}n_{d_1}n_{d_2}\left(d_1^2d_2^2\binom{3d-4}{3d_1-2}-d_1^3d_2\binom{3d-4}{3d_1-1}\right).\]


Let $X_d:=\PGL_2\backslash{\rm Rat}_d(\Pb^1,\Pb^2)/\PGL_3$ denote the moduli space of degree $d$ rational curves.  Let
\[Y_d:=\{(p_1,\ldots, p_{3d-1}, C_1,\ldots ,C_{n_d}): p_j\in C_k \ \forall j,k\}\subset  {\cal P}_d\times X_d^{n_d}.\]

Denote by $\pi_d:Y_d\to  {\cal P}_d$ the natural projection.  Then $\pi_d$ is an $n_d$-sheeted branched cover.

\begin{problem}
Compute the monodromy of $\pi_d$, as well as of the intermediate covers.  Compute $\RD(\pi_d)$.
\end{problem}

Among many other variations, we mention the following.

\begin{problem}
All general degree $n$ curves through $\frac{1}{2}n(n+3)-1$ fixed points pass through $\frac{1}{2}(n-1)(n-2)$ other fixed points (see,e.g., p.191 of \cite{Hil}). Compute $\RD$ for the problem of finding one of the $\frac{1}{2}(n-1)(n-2)$ other points, as well as its monodromy.
\end{problem}

\subsection{Finding a point on a projective subvariety}\label{subsection:findpoint}
In relating different problems about varieties in projective space, it will sometimes be useful to pick a basepoint on a variety in a way that varies algebraically over a parameter space.  The following proposition, which we isolate because it might be useful in other contexts, states that to compute $\RD$ for any algebraic problem for degree $d$ varieties of a fixed dimension in $\Pb^n$, one can add the data of a basepoint at the cost of finding a root of a degree $d$ polynomial.

\begin{proposition}[{\bf Finding a point on a subvariety of $\Pb^n$}]
\label{proposition:find-a-point}
    Let $X$ be any variety over $k$, and let
    \begin{equation*}
        \xymatrix{
            S \ar[dr] \ar[r] & X\times\Pb^n \ar[d] \\
            & X
        }
    \end{equation*}
    be any family of $r$-dimensional, degree $d$ varieties in $\Pb^n$ such that $S\to X$ is a dominant map. Let
        \begin{equation}
            \label{eq:pullback-m-to-u}
            \xymatrix{
                Y \ar@{-->}[r]\ar@{-->}[d] & \tX \ar@{-->}[d] \\
                S \ar@{-->}[r] & X
            }
        \end{equation}
    be any pullback diagram with vertical maps being rational covers.  Then
    \[\RD(Y\dr S)\leq \RD(\tX\dr X)\leq \max\{\RD(Y\dr S), \RD(S_d)\}.\]
\end{proposition}

\begin{proof}
The first inequality follows from Lemma \ref{lemma:first}.  We now prove the second inequality.  Fix an $n-r$-dimensional linear subspace $L\subset\Pb^n$. Let $U\subset X$ be the Zariski open consisting of all $x\in X$ such that the variety $S_x$ intersects $L$ transversely. Define
\begin{equation*}
    U_1:=(U\times L)\cap S.
\end{equation*}
By Bezout's theorem, the map $U_1\to U$ given by projection is a generically $d$-to-$1$ rational cover.  Therefore, by Lemma \ref{lemma:universal},
\begin{equation*}
    \RD(U_1\to U)\le\RD(\tPc_d\to\Pc_d)=\RD(S_d).
\end{equation*}
By construction, we have a commuting triangle
\begin{equation*}
    \xymatrix{
        & S \ar[d] \\
        U_1 \ar[r] \ar[ur] & X
    }
\end{equation*}
Form the pullback
\begin{equation*}
    \xymatrix{
        U_1\times_S Y \ar[r] \ar@{-->}[d] & Y\ar@{-->}[d]\\
        U_1\ar[r] & S
    }
\end{equation*}
By construction,
\begin{equation*}
    U_1\times_{S}Y\dr U_1\to X
\end{equation*}
is a tower solving $\tX\dr X$. The definition of resolvent degree and Lemmas \ref{lemma:first} and \ref{lemma:max} imply that
\begin{align*}
    \RD(\tX\dr X)&\le\RD(U_1\times_{S}Y\dr X)\\
    &\le \max\{\RD(U_1\times_{S}Y\dr U_1),\RD(U_1\to X)\}\\
    &\le \max\{\RD(Y\dr S),\RD(\tPc_d\to\Pc_d)\}
\end{align*}
as claimed.
\end{proof}

\subsection{Resolvent degree and Bezout's theorem}\label{subsection:Bezout}
Recall that $\Hc_{r,2}$ denotes the moduli space of smooth degree $r$ curves in $\Pb^2$.  Fix $r,s\geq 1$.  Bezout's Theorem gives that for each pair of curves $C,C'\subset\Pb^2$ of degrees $r$ and $s$, the intersection $C\cap C'$ has $rs$ points, where each $p\in C\cap C'$ is counted with the intersection multiplicity $I_p(C,C')$.  Let
\begin{equation*}
    \Hc_{(r,s),2}:=\left((\Pb^{\binom{r+2}{2}-1}-\Sigma_r)\times(\Pb^{\binom{s+2}{2}-1}-\Sigma_s)\right)/\PGL_3
\end{equation*}
denote the moduli of pairs of smooth plane curves $(C,C')$ with $\deg(C)=r$, $\deg(C')=s$ (where $\Sigma_r$, and $\Sigma_s$ denote the loci of singular curves). Let $U_{r,s}$ denote the Zariski open
\[U_{r,s}:=\{(C,C'): I_p(C,C')=1 \ \forall p\in C\cap C'\}\subset \Hc_{(r,s),2}\]
and consider the covering
\[
\begin{array}{ll}
\widetilde{U}_{r,s}&:= \{(C,C',p): p\in C\cap C'\}\subset U_{r,s}\times\Pb^2\\
\pi \big\downarrow&\\
U_{r,s}
\end{array}
\]
given by $\pi(C,C',p):=(C,C')$.  Note that $\pi^{-1}(C,C')=C\cap C'\subset\Pb^2$.  Bezout's Theorem
implies that $\pi:\widetilde{U}_{r,s}\to U_{r,s}$ is an $rs$-sheeted cover.   It is known that the monodromy of this cover is the full symmetric group $S_{rs}$; see, for example \cite[Corollary 1]{HS}.  Thus there is a formula in radicals for the intersection of two curves of degrees $r,s\leq 2$, but there is no such formula when $rs>4$.  It is natural to ask for the minimal number $\RD(\widetilde{U}_{r,s}\to U_{r,s})$ of parameters for any formula for an intersection point of two smooth curves, given the coefficients defining those curves. By the computation of the monodromy, we have
\begin{equation}\label{Bezout}
    \RD(\widetilde{U}_{r,s}\to U_{r,s})\le \RD(S_{rs}).
\end{equation}
\begin{problem}
    For which $r$ and $s$ does equality hold in \eqref{Bezout}?
\end{problem}

\subsection{Finding flexpoints}

Let $C$ be a degree $d\geq 2$ plane curve.  For a generic point $p\in C$, the tangent line $\ell_p$ to $C$ at $p$ intersects $C$ with multiplicity $m_p(C\cdot \ell_p)=2$.  Recall that $p$ is a {\em flex point} of $C$ if $m_p(C\cdot \ell_p)\geq 3$; it is a {\em simple flex} if $m_p(C\cdot \ell_p)=3$.   It is known that any degree $d$ curve $C$ has $3d(d-2)$ flex points, counted with multiplicity.  Recall that $\Hc_{d,2}$ denotes the moduli space of smooth degree $d$ curves on $\Pb^2$.  Let $\Hc_{d,2}({\rm flex})\subset \Hc_{d,2}\times \Pb^2$ be the moduli space of pairs $(C,p)$ where $p\in C$ is a flex point.
The projection map $\Hc_{d,2}({\rm flex})\to \Hc_{d,2}$ given by $(C,p)\mapsto p$ is a $3d(d-2)$-sheeted covering when restricted to the Zariski open in $\Hc_{d,2}$ consisting of those degree $d$ curves $C$ all of whose flex points are simple.

The monodromy of $\Hc_{3,2}({\rm flex})\to \Hc_{3,2}$ is solvable (see \cite[II.2]{Har}), so that
\[\RD(\Hc_{3,2}({\rm flex})\to \Hc_{3,2})=1.\]  In contrast, Harris proves in II.3 of \cite{Har} that for $d\geq 4$, the monodromy of $\Hc_{d,2}({\rm flex})\to \Hc_{d,2}$ is $S_{3d(d-2)}$, which is not solvable if $d\geq 4$.  While Harris concludes from this that there is no formula in radicals for the flex points of a general degree $d\geq 4$ smooth plane curve, the basic question remains as to how complicated any formula not-in-radicals actually is.

\begin{problem}[{\bf Finding flexpoints}]\label{problem:flex}
Compute the resolvent degree for the problem of finding a flexpoint on a smooth degree $d\geq 4$ plane curve; that is, compute $\RD(\Hc_{d,2}({\rm flex})\to \Hc_{d,2})$.
\end{problem}

It is a classical fact that for a degree $d$ curve $C$, the flexpoints of $C$ are precisely the intersection points of $C$ with its associated Hessian curve $H_C$, which has degree $3(d-2)$. However, Problem \ref{problem:flex} is quite different than the situation considered in \S\ref{subsection:Bezout}. Indeed, while the map
\begin{align*}
    \Hc_{d,2}&\to^H U_{d,3(d-2)}\\
    C&\mapsto (C,H_C)
\end{align*}
fits into a pullback square
\begin{equation*}
    \xymatrix{
        \Hc_{d,2}({\rm flex}) \ar[r] \ar[d] & \widetilde{U}_{d,3(d-2)} \ar[d] \\
        \Hc_{d,2} \ar[r]^H & U_{d,3(d-2)}
    },
\end{equation*}
the codimension of $H(\Hc_{d,2})\subset U_{d,3(d-2)}$ is always positive and grows quadratically in $d$. 

\section{The resolvent degree of the roots of a polynomial}
\label{section:brauer}

While the problem of simplifying the formulas needed to solve a general polynomial has been central to the mathematical tradition since the Babylonians, the study of the resolvent degree of polynomials essentially originates with work of Tschirnhaus \cite{Ts} in the $17$th century. Tschirnhaus introduced the {\em Tschirnhaus transformation}, which remains essentially the only method for providing general upper bounds on $\RD(\tPc_n\to\Pc_n)$. We review Tschirnhaus transformations from a geometric standpoint below, and then we treat several of the classical upper bounds from this perspective.

\subsection{Tschirnhaus tranformations and classical solutions of polynomials}

\paragraph*{Elementary perspective.}
Consider the general degree $n$ polynomial
\begin{equation*}
    p(x):=x^n+a_1x^{n-1}+\cdots+a_n=0,
\end{equation*}
with roots $x_1,\ldots,x_n$. A {\em Tschirnhaus transformation} $T(b_0,\ldots,b_{n-1})$ (for some $b_0,\ldots,b_{n-1}$) sends the roots $x_i$ to
\begin{equation*}
    T(b_0,\ldots,b_{n-1})(x_i):=b_0x_i^{n-1}+b_1x_i^{n-2}+\cdots+b_{n-1}.
\end{equation*}
The Tschirnhaus transformation of the polynomial $p(x)$ is defined by
\begin{equation*}
    T(b_0,\ldots,b_{n-1})(p)(x):=\prod_{i} (x-T(b_0,\ldots,b_{n-1})(x_i)).
\end{equation*}
Because the assignment $x_i\mapsto T(b_0,\ldots,b_{n-1})(x_i)$ is symmetric in the roots, the coefficients of $T(b_0,\ldots,b_{n-1})(p)$ are polynomials in the $a_i$ and the $b_j$. Accordingly, by solving polynomials in the $b_j$ whose coefficients are polynomials in the $a_i$, we can find special Tschirnhaus transformations which convert our original polynomial $p(x)$ into a polynomial whose coefficients satisfy special conditions, e.g. some collection of the coefficients are zero.

Note that, given the roots of $T(b_0,\ldots,b_{n-1})(p)$, we can recover the roots of $p$ by a rational transformation.  See \cite[Lemma 4.2.1]{Hu} for a clear treatment.

\paragraph*{As covariants.}
Tschirnhaus transformations can also be defined as $S_n$-equivariant maps
\begin{equation*}
    T\colon \Ab^n\to\Ab^n
\end{equation*}
In the setting above, we have an auxiliary affine space parametrizing Tschirnhaus transformations
\begin{equation*}
    \Ab^n_T:=\{(b_0,\ldots,b_{n-1})\}
\end{equation*}
and a map
\begin{equation*}
    \Ab^n_T\to\Alg_{S_n}(\Ab^n,\Ab^n)
\end{equation*}
from the affine space parametrizing Tschirnhaus transformations to the space of maps of $S_n$-varieties $\Ab^n\to\Ab^n$.

\paragraph*{Geometric perspective.}
Equivalently, we have an $S_n$-equivariant ``evaluation'' map
\begin{equation*}
    \Ab^n\times\Ab^n_T\to^{\varepsilon} \Ab^n
\end{equation*}
where $S_n$ acts trivially on the $\Ab^n_T$ factor, and via the permutation representation on each $\Ab^n$. Passing to the quotients, we obtain a commuting square
\begin{equation*}
    \xymatrix{
        \Ab^n\times\Ab^n_T \ar[r]^{\varepsilon} \ar[d] & \Ab^n \ar[d] \\
        \Pc_n\times\Ab^n_T \ar[r]^{\bar{\varepsilon}} & \Pc_n
    }
\end{equation*}
To bound the resolvent degree of $\tPc_n\to\Pc_n$ via a Tschirnhaus transformation, one now specifies
\begin{enumerate}
    \item a Zariski closed $S_n$-invariant subvariety $V\subset\Ab^n$, and
    \item a rational cover $U\dr\Pc_n$ along with a section
        \begin{equation*}
            \xymatrix{
                & \varepsilon^{-1}(V) \ar[d] \\
                U \ar@{-->}[ur] \ar@{-->}[r] & \Pc_n
            }
        \end{equation*}
\end{enumerate}
Given these data, one obtains
\begin{equation*}
    \RD(\tPc_n\to\Pc_n)\le\max\{\RD(U\dr\Pc_n),\dim(V)\}.
\end{equation*}

\begin{remark}
    Standard examples of $V$ are given by
    \begin{equation*}
        V_{1\cdots i}:=\bigcap_{j=1}^i\{\sigma_j=0\}\subset\Ab^n,
    \end{equation*}
    where the $\sigma_j$ are the elementary symmetric functions. Finding $U\dr\Pc_n$ with a map $U\dr\varepsilon^{-1}(V_{1\cdots i})$ over $\Pc_n$ is just to find a Tschirnhaus transformation which sets the first $i$ coefficients of the general degree $n$ polynomial to 0.
\end{remark}

We now illustrate this procedure in several classical examples.

\subsection{The Bring-Hamilton $4$-parameter reduction}
In 1786 Bring \cite{Bri} proved the following, which was independently discovered by Hamilton \cite{Ham}.
\begin{theorem}[{\bf Bring-Hamilton \boldmath$4$-parameter reduction}]
     For any $n\ge 5$ :
     \[\RD(\tPc_n\to\Pc_n)\le n-4.\]
\end{theorem}

From the above perspective, Bring's proof is as follows.
\begin{proof}
    First, restrict to the space of quartic Tschirnhaus transformations, i.e.
    \begin{equation*}
        T(b_0,\ldots,b_4)(x_i)=b_0x_i^4+\cdots+b_4.
    \end{equation*}
    Next, observe that the fiberwise projectivization of $\varepsilon^{-1}(V_1)\to\Pc_n$ is a trivial $\Pb^3$ bundle, since the condition that the first coefficient vanish is linear in the $b_j$, and this 3-plane bundle admits a rational section. Therefore, the fiberwise projectivization of $\varepsilon^{-1}(V_{12})\to\Pc_n$ is a bundle of quadric surfaces in $\Pb^3$. Denote by $\Hc_{2,3}$ the moduli of quadric surfaces and let $\Hc_{2,3}(L)\subset \Hc_{2,3}\times\Gr(2,4)$ be the moduli of quadric surfaces equipped with a line, so that the two connected components of $\Hc_{2,3}(L)$ (corresponding to the two rulings of the quadric) each give a $\Pb^1$-bundle over $\Hc_{2,3}$.
    We have a map
    \begin{align*}
        \Pc_n&\to \Hc_{2,3}\\
        p&\mapsto \varepsilon^{-1}(V_{12})|_p
    \end{align*}
    By the classical theory of quadratic forms (for a detailed contemporary treatment, see e.g. \cite[Lemma 5.2]{W}), after passing to a branched cover $U_1\to \Pc_n$ of degree $2^4$ (i.e. by adjoining 4 square roots of polynomials in the coefficients), we can diagonalize the associated quadratic form, i.e.
    \[
        V_{12}|_U \cong V(\sum_{i=0}^{3} L_i^2)
    \]
    for rational hyperplanes $L_i\subset \Pb^3_U$. Then $\{L_0+\sqrt{-1} L_1=0,L_2+\sqrt{-1} L_3=0\}$ defines a line on the quadric. In other words, there exists a lift of the map $U_1\to\Hc_{2,3}$
    \begin{equation*}
        \xymatrix{
            & \Hc_{2,3}(L) \ar[d] \\
            U_1 \ar[r] \ar[ur]^{L} & \Hc_{2,3}
        }
    \end{equation*}
    By intersecting the family of cubics $\varepsilon^{-1}(V_3)$ with this line, we obtain a map
    \begin{align*}
        U_1&\to\Pc_3\\
        u&\mapsto L(u)\cap \varepsilon^{-1}(V_{3})|_u
    \end{align*}
    Forming the pullback
    \begin{equation*}
        \xymatrix{
            U_2 \ar[r] \ar[d] & \tPc_3 \ar[d] \\
            U_1 \ar[r] & \Pc_3
        }
    \end{equation*}
    we obtain a branched cover $U_2\to\Pc_n$ and a section
    \begin{equation*}
        \xymatrix{
            & \varepsilon^{-1}(V_{123}) \ar[d]\\
            U_2 \ar[r] \ar[ur] & \Pc_n
        }
    \end{equation*}
    By construction,
    \begin{align*}
        \RD(U_2\to\Pc_n)&=\max\{\RD(U_2\to U_1),\RD(U_1\to\Pc_n)\}\\
        &\le\max\{\RD(\tPc_3\to\Pc_3),1\}=1.
    \end{align*}
    Therefore
    \begin{align*}
        \RD(\tPc_n\to\Pc_n)\le\max\{\RD(\sqrt{-}),\RD(\tPc_3\to\Pc_3),\RD(V_{123}\to\Ab^{n-3})\}
    \end{align*}
    where the final space $\Ab^{n-3}$ is the moduli space of all monic degree $n$ polynomials of the form
    \begin{equation*}
        x^n+a_4x^{n-4}+\cdots+a_{n-1}x+a_n=0.
    \end{equation*}
    Restricting to locus $U\subset\Ab^{n-3}$ where $a_{n-1}\neq 0\neq a_n$, we can define a linear Tschirnhaus transformation
    \begin{equation*}
        T(x_i):=\frac{a_{n-1}}{a_n}x_i
    \end{equation*}
    to set the last two coefficients to be equal. This defines a pullback diagram
    \begin{equation*}
        \xymatrix{
            V_{123}|_U \ar[r]^T \ar[d] & V_{123,(n-1)=n}\ar[d] \\
            U\ar[r]^{\bar{T}} & \Ab^{n-4}
        }
    \end{equation*}
    where $\Ab^{n-4}$ denotes the space of all polynomials of the form
    \begin{equation*}
        x^n+b_4x^{n-4}+\cdots+b_{n-1}x+b_{n-1}=0.
    \end{equation*}
    We conclude that, for $n\ge 5$,
    \begin{align*}
        \RD(\tPc_n\to\Pc_n)\le\max\{\RD(\sqrt{-}),\RD(\tPc_3\to\Pc_3),\RD(V_{123,(n-1)=n}\to\Ab^{n-4})\}\le n-4
    \end{align*}
    as desired.
\end{proof}

As a consequence of the Bring/Hamilton theorem, we obtain the upper bounds in Hilbert's Sextic and Octic Conjectures Hilbert's and 13th Problem.

\begin{corollary}
    $\RD(\tPc_6\to\Pc_6)=\RD(S_6)\le 2$, $\RD(\tPc_7\to\Pc_7)=\RD(S_7)\le 3$, and $\RD(\tPc_8\to\Pc_8)\le 4$.
\end{corollary}

\subsection{Brauer's bounds}
Hamilton \cite{Ham} was the first to show that
\begin{equation*}
    \lim_{n\to\infty} n-\RD(\tPc_n\to\Pc_n)=\infty.
\end{equation*}
More precisely, he showed the existence of a function $H\colon\Nb\to\Nb$, such that for $n\ge H(r)$, $n-\RD(\tPc_n\to\Pc_n)\ge r$, and he computed the initial values of $H$:
\begin{equation*}
    \begin{array}{c|c|c|c|c|c|c}
        r & 4 & 5 & 6 & 7 & 8 & 9 \\ \hline
        H(r) & 5 & 11 & 47 & 923 & 409,619 & 83,763,206,255
    \end{array}
\end{equation*}
By the mid-20th century, Hamilton's work appears to have been forgotten. Segre \cite{Seg1}, building on Hilbert's work on the degree 9 equation, proved that $\RD(\tPc_n\to\Pc_n)\le n-6$ for $n\ge 157$.
He further conjectured that
\begin{equation*}
    \lim_{n\to\infty} n-\RD(\tPc_n\to\Pc_n)=\infty;
\end{equation*}
that is, he conjectured precisely what Hamilton had shown over a century earlier.  Shortly after, in 1945, Brauer \cite{Br1} and Segre each reproved this statement, but without giving effective bounds.  Three decades later, Brauer \cite{Br} proved the following theorem, which provides the best general upper bounds on $\RD(\tPc_n\to\Pc_n)$ to date.

\begin{theorem}[{\bf Brauer} \cite{Br}]
\label{theorem:Brauer}
    Let $n>3$.  For any $r\geq 2$
    \[\RD(\tPc_n\to\Pc_n)\le n-r \ \ \text{for all $n\ge (r-1)!+1$}.\]
\end{theorem}

We include a streamlined version of Brauer's proof of Theorem~\ref{theorem:Brauer} for completeness.

\begin{proof}
     We prove this by induction on $r$. The base case $r=1$ follows from the Babylonians:
     $\RD(n)\le n-1$ for all $n\geq 2$, via a linear translation of the roots.

    For the inductive step, consider the full space of Tschirnhaus transformations $\Pb^{n-1}_T$. Observe that
    \begin{equation*}
        \bar{\varepsilon}^{-1}(V_{1\cdots (r-1)})\to\Pc_n
    \end{equation*}
    is a bundle of $(n-r+1)$-dimensional , degree $(r-1)!$ subvarieties of $\Pb^{n-1}_T$. By construction, there is an isomorphism of varieties over $\bar{\varepsilon}^{-1}(V_{1\cdots(r-1)})$ :
    \begin{equation*}
        \tPc_n\times_{\Pc_n}\bar{\varepsilon}^{-1}(V_{1\cdots(r-1)})\cong \bar{\varepsilon}^{-1}(V_{1\cdots(r-1)})\times_{\Ab^{n-(r-1)}}\tPc_n|_{\Ab^{n-r-1}}
    \end{equation*}
    where $\Ab^{n-(r-1)}\subset\Pc_n$ denotes the space of all monic polynomials with the first $(r-1)$ coefficients vanishing. Therefore

     \begin{equation*}
        \RD(\tPc_n\times_{\Pc_n}\bar{\varepsilon}^{-1}(V_{1\cdots(r-1)})\to\bar{\varepsilon}^{-1}(V_{1\cdots(r-1)}))\le \RD(\tPc_n|_{\Ab^{n-(r-1)}}\to\Ab^{n-(r-1)}).
    \end{equation*}

    Proposition \ref{proposition:find-a-point} then implies
    \begin{align*}
        \RD(\tPc_n\to\Pc_n)&\le\max\{\RD(\tPc_n\times_{\Pc_n}\bar{\varepsilon}^{-1}(V_{1\cdots(r-1)})\to\bar{\varepsilon}^{-1}(V_{1\cdots(r-1)})),\RD(\tPc_{(r-1)!}\to\Pc_{(r-1)!})\}\\
        &\le \max\{\RD(\tPc_n|_{\Ab^{n-(r-1)}}\to\Ab^{n-(r-1)}),\RD(\tPc_{(r-1)!}\to\Pc_{(r-1)!})\}.
    \end{align*}
    An analogous linear Tscirnhaus transformation to that in Bring and Hilbert shows
    \begin{equation*}
        \RD(\tPc_n|_{\Ab^{n-(r-1)}}\to\Ab^{n-(r-1)})\le n-r.
    \end{equation*}
    The inductive hypothesis then gives
    \begin{equation*}
        \RD(\tPc_{(r-1)!}\to\Pc_{(r-1)!})\le (r-1)!-(r-1)\le n-r,
    \end{equation*}
    completing the proof of the induction step.
\end{proof}

\begin{remark}
    Note that Brauer's proof does not make use of the Bring/Hamilton idea. Moreover, Hilbert \cite{Hi2} sketched an approach using lines on cubic surfaces to show that $\RD(9)\le 4$.  Brauer needs $n\ge 25$ in order to conclude $\RD(\tPc_n\to\Pc_n)\le n-5$. In \cite{W}, an extension of Hilbert's argument leads to a substantial improvement over Brauer's bounds for general $n$.
\end{remark}

\section{The equivalence of Hilbert's conjectures to classical geometry problems}
\label{section:equivalence}

As with many Hilbert problems, the specific statement of Hilbert's Sextic Conjecture, 13th Problem and Octic Conjecture (see Problem~\ref{hilbert:problems}) turns out to be much broader and more widely connected to other problems than one might at first glance guess.  The goal of this section is to use the theory we have developed so far to prove the equivalence of each of these problems with many other natural problems of both geometric and arithmetic natures.    We give each statement in English form, and name the corresponding problem in terms of moduli spaces when we have already named them explicitly.

We organize things into five groups of examples, according to the group that is acting. The five classes of examples are ordered in complexity via:
\[
\begin{array}{ccc}
&\RD(W(E_6)) &\\
\RD(S_6)\leq &&\leq\RD(W(E_7))\\
&\RD(S_7)\leq \RD(S_8)&
\end{array}
\]

\subsection{\boldmath$S_6$-varieties and Hilbert's Sextic Conjecture}

We start with the Sextic Conjecture.

\begin{theorem}[{\bf \boldmath$\RD$ of \boldmath$S_6$ varieties}]
\label{theorem:S6:varieties}
The following statements are equivalent:

\begin{enumerate}
\item \label{item:HSC} Hilbert's Sextic Conjecture is true: $\RD(\tPc_6\to\Pc_6)=2.$
\item $\RD(S_6)=2$. \label{item:RD:S6}
\item $\RD(V\to V/S_6)=2$ for any faithful, linear $S_6$-variety $V$. \label{item:faithful}
\item $\RD(\Mc_{0,6}\to\Mc_{0,6}/S_6)=2$.\label{item:M06}
\item  $\RD=2$ for the problem of finding a fixed point for the $\Z/3\Z$ action on a genus $4$ curve of the form
$y^3=P(x)$, where $P(x)$ is a square-free polynomial of degree $6$ :\label{item:C36}
\[\RD(\widetilde{\Cc_{3,6}}\to \Cc_{3,6})=2.\]
\item  $\RD=2$ for the problem of finding a fixed point for the hyperelliptic involution on a genus $2$ curve:  \label{item:hypinv2}
\[\RD(\Mc_2(\widetilde{\Delta})\to \Mc_2)=2.\]
\item \label{item:doublesix} $\RD$ of finding the $27$ lines on a cubic, given a double-six: \[\RD(\Hc_{3,3}(27)\to\Hc_{3,3}(6,6))=2.\]
\item \label{item:hexahedral} $\RD$ of finding the $27$ lines on a smooth cubic surface $S$ given the unordered
hexahedral form of $S$: \[\RD(\Hc_{3,3}(27)\to H/S_6)=2.\]
\end{enumerate}

In fact, the resolvent degrees of all of the above problems coincide.

\end{theorem}


\begin{proof} We prove the theorem via chains of equivalences.

\medskip
\noindent
{\bf Equivalence of \ref{item:HSC}, \ref{item:RD:S6}, \ref{item:faithful}, and \ref{item:M06}:} The equivalence of the first four follows from Corollary \ref{cor:Hfaithandn1} \ref{corollary:Hilbert:for:faithful} together with Corollary \ref{cor:Cn}.

\medskip
\noindent
{\bf Equivalence of \ref{item:M06}, \ref{item:C36}:} Consider the moduli space $\Cc_{3,6}$ of isomorphism classes of algebraic curves of the form $y^3=P(x)$ where $P$ has is a square-free polynomial of degree $6$. These are genus $4$ curves equipped with a $\Z/3\Z$ action, the quotient giving a branched cover $\Sigma_4\to\Pb^1$ branched over $6$ points, each of order $3$.  Let $\widetilde{\Cc_{3,6}}$ denote the moduli of curves in $\Cc_{3,6}$ equipped with an ordering of the $\Z/3\Z$-fixed points.  The forgetful map
\begin{equation*}
    \widetilde{\Cc_{3,6}}\to\Cc_{3,6}
\end{equation*}
is a Galois $S_6$-cover. By mapping the fixed points to $\Pb^1$ under the $\Z/3\Z$-quotient, we obtain the commutative diagram
\begin{equation}
\label{eq:picard3}
    \xymatrix{
        \widetilde{\Cc_{3,6}} \ar[r] \ar[d] & \Mc_{0,6}\ar[d]\\
        \Cc_{3,6} \ar[r] & \Mc_{0,6}/S_6
    }
\end{equation}
in which the horizontal arrows are birational, equivariant with respect to the $S_6$ actions, and the bottom row is the quotient of the top row by the $S_6$ action. The stabilizer of a fixed point is $S_5\subset S_6$, and thus $\widetilde{\Cc_{3,6}}\to \Cc_{3,6}$ is the Galois closure of the cover parametrizing curves in $\Cc_{3,6}$ with a single choice of fixed point. Together with Lemma \ref{lemma:galois}, this proves the equivalence of  \ref{item:M06}, and \ref{item:C36}.

\medskip
\noindent
{\bf Equivalence of \ref{item:M06} and \ref{item:hypinv2}:}
The {\em Segre cubic threefold} $X_3$ is the threefold in $\Pb^5$ given by
\[X_3:=\{[x_0:\cdots :x_5]\in\Pb^5: \sum_{i=0}^5x_i=0=\sum_{i=0}^5x_i^3\}.\]
The permutation action of $S_6$ on $\Pb^5$ leaves invariant $X_3$, permuting its $10$ nodes. It's classically known that $X_3\cong \Mc_{0,6}$ as $S_6$-varieties.

Hunt proves in \cite[Theorem 3.3.11]{Hu}  that the dual variety to $X_3$ is the so-called {\em Igusa quartic} ${\mathcal I}_4$, which is the moduli space of $6$ points on a conic in $\Pb^2$. The two varieties $X_3$ and ${\mathcal I}_4$ are $S_6$-equivariantly birational. The Igusa quartic ${\mathcal I}_4$ is the Satake compactification of the moduli space $\Mc_2(\widetilde{\Delta})$ of hyperelliptic curves of genus $2$ with a marking of the $6$ branch points.  The group $S_6$ acts by permuting these marked points.  We thus obtain a commutative diagram in which all horizontal arrows are birational equivalences
\begin{equation}
\label{eq:picard3b}
\xymatrix{
    \Mc_{0,6} \ar[r]^\sim \ar[d] & {\mathcal I}_4 \ar[d] & \ar[l]_\sim \ar[d] \Mc_2(\widetilde{\Delta})\\
    \Mc_{0,6}/S_6 \ar[r]^\sim &{\mathcal I}_4/S_6 & \ar[l]_\sim \Mc_2
}
\end{equation}
Thus each of the rational covers in \eqref{eq:picard3b} have equal resolvent degree.

\medskip
\noindent
{\bf Equivalence of  \ref{item:RD:S6}, \ref{item:doublesix} and \ref{item:hexahedral} : }
As explained in \eqref{eq:doublesix}, the moduli space of pairs $(S,D)$ where $S\in\Hc_{3,3}$ and $D$ is a double-six in $S$ can be identified with $\Hc_{3,3}(27)/S_6$.  Thus the problem of finding all $27$ lines on a smooth cubic surface given a double-six is
$\RD(\Hc_{3,3}(27)\to \Hc_{3,3}(27)/S_6)$.  By Proposition \ref{prop:linesversal}, $\Hc_{3,3}(27)$ is versal for any $G\subset W(E_6)$. Therefore, by Proposition \ref{prop:whyversal},
\[\RD(\Hc_{3,3}(27)\to \Hc_{3,3}(27)/S_6)=\RD(S_6),\]
proving the equivalence of  \ref{item:RD:S6} and \ref{item:doublesix} .

Now recall from \S\ref{eq:hexahedral3}  that the moduli space of unordered hexahedral forms for smooth cubic surfaces fits in to the sequence of branched covers (see  \eqref{eq:hexahedral3}) :

\[\Hc_{3,3}(27)\stackrel{t_1}{\to} H\stackrel{t_2}{\to}H/S_6\stackrel{t_3}{\to}\Hc_{3,3}
\]
where $t_1$ is an unramified $2$-sheeted cover, $t_3$ is an unramified $36$-sheeted cover, and
$t_2$ is a generically $720$-to-$1$ branched cover.  The composite is a Galois branched cover, with deck group $S_2\times S_6\subset W(E_6)$, i.e.
\begin{equation*}
    H/S_6=\Hc_{3,3}/(S_2\times S_6)
\end{equation*}
Proposition \ref{prop:linesversal} therefore implies
\[\RD(\Hc_{3,3}(27)\to H/S_6)=\RD(S_2\times S_6)=\RD(S_6),\]
proving the equivalence of  \ref{item:hexahedral} and \ref{item:RD:S6}.
\end{proof}

\subsection{\boldmath$W(E_6)$-varieties and lines on a smooth cubic surface}

In this subsection we summarize the equality of the resolvent degree of different $W(E_6)$-varieties proven above.

\begin{theorem}[{\bf \boldmath$\RD$ of \boldmath$W(E_6)$ varieties}]
The following are equal:

\begin{enumerate}
\item $\RD(W(E_6))$.\label{item:RD:WE6}
\item $\RD(V\to V/W(E_6))$ for $V$ any faithful representation of $W(E_6)$.\label{item:E6:faithful}
\item \label{item:alllines} $\RD$ of finding all 27 lines on a smooth cubic surface:
    \[\RD(\Hc_{3,3}(27)\to \Hc_{3,3}).\]
\item \label{item:lines} $\RD$ of finding a line on a smooth cubic surface: \[\RD(\Hc_{3,3}(1)\to\Hc_{3,3}).\]
\item \label{item:bitangents4} $\RD$ of finding $28$ bitangents on a smooth plane quartic, given one of them: \[\RD(\Hc_{4,2}(28)\to\Hc_{4,2}(1)).\]
\end{enumerate}
Further, all of the above are at most 3.
\end{theorem}

\begin{proof} We prove the theorem in chains of equivalences.

\medskip
\noindent
{\bf Equivalence of \ref{item:RD:WE6}, \ref{item:E6:faithful}, \ref{item:alllines} and \ref{item:lines}: } This follows from the proof of Theorem \ref{thm:RD:lines}. Moreover, from Theorem \ref{thm:RD:lines}, we obtain the upper bound of 3.

\medskip
\noindent
{\bf Equivalence of \ref{item:alllines} and \ref{item:bitangents4} :} This is the statement of Theorem \ref{theorem:quartics-to-cubics} above.

\end{proof}

\subsection{\boldmath$S_7$-varieties and Hilbert's 13th Problem}

We now prove the equivalence of Hilbert's 13th problem with various other problems. Recall that $\Cc_n(\Pb^m)$ denotes the moduli space of ordered $n$-tuples of distinct points in $\P^m$ modulo the action of $\PGL_{m+1}$.

\begin{theorem}[{\bf \boldmath$\RD$ of \boldmath$S_7$ varieties}]\label{theorem:S7:varieties}
The following are equivalent:

\begin{enumerate}
\item \label{item:H13} Hilbert's 13th problem: $\RD(\tPc_7\to\Pc_7)=3$.
\item $\RD(V\to V/S_7=3)$ for any faithful linear representation $V$ of $S_7$.\label{item:S7:faithful}
\item $\RD(S_7)=3$.\label{item:RD:S7}
\item \label{item:C7}$\RD( \Cc_7(\Pb^n)\to \Cc_7(\Pb^n)/S_7)=3$ for $n\le 4$; in particular
\[\RD(\Mc_{0,7}\to\Mc_{0,7}/S_7)=3.\]
\item \label{item:Aronhold} $\RD=3$ for the problem of finding the $28$ bitangents on a smooth quartic $C$, given an Aronhold set on $C$:
    \begin{equation*}
        \RD(\Hc_{4,2}(28)\to\Hc_{4,2}(\Ac))=3.
    \end{equation*}
\end{enumerate}
In fact, the resolvent degrees of all of the above problems coincide.
\end{theorem}

\begin{proof}
\medskip
\noindent
{\bf Equivalence of \ref{item:H13}, \ref{item:S7:faithful}, \ref{item:RD:S7} and \ref{item:C7}:} This follows from Corollary \ref{cor:Hfaithandn1} \ref{corollary:Hilbert:for:faithful} together with Corollary \ref{cor:Cn}.

\medskip
\noindent
{\bf Equivalence of \ref{item:RD:S7} and \ref{item:Aronhold}}
The equivalence of \ref{item:RD:S7} and \ref{item:Aronhold} follows from Theorem \ref{theorem:quartics-to-cubics}.
\end{proof}

\subsection{\boldmath$S_8$-varieties and Hilbert's Octic Conjecture}
We now prove the equivalence of Hilbert's Octic Conjecture to several problems about plane quartics and genus 3 curves.

\begin{theorem}[{\bf \boldmath$\RD$ of \boldmath$S_8$ varieties}]\label{theorem:S8:varieties}
The following are equivalent:
\begin{enumerate}
\item Hilbert's Octic Conjecture: $\RD(\tPc_8\to\Pc_8)=4$.\label{item:octic}
\item $\RD(V\to V/S_8=4)$ for any faithful linear representation $V$ of $S_8$.\label{item:S8:faithful}
\item $\RD(S_8)=4$.\label{item:RD:S8}
\item \label{item:C8}  $\RD(\Cc_8(\Pb^n)\to \Cc_8(\Pb^n)/S_8)=4$, for $n\le 5$;  in particular
\[\RD(\Mc_{0,8}\to\Mc_{0,8}/S_8)=4.\]
\item \label{item:theta} $\RD=4$ for the problem of finding the $28$ bitangents on a smooth quartic $C$, given an even $\theta$-characteristic:
    \begin{equation*}
        \RD(\Hc_{4,2}(28)\to\Hc_{4,2}(\theta^{\rm ev}))=4,
    \end{equation*}
\item \label{item:S8:aron} $\RD=4$ for the problem of finding an Aronhold set on a smooth plane quartic $C$ given an even $\theta$-characteristic:
    \begin{equation*}
        RD(\Hc_{4,2}(\Ac)\to\Hc_{4,2}(\theta^{\rm ev}))=4,
    \end{equation*}
\item $\RD=4$ for the problem of finding the 28 bitangents on a quartic, given a Cayley octad:\label{item:Cay}
    \begin{equation*}
        \RD(\Hc_{4,2}(28)\to\Cay)=4.
    \end{equation*}
\end{enumerate}
In fact, the resolvent degrees of all of the above problems coincide.
\end{theorem}
\begin{proof}
    The equivalence of \eqref{item:octic}, \eqref{item:S8:faithful} and \eqref{item:RD:S8} follows from Corollary \ref{cor:Hfaithandn1} \ref{corollary:Hilbert:for:faithful}.

    For the equivalence of \eqref{item:RD:S8}, \eqref{item:C8}, and \eqref{item:theta}, observe that there exists a diagram of $W(E_7)^+$-equivariant maps
    \begin{equation}\label{H13}
        \Ab(\mathfrak{h})\dr\Pb(\mathfrak{h})\dr \Cc_7(\Pb^2)\dr\Hc_{4,2}(28)
    \end{equation}
    Indeed, the sequence
    \begin{equation*}
        \Ab(\mathfrak{h})\dr\Pb(\mathfrak{h})\dr \Cc_7(\Pb^2)\dr\Hc_{4,2}(28)
    \end{equation*}
    was constructed as \eqref{bitan} in the proof of Proposition \ref{prop:bitangentsversal}. Because $W(E_7)^+$ is simple, all the varieties in \eqref{H13} are faithful $W(E_7)^+$-vareties. By Proposition \ref{proposition:RD:linear} and Lemma \ref{lemma:dom}, we conclude that all of these varieties are versal $G$-varieties for any $G\subset W(E_7)^+$, in particular for $G=S_8$. The equivalence of \eqref{item:RD:S8}, \eqref{item:C7}, and \eqref{item:theta} now follows from Proposition \ref{prop:whyversal}. The equivalence of \eqref{item:theta} and \eqref{item:S8:aron} follows from Lemma \ref{lemma:galois} and the fact that
    \begin{align*}
        \Hc_{4,2}(28)&\to\Hc_{4,2}(\theta^{\rm ev})\intertext{is a Galois closure of the cover}
        \Hc_{4,2}(\Ac)&\to\Hc_{4,2}(\theta^{\rm ev}).
    \end{align*}
    Finally, the equivalence of \eqref{item:RD:S8} and \eqref{item:Cay} follows from the classical fact that there is a birational map
    \begin{equation*}
        \Hc_{4,2}(28)/S_8\simeq \Cay
    \end{equation*}
    from the $S_8$ quotient of the moduli of smooth plane quartics with an ordering of their 28 bitangents to the moduli of Cayley octads.
\end{proof}

\subsection{\boldmath$W(E_7)$ and bitangents to a planar quartic}
In this subsection we prove the equality of the resolvent degree of different $W(E_7)^+$-varieties.

\begin{theorem}[{\bf \boldmath$\RD$ of \boldmath$W(E_7)$ and bitangents to a planar quartic}]
    The following are equal:
    \begin{enumerate}
        \item $\RD(W(E_7))$.
        \item $\RD(W(E_7)^+)$
        \item $\RD(V\to V/G)$ for $G=W(E_7)^+, W(E_7)$ and $V$ any faithful representation of $G$.
        \item $\RD(\Cc_7(\Pb^2)\to \Cc_7(\Pb^2)/W(E_7)^+)$.
        \item $\RD(\Hc_{4,2}(28)\to\Hc_{4,2})$.
    \end{enumerate}
\end{theorem}

\begin{proof}
As noted above, there is an isomorphism
\[W(E_7)\cong W(E_7)^+\times\Z/2\Z;\]
Theorem \ref{thm:JH} implies that
\[\RD(W(E_7)=\max\{\RD(\Z/2\Z),\RD(W(E_7)^+)\}=\RD(W(E_7)^+).\]
In the proof of Theorem \ref{theorem:S8:varieties}, we constructed a diagram \eqref{H13} of varieties which are versal for every $G\subset W(E_7)^+$, in particular for $G=W(E_7)^+$. By Proposition \ref{prop:whyversal}, we conclude that
\begin{equation*}
    \RD(X\to X/W(E_7)^+)=\RD(W(E_7)^+)
\end{equation*}
for all $X$ in the diagram \eqref{H13}. The theorem now follows.
\end{proof}

\newpage
\section{Appendix}
\begin{center}
    {\bf Versal covers for subgroups of $W(E_6)$ and covers related to lines on cubics.}
\end{center}
\begin{equation*}
    \xymatrix{
        \Ab(\mathfrak{h}_6) \ar@{-->}[r] & \Pb(\mathfrak{h}_6)  \ar@{-->}[r] & \Cc_6'(\Pb^2) \ar[r]^\cong  \ar@{-->}[d]^{S_2^{\times 4}} & \Hc_{3,3}(27) \ar[d]^{S_2^{\times 4}} \ar[dr]^{S_6} \ar[drr]^{W(E_6)^+} \\
        && \Mc_{0,7}/S_2 \ar@{-->}[r]^\sim \ar[d]^{S_5} & \Hc_{3,3}(27)/S_2^{\times 4} \ar[d]^{S_5} & \Hc_{3,3}^{skew}(6)/S_6 \ar[d]^{S_2} & \Hc_{3,3}(\sqrt{\Delta}) \ar[ddl]^{S_2} & \tPc \ar[l]_{A_5} \\
        && \Mc_{0,7}/S_2\times S_5 \ar@{-->}[r]^\sim & \Hc_{3,3}(1) \ar[dr]^{27:1} & \Hc_{3,3}(6,6) \ar[d]^{36:1} \\
        &&&&\Hc_{3,3}
    }
\end{equation*}
The diagram above shows the relation between many covers of classical interest of the moduli space $\Hc_{3,3}$ of smooth cubic surfaces. The column involving $\Mc_{0,7}$ was constructed by Doran \cite{Dor}.

\begin{center}
    {\bf Versal covers for subgroups of $W(E_7)^+$ and covers related to bitangents on plane quartics.}
\end{center}
\begin{equation*}
\xymatrix{
        \Ab(\mathfrak{h}_7) \ar@{-->}[r] & \Pb(\mathfrak{h}_7)  \ar@{-->}[r] & \Cc_7(\Pb^2) \ar@{-->}[r]^{2:1} \ar[d]_{W(D_5)} & \Hc_{4,2}(28) \ar[d]_{W(D_5)} \ar[dr]_{S_7} \ar[ddr]_{S_8} \ar[drr]^{W(D_6)} \\
        && \Uc_{3,3}(1) \ar@{-->}[r]^{2:1} \ar[d]_{27:1} & \Hc_{4,2}(2) \ar[d]_{27:1} & \Hc_{4,2}(\Ac) \ar[d]^{8:1} & \Hc_{4,2}(\St) \ar[ddl]^{63:1} \\
        && \Uc_{3,3} \ar@{-->}[r]^{2:1} & \Hc_{4,2}(1) \ar[dr]_{28:1} & \Hc_{4,2}(\theta^{\rm ev}) \ar[d]_{36:1} \\
        &&&&\Hc_{4,2}
    }
\end{equation*}
The diagram above shows the relation between many covers of classical interest of the moduli space $\Hc_{4,2}$ of smooth plane quartics.

\bigskip{\noindent
Dept. of Mathematics, University of Chicago\\
E-mail: farb@math.uchicago.edu\\
\\
Dept. of Mathematics, University of California, Irvine\\
E-mail: wolfson@uci.edu

\begin{thebibliography}{DGM00}
\footnotesize

\bibitem[Ar]{Ar}
V.I. Arnol'd, On the representation of continuous functions of three variables by superpositions of continuous functions of two variables, {\em Mat. Sb.} (n.S.) 48 (90):1 (1959), 3--74.

\bibitem[AS]{AS}
V. Arnold and G. Shimura, Superpositions of algebraic functions, {\em Proc. Symposia in Pure Math.} vol. 28 (1976), AMS, Providence, 45--46.

\bibitem[Atl]{Atl}
{\em Atlas of finite group representations, v.3}, http://brauer.maths.qmul.ac.uk/Atlas/v3/,\\
 accessed Jan. 5, 2018.

\bibitem[Bra1]{Br1}
R. Brauer, A note on systems of homogeneous algebraic equations, {\em Bull.} AMS, vol. 51 (1945), 749--755.

\bibitem[Bra2]{Br}
R. Brauer, On the resolvent problem, {\em Ann. Mat. Pura Appl.} (4) 102 (1975), 45--55.

\bibitem[Bri]{Bri}
E. Bring, Meletemata qu{\ae}dam Mathematica circa Transformationem {\AE}quationum Algebraicarum (``Some Selected Mathematics on the Transformation of Algebraic Equations''), Lund, 1786.

\bibitem[BuRe1]{BuRe1}
J. Buhler and Z. Reichstein, On the essential dimension of a finite group, {\em Compositio Math.} vol. 106 (1997), no. 2, 159--179.

\bibitem[BuRe2]{BuRe2}
J. Buhler and Z. Reichstein, On Tschirnhaus transformations, {\em Topics in Number Theory}, (University Park, PA, 1997), 127--142, Math. Appl., 467, Kluwer Acad. Publ., Dordrecht, 1999.

\bibitem[Bur]{Bur}
H. Burkhardt, Untersuchungen aus dem Gebiet der hyperelliptischen Modulfunctionen, II, {\em Math. Ann.} vol. 38 (1890),  161--224.

\bibitem[CGR]{CGR}
V. Chernousov, P. Gille and Z. Reichstein, Resolving G-torsors by abelian base extensions, {\em J. Algebra} vol. 296 (2006), no. 2, 561--581.

\bibitem[CHM]{CHM}
A. Chen, Y.-H. He, and J. McKay, Erland Samuel Bring's ``Transformation of algebraic equations'', arXiv:1711.09253v1.

\bibitem[Di]{Di}
J. Dixmier, Histoire de 13e probl\`{e}me de Hilbert, {\em Cahiers du s\'{e}minaire d'histoire des math\'{e}matiques}, 
2e s\'{e}rie, tome 3 (1993), 85--94.


\bibitem[Dol]{Do}
I. Dolgachev, Classical algebraic geometry: A modern viewpoint, {\em Cambridge Univ. Press}, 2012.

\bibitem[DoOr]{DoOr}
I. Dolgachev and D. Ortland, Point sets in projective spaces and theta functions, {\em Asterisque} vol. 165 (1988).

\bibitem[Dor]{Dor}
B. Doran, Hurwitz spaces and moduli Spaces as ball quotients via pull-back, arXiv:math/0404363v1.

\bibitem[DoMc]{DoMc}
P. Doyle and C. McMullen, Solving the quintic by iteration, {\em Acta Math.} vol. 163 (1989), 151--180.

\bibitem[DuRe1]{DuRe1}
A. Duncan and Z. Reichstein, Versality of algebraic group actions and rational points on twisted varieties, with an appendix containing a letter from J.-P. Serre, {\em J. Algebraic Geom.} vol. 24 (2015), no. 3, 499--530.

\bibitem[DuRe2]{DuRe2}
A. Duncan and Z. Reichstein, Pseudo-reflection groups and essential dimension, {\em J. Lond. Math. Soc.} (2) vol. 90 (2014), no. 3, 879--902.

\bibitem[EH]{EH}
D. Eisenbud and J. Harris, $3264$ and all that -- a second course in algebraic geometry, {\em Cambridge Univ. Press}, 2016.

\bibitem[EJ]{EJ}
A-S Elsenhans and J. Jahnel, Moduli spaces and the inverse Galois problem for cubic surfaces, {\em Trans. Amer. Math. Soc.} 367 (2015), no. 11, 7837--7861.

\bibitem[Gr]{Green}
M. Green, On the analytic solution of the equation of fifth degree, {\em Compos. Math.} vol. 37, no 3 (1978), p. 233--241.

\bibitem[FKW]{FKW}
B. Farb, M. Kisin and J. Wolfson, Modular functions and resolvent problems, arXiv:1912.12536.

\bibitem[Ham]{Ham}
W. Hamilton, Inquiry into the validity of a method recently proposed by George B. Jerrard, esq., for transforming and resolving equations of elevated degrees, {\em Report of the Sixth Meeting of the British Association for the Advancement of Science} (1836), Bristol, 295--348. 

\bibitem[Har]{Har}
J. Harris, Galois groups of enumerative problems, {\em Duke Math. J.} 46 (1979), no. 4, 685--724.

\bibitem[He]{He}
C. Hermite, Sur l'invariant du dix-huiti\`eme ordre des formes du cinqui\`eme degr\'e., {\em J. Reine Angew. Math.} vol. 59 (1861), 304--305.

\bibitem[Hi1]{Hi1}
D. Hilbert, Mathematical Problems, from {\em Proceedings of the} 1900 ICM, English translation reprinted in {\em Bull.} AMS, Vol. 37, No. 4 (2000), 407--436.

\bibitem[Hi2]{Hi2}
D. Hilbert, \"{U}ber die Gleichung neunten Grades, {\em Math. Ann.} 97 (1927), no. 1, 243--250.

\bibitem[Hil]{Hil}
H. Hilton, Plane Algebraic Curves, Oxford Univ. Press, 1920.

\bibitem[HS]{HS}
A. Hefez and G. Sacchiero, The Galois group of the tangency problem for plane curves, {\em Math. Scand.}, Vol. 56 (1985), 171--190.

\bibitem[Hu]{Hu}
B. Hunt, The geometry of some special arithmetic quotients, {\em Springer Lect. Notes in Math.}, Vol. 1637, 1996.

\bibitem[Jou]{Jou}
P. Joubert, Sur l'equation du sixi\`eme degr\'e. {\em C. R. Acad. Sci. Paris} vol. 64 (1867), 1025--1029.

\bibitem[Kl1]{Kl71}
F. Klein, Ueber eine geometrische Repr\"asentation der Resolventen algebraischer Gleichungen, {\em Math. Ann.}, vol. 4 (1871), 346--358.

\bibitem[Kl2]{Kl2}
F. Klein, Vorlesungen \"uber das Ikosaeder und die Aufl\"osung der Gleichungen vom f\"unften
Grade, Teubner, Leipzig, 1884. English translation: Lectures on the icosahedron and solution of
equations of the fifth degree, translated by G. G. Morrice, 2nd and rev. edition, New York, Dover
Publications, 1956.

\bibitem[Kl3]{Kl1}
F. Klein, Sur la r\'esolution, par fonctions hyperelliptique, de l'\'equation du vingt-septi\`eme degr\'e, de laquelle d\'epend la d\'etermination des vingt-sept droites d'une surface cubique, {\em Jour. de math. pure et appl.} (4) vol. 4, 1888.


\bibitem[Kr]{Kr}
L. Kronecker, Ueber die Gleichungen f\"unften Grades, {\em Journal f\"ur die reine und angewandte Mathematik} vol. 59 (1861), 306--310.

\bibitem[KrKr]{KK}
M. Kracht and E. Kreyszig, E. W. von Tschirnhaus: His Role in Early Calculus and His
Work and Impact on Algebra, {\em Historia Mathematica} vol. 17 (1990), 16--35.

\bibitem[Le]{Le}
D. Lehavi, Any smooth plane quartic can be reconstructed from its bitangents, {\em Israel J. Math}, vol. 146 (2005), 371--379.

\bibitem[Ma]{Ma}
L. Manivel, Configurations of lines and models of lie algebras, {\em J. Algebra} vol. 304 (2006), no. 1, 457--486.

\bibitem[Me1]{Me1}
A. Merkurjev, On the norm residue symbol of degree 2,  {\em Dokl. Akad. Nauk SSSR}, vol.  261 (1981), no. 3, 542--547.
 
\bibitem[Me2]{Me}
A. Merkurjev, Essential dimension: a survey, {\em Transformation Groups} vol. 18 (2013), no. 2, 415--481.

\bibitem[MR]{MR}
A. Meyer and Z. Reichstein, Some consequences of the Karpenko-Merkurjev theorem, {\em Documenta Math.} Extra vol.: Andrei A. Suslin's Sixtieth Birthday (2010), 445--457.

\bibitem[Pi]{Pi}
H. Pinkham, R\'esolutions simultan\'ee de points doubles rationnels, {\em Lecture Notes in Math.} vol. 777 (1980), Springer-Verlag, 179--204.

\bibitem[PSV]{PSV}
D. Plauman, B. Sturmfels and C. Vinzant, Supplementary material to ``Quartic curves and their bitangents'', {\em Jour. of Symb. Comp.}, Vol. 46, No. 6, (2011), 712--733, available at \text{http://www4.ncsu.edu/~clvinzan/quartics.html}.

\bibitem[Re]{Re}
Z. Reichstein, {\em Proceedings of the Inter. Cong. of Mathematicians}, Vol. II, 162--188, Hindustan Book Agency, New Delhi, 2010.

\bibitem[Rei]{Rei}
M. Reid, The complete intersection of two or more quadrics, 1972 preprint,
Trinity College, Cambridge.

\bibitem[Se]{Se}
J. Sekiguchi, The configuration space of $6$ points in $\Pb^2$, the moduli space of cubic surfaces and the Weyl group of type $E_6$, Kyoto University Research Information Repository, 1993-09, http://hdl.handle.net/2433/83660.

\bibitem[Seg1]{Seg1}
B. Segre, The Algebraic Equations of Degrees 5, 9, 157, ..., and the Arithmetic Upon an Algebraic Variety, {\em Ann. of Math.} (2), vol. 46 (1945), 287-301.

\bibitem[Seg2]{Seg2}
B. Segre, {\em Arithmetical Questions on Algebraic Varieties}, University of London, Athlone Press, London, 1951.

\bibitem[Sy]{Sy}
J. Sylvester, On the so-called Tschirnhausen transformation, {\em J. Reine Angew. Math.} vol. 100 (1887), 465--486.

\bibitem[SH1]{SH1}
J. Sylvester and J. Hammond, On Hamilton's Numbers, {\em Phil. Trans. R. Soc. London A} vol. 178 (1887), 285--312.

\bibitem[SH2]{SH2}
J. Sylvester and J. Hammond, On Hamilton's Numbers II, {\em Phil. Trans. R. Soc. London A} vol. 179 (1888), 65--71.

\bibitem[SS]{SS}
M. Sch\"{u}tt and T. Shioda, Mordell-Weil Lattices, preprint, March 2017.

\bibitem[Ts]{Ts}
E.W. von Tschirnhaus, Methodus auferendi omnes terminos intermedios ex data aeqvatione (Method of eliminating
all intermediate terms from a given equation), {\em Acta Eruditorum} (1683), 204--207.

\bibitem[W]{W}
J. Wolfson, Tschirnhaus transformations after Hilbert, arXiv:2001.06515.


\end{thebibliography}
\end{document}